\documentclass[10pt]{amsart}
\usepackage{amsfonts}
\usepackage{mathrsfs}
\usepackage{amsmath}
\usepackage{amssymb}
\usepackage{latexsym}
\usepackage{lscape}
\usepackage[all,cmtip]{xy}
\usepackage{comment}
\usepackage{appendix}
\usepackage{geometry}
\usepackage[nice]{nicefrac}
\usepackage{dsfont}

\usepackage[pagebackref,hyperindex,linktocpage=true]{hyperref}
\hypersetup{
    colorlinks,
    linkcolor={red!50!black},
    citecolor={blue!50!black},
    urlcolor={blue!80!black}
}

\usepackage{color}
\newcommand{\cc}{\color{cyan}}
\newcommand{\cb}{\color{blue}}
\newcommand{\cm}{\color{magenta}}
\newcommand{\cg}{\color{green}}
\newcommand{\cp}{\color{pink}}

\usepackage[colorinlistoftodos]{todonotes}

\newtheorem{thm}{Theorem}[subsection]

\newtheorem{prop}[thm]{Proposition}
\newtheorem{lem}[thm]{Lemma}

\newtheorem{lem-def}[thm]{Lemma-Definition}
\newtheorem{cor}[thm]{Corollary}

\theoremstyle{definition}

\newtheorem{rem}[thm]{Remark}
\newtheorem{defn}[thm]{Definition}

\numberwithin{equation}{section}

\newcommand{\into}{\hookrightarrow}

\newcommand{\bbA}{\mathbb{A}}
\newcommand{\bbB}{\mathbb{B}}
\newcommand{\bbC}{\mathbb{C}}
\newcommand{\bbD}{\mathbb{D}}
\newcommand{\bbE}{\mathbb{E}}
\newcommand{\bbF}{\mathbb{F}}
\newcommand{\bbG}{\mathbb{G}}

\newcommand{\bbK}{\mathbb{K}}

\newcommand{\bbN}{\mathbb{N}}

\newcommand{\bbQ}{\mathbb{Q}}
\newcommand{\bbR}{\mathbb{R}}

\newcommand{\bbT}{\mathbb{T}}

\newcommand{\bbV}{\mathbb{V}}

\newcommand{\bbZ}{\mathbb{Z}}

\newcommand{\bfG}{\mathbf{G}}
\newcommand{\bfP}{\mathbf{P}}

\newcommand{\cA}{\mathcal{A}}

\newcommand{\cD}{\mathcal{D}}

\newcommand{\cF}{\mathcal{F}}
\newcommand{\cG}{\mathcal{G}}

\newcommand{\cL}{\mathcal{L}}

\newcommand{\cN}{\mathcal{N}}
\newcommand{\cO}{\mathcal{O}}

\newcommand{\cX}{\mathcal{X}}

\newcommand{\fM}{\mathfrak{M}}

\newcommand{\fp}{\mathfrak{p}}
\newcommand{\fq}{\mathfrak{q}}

\newcommand{\ga}{\gamma}
\newcommand{\la}{\lambda}
\newcommand{\La}{\Lambda}

\newcommand{\spec}{\mathrm{Spec}}
\newcommand{\Sh}{\mathrm{Sh}}
\newcommand{\tr}{\mathrm{tr}}
\newcommand{\vol}{\mathrm{vol}}

\title{On the cohomology of simple Shimura varieties with \\ non quasi-split local groups}

\author{Jingren Chi}
\email{jrenchi@amss.ac.cn}
\address{Morningside Center of Mathematics and Hua Loo-Keng Key Laboratory of Mathematics, Academy of Mathematics and Systems Science, Chinese
Academy of Sciences, Beijing 100190, China;}

\author{Thomas J. Haines}
\email{tjh@umd.edu}
\address{Department of Mathematics, University of Maryland, College Park, MD 20742-4015, USA.}

\begin{document}

\begin{abstract}
We study the Scholze test functions for bad reduction of simple Shimura varieties at a prime where the underlying local group is any inner form of a product of Weil restrictions of general linear groups. Using global methods, we prove that these test functions satisfy a vanishing property of their twisted orbital integrals, and we prove that the pseudostabilization base changes of such functions exist (even though the local group need not be quasi-split) and can be expressed in terms of explicit distributions in the stable Bernstein center. We then deduce applications to the stable trace formula and local Hasse-Weil zeta functions for these  Shimura varieties.
\end{abstract}

%\thanks{\iffalse {\it 2010 Mathematics Subject Classification} 14G35(???????).\fi }

\maketitle 
\tableofcontents

\section{Introduction}
In this paper, we study the bad reduction of the simple Shimura varieties attached to unitary similitude groups introduced in the work of Kottwitz \cite{Ko-simple}. Our goal is twofold. On the one hand, we extend the results of \cite{Sch-Shin} that require the local group to be a product of Weil restrictions of general linear groups, instead allowing the local group to be any inner form of such a group. We thus systematically deal with issues arising from the failure of the local group to be quasi-split. This extension includes the proof of a vanishing property of twisted orbital integrals of various test functions and a theory of pseudostabilization base change transfer for such functions (both of these are new features arising for non-quasi-split groups). Further, it includes a comparison of certain generalizations of the Scholze test functions defined in \cite{Sch-deformation} with certain standard test distributions in the stable Bernstein center detailed in \cite{Ha14}. On the other hand, we use these local results to describe the cohomology and the semisimple Hasse-Weil local factors of Kottwitz-type simple Shimura varieties, generalizing previous works, such as  \cite{Sch-Shin} and \cite{shen-1}, to more cases of bad reduction.

\subsection{Main results}
Let $\bbF$ be a CM field with totally real subfield $\bbF_0$. Let $(\bbD,*)$ be a division algebra with center $\bbF$ together with an involution $*$ of second kind, giving rise to the unitary similitude group $\bbG$ and simple Shimura varieties $\mathrm{Sh}_K$ as in \cite{Ko-simple}, where $K\subset\bbG(\bbA_f)$ is sufficiently small compact open subgroup. Fix a prime number $p$. Let $\bbE$ be the reflex field and $\fp$ a prime of $\bbE$ above the rational prime $p$. We assume that all places of $\bbF_0$ above $p$ are split in $\bbF$. In particular, $\bbG_{\bbQ_p}$ is a product of inner forms of Weil restrictions of general linear groups. Also we assume that $K=K^pK_p$ for compact open subgroups $K_p\subset\bbG(\bbQ_p)$ and $K^p\subset\bbG(\bbA_f^p)$.
\par 

Fix another prime number $\ell\ne p$ and let $\xi$ be an algebraic representation of $\bbG$ giving rise to an $\ell$-adic local system $\cF_{\xi,K}$ on $\mathrm{Sh}_K$. Consider the $\ell$-adic \'etale cohomology groups
\[H_\xi^i:=\varinjlim_K H_{\acute{e}t}^i(\Sh_K\otimes_{\bbE}\overline{\bbQ},\cF_{\xi,K})\]
and form their alternating sum $H_\xi^*:=\sum_i(-1)^iH_\xi^i$ in the Grothendieck group of $\bar{\bbQ}_\ell$-representations of  $\mathrm{Gal}(\overline{\bbQ}/\bbE)\times\bbG(\bbA_f)$. The main global theorem concerns the restriction of this representation to the local Weil group $W_{\bbE_\fp}$. We choose a smooth $\bbZ_p$-model $\cG$ of $\bbG_{\bbQ_p}$ such that $\cG(\bbZ_p)$ is a  parahoric subgroup of $\bbG(\bbQ_p)$, and we always assume $K_p \subseteq \cG(\bbZ_p)$.

\begin{thm}\label{thm:main-global-intro}
There is an identity in the Grothendieck group of $\bbG(\bbA_f^p)\times \cG(\bbZ_p)\times W_{\bbE_\fp}$ representations
\[H^*_\xi\cong\sum_{\pi_f}a(\pi_f)\pi_f\otimes(r_{-\mu}\circ\varphi_{\pi_p}|_{W_{\bbE_\fp}})|\cdot|^{-\dim\mathrm{Sh}/2}.\]
\end{thm}

We refer to Theorem \ref{thm:cohomology-isom} for the more precise statement. Here we only mention that $\varphi_{\pi_p}$ is the semisimple $L$-parameter of the irreducible smooth representation $\pi_p$ of $\bbG(\bbQ_p)$. Note that by our assumption $\bbG_{\bbQ_p}$ is an inner form of a product of Weil restrictions a general linear groups, hence the local Langlands correspondence is known for $\bbG_{\bbQ_p}$; see, e.g.,\,\cite{Coh18}.\par 
This result is proved by adapting Scholze's generalization of the Langlands-Kottwitz method to bad reduction cases as in \cite{Sch-deformation}. Roughly speaking, using an $\cO_{\bbE_\fp}$-model of the Shimura variety with $\cG(\bbZ_p)$-level
structure and the description of points in its special fiber, one gets the following expression from the Grothendieck-Lefschetz trace formula:
\begin{equation}\label{eq:tr-intro}
    \mathrm{Tr}(\tau\times hf^p| H_\xi^*)=\sum_{(A,u,\la)}
\mathrm{vol}(I(\bbQ)\backslash I(\bbA_f))~O_{\gamma}(f^p)~TO_{\delta\sigma}(\phi_{\tau,h})\,\mathrm{tr}\xi(\gamma_\ell).
\end{equation}

Here $\tau\in\mathrm{Frob}^jI_{\bbE_\fp}\subset W_{\bbE_\fp}$ for some $j\ge1$ (where $\mathrm{Frob}\in W_{\bbE_\fp}$ is a geometric Frobenius element and $I_{\bbE_\fp}\subset W_{\bbE_\fp}$ is the inertia subgroup), $h\in C_c^\infty(\cG(\bbZ_p))$ is a cut-off function at $p$ and $f^p\in C_c^\infty(\bbG(\bbA_f^p))$ is a function corresponding to a Hecke operator away from $p$. The sum runs over isomorphism classes of virtual polarized abelian varieties $(A,u,\la)$ which, roughly speaking, parametrize isogeny classes in the set of $\kappa_\fp$-points of the special fiber of the integral model. For each isogeny class one has an associated $\sigma$-conjugacy class $\delta\in\bbG(\bbQ_{p^r})$ where $\bbQ_{p^r}$ is the degree $r:=j[\kappa_\fp:\bbF_p]$ unramified extension of $\bbQ_p$ and $\sigma$ is the Frobenius automorphism on $\bbQ_{p^r}$, and a conjugacy class $\ga=(\ga_\ell)_{\ell\ne p}\in\bbG(\bbA_f^p)$. See \S\ref{sec:fixed-points-of-correspondence} for details. The main ingredient in each summand is the twisted orbital integral of certain test function $\phi_{\tau,h}\in C_c^\infty(\bbQ_{p^r})$ that is defined in a purely local way using certain deformation spaces of $p$-divisible groups with extra structures.\par 
Next one would like to relate the right hand side of  \eqref{eq:tr-intro} with the geometric side of the Arthur-Selberg trace formula for $\bbG$. The first step would be to associate a global (stable) conjugacy class $\ga_0\in\bbG(\bbQ)$ to each summand so that $(\ga_0;\ga,\delta)$ forms a Kottwitz triple. However, when the local group $\bbG_{\bbQ_p}$ is not quasi-split, there is an obstruction to the existence of $\ga_0$. Namely, such a $\ga_0$ exists if and only if the naive norm $N\delta:=\delta\sigma(\delta)\dotsm\sigma^{r-1}(\delta)\in\bbG(\bbQ_{p^r})$ is conjugate to an element in $\bbG(\bbQ_p)$; see Proposition \ref{prop:Kott-triple}. When $\bbG_{\bbQ_p}$ is quasi-split, as in \cite{Ko-simple}, \cite{Sch-deformation}, and \cite{Sch-Shin}, this is always true and there is no obstruction to finding $\ga_0$. Now we can state our first local result, the aforementioned vanishing result.
\begin{thm}\label{thm:vanishing-intro}
    With notations as above, if $N\delta\in\bbG(\bbQ_{p^r})$ is not $\bbG(\bbQ_{p^r})$-conjugate to an element of $\bbG(\bbQ_{p})$, then $TO_{\delta\sigma}(\phi_{\tau,h})=0$.
\end{thm}
The special case of this result where $\bbG(\bbQ_p)$ is the unit group of a division algebra and the relevant cocharacter is $(1,0,\dotsc,0)$ is proved in previous work of Shen \cite{shen-1}. Our approach to the general case is different from Shen's method. Vanishing statements like this were first conjectured by Rapoport, for test functions for certain Shimura varieties with parahoric level structure at $p$ (see \cite[Conj.\,5.7]{Rap90}). Special cases of his conjecture were proved by Rapoport in {\em loc.\,cit.}, and were later generalized by Waldspurger to some additional cases (private communication), also by passing to the spectral side. Rapoport made a related vanishing conjecture in the framework of the Langlands-Rapoport Conjecture, see \cite[Conj.\,10.2]{Rap05}. \par

As an immediate consequence of Theorem~\ref{thm:vanishing-intro}, only those isogeny classes to which we can associate a Kottwitz triple contribute to the sum \eqref{eq:tr-intro} and we get
\[\mathrm{Tr}(\tau\times hf^p| H_\xi^*)=\sum_{(\ga_0;\ga,\delta)}c(\ga_0;\ga,\delta)\,O_\ga(f^p)\,TO_{\delta\sigma}(\phi_{\tau,h})\,\mathrm{tr}\xi(\ga_0).\]
To proceed further, we need to know the spectral information of the test function $\phi_{\tau,h}$. Let us introduce more notations. The Shimura data determines a conjugacy class of cocharacters $\mu$ of $\bbG(\bbC)$ (whose field of definition is the reflex field $\bbE$) that we also view as characters of the maximal torus $\hat{\bbT}$ in the Langlands dual group $\hat{\bbG}$. Let $r_{-\mu}$ be the irreducible algebraic representation of $\hat{\bbG}\rtimes W_{\bbE_\fp}$ whose restriction to
$\hat{\bbG}$ has extreme weight $-\mu$, such that $W_{\bbE_\fp}$ acts trivially on the extreme weight spaces. For each $\tau\in W_{\bbE_\fp}$, let $z_{\tau,-\mu}$ be the element of the (stable) Bernstein center of $\bbG_{\bbQ_p}$ that acts on any irreducible smooth representation $\pi$ of $\bbG(\bbQ_p)$ by the scalar 

\[\mathrm{Tr}(\tau|(r_{-\mu}\circ\varphi_{\pi}|_{W_{\bbE_\fp}})|\cdot|_{\bbE_\fp}^{-\langle\rho,\mu\rangle})\]
where $\varphi_\pi:W_{\bbQ_p}\to\hat{\bbG}\rtimes\mathrm{Gal}(\bar{\bbQ}_p/\bbQ_p)$ is the semisimple Langlands parameter of $\pi$ and $\rho$ is half the sum of positive roots of $\bbG$. Here the notion of positive root is determined by a choice of Borel subgroup $\bbB$ containing $\bbT$, and in forming $\langle \rho, \mu \rangle$ we also choose a $\bbB$-dominant representative $\mu \in X_*(\bbT)$ in its Weyl-group orbit. Now we can state our second local result. 
\begin{thm}\label{thm:local-matching-intro}
    The Scholze test function $\phi_{\tau,h}\in C_c^\infty(\bbG(\bbQ_{p^r}))$ has matching \textup{(}twisted\textup{)} orbital integrals with $z_{\tau,-\mu}*h\in C_c^\infty(\bbG(\bbQ_p))$.
\end{thm}
This establishes a key relation between the Scholze test functions and the stable Bernstein center. Similar relations of test functions for bad reduction with the stable Bernstein center were predicted by the second author and Kottwitz, and later by Scholze and Shin. We refer the reader to \cite{Ha14} for some background on the test function conjecture and on the transfer conjectures for the stable Bernstein center; these ideas of the second author and Kottwitz were fleshed out during the period 1998-2011. Subsequently, Scholze and Shin formulated a more flexible version (in particular allowing for arbitrary cut-off functions $h$) in \cite{Sch-Shin}.  The version of Theorem \ref{thm:local-matching-intro} is adapted to the process of pseudostabilization which can be used in the setting of this article; the above references are in the framework of stabilization, which is more widely applicable.  When $\bbG_{\bbQ_p}$ is quasi-split, Theorem \ref{thm:local-matching-intro} was proved in \cite{Sch-Shin}. Our proof involves a global-local argument which has the same basis as the argument in \cite{Sch-Shin}, but which is more  complicated both in terms of local harmonic analysis and the global geometric argument, when the local group is allowed to be non-quasi-split. In particular, for our argument it is necessary to introduce two companion global groups $\mathbb G_\beta$ and $\mathbb G'$ whose roles will be detailed later in this introduction. Roughly speaking, we will study the Shimura varieties attached to the group $\bbG_\beta$, and express the point-counting of their special fibers in terms of automorphic representations attached to the other group $\bbG'$ (instead of the group $\bbG_\beta$ itself).  \par 
%{\cg Tom: I think it is less confusing to retain, throughout the introduction, the distinction between the original $\bbG$ and $\bbG_\beta$.}

In fact, for many cut-off functions $h$ we can also find an element in the stable Bernstein center of $\bbG(\bbQ_{p^r})$ whose convolution with 
certain cut-off functions on $\bbG(\bbQ_{p^r})$ have matching twisted orbital integrals with $\phi_{\tau,h}$, see Corollary \ref{cor:TO-equality}. As a consequence, we verify the special case of the test function conjecture of the second author and Kottwitz in our setting; in particular this expresses the semisimple trace of Frobenius in terms of certain functions in the Bernstein center:

\[\mathrm{Tr}^{\mathrm{ss}}(\mathrm{Frob}_\fp^j\!\times\!f^p| H^*(\Sh_{K_pK^p}\otimes_\bbE\bar{\bbQ},\mathcal{F}_{\xi, K_pK^p}))=\sum_{(\ga_0;\ga,\delta)}c(\ga_0;\ga,\delta)\,O_\ga(f^p)\,TO_{\delta\sigma}(z^{(r)}_{-\mu}*e_{K_{p^r}})\,\mathrm{tr}\xi(\gamma_0).\]

Here the level $K_p\subset\cG(\bbZ_p)$ at $p$ can be arbitrarily small, $K_{p^r}\subset\bbG(\bbQ_{p^r})$ is a compact open subgroup corresponding to $K_p$ in a suitable sense, and $e_{K_{p^r}}$ is the characteristic function of $K_{p^r}$ divided by its volume. Further, $f^p \in C^\infty_c(K^p\backslash\mathbb{G}(\mathbb A^p_f)/K^p)$ is any function. Finally the main ingredient is the element $z^{(r)}_{-\mu}$ in the stable Bernstein center of $\bbG(\bbQ_{p^r})$, which is defined similarly to $z_{\tau,-\mu}$, using the semisimple trace (which coincides with trace on inertia invariants in the current setting) instead of the usual trace. See Corollary \ref{cor:Lefschetz-number-central-function} for a complete statement and section \S\ref{sec:Lefschetz-number-central-function} for the details.\par

The proof of Corollary \ref{cor:TO-equality} involves more delicate harmonic analysis than what is needed for the proof of Theorem \ref{thm:local-matching-intro}: it requires a result on the pseudostabilization base-change of the stable Bernstein center, namely Theorem \ref{thm:center-base-change}. The rather involved proof of this result relies on the permanence of the vanishing property of twisted orbital integrals, with respect to convolution by elements in the stable Bernstein center (Proposition \ref{prop:vanishing-property}), and all of these ingredients depend on our twisted local Jacquet-Langlands correspondence (Theorem \ref{thm:twisted-JL}), which to our knowledge is a novel result. We refer to \S\ref{sec:prep-local} for the details.

Our debt to the ideas of Kottwitz, Rapoport, Harris-Taylor, Scholze, and Shin should be clear to the reader. Like the related articles \cite{shen-1, Sch-deformation, Sch-Shin}, our method builds on Rapoport's extension of the Langlands-Kottwitz method to bad reduction, which proposed the idea of counting points by weighting fixed points of Frobenius-Hecke correspondences by the (semisimple) trace of Frobenius on the stalk of the nearby cycles of appropriate local systems on the generic fiber (see \cite[\S10]{Rap05}); but following Scholze, we avoid use of integral models except at the bottom of the tower, where we use an integral model for parahoric level structure. Further, like the earlier articles, we also rely heavily on the global-local arguments and Galois representation constructions of Harris-Taylor \cite{HT01}, in the sense that we generalize their arguments to work in our specific context.

\subsection{Strategy of proof}
Although the global result Theorem~\ref{thm:main-global-intro} will be an immediate consequence of the local results Theorem~\ref{thm:vanishing-intro} and Theorem~\ref{thm:local-matching-intro}, our proof of the local results is by a global method. Starting with the original Shimura data $(\mathbb G, X)$, we construct a companion Shimura data $(\mathbb G_\beta, X_\beta)$ whose localization at $p$ is the same as that for $(\bbG, X)$, and we show that in order to prove our local theorems, it is enough to work with $(\bbG_\beta, X_\beta)$ (see Proposition \ref{prop:independence-of-EL-data}). More precisely, the group $\mathbb G = {\rm GU}(\mathbb D, *)$ and the CM field $\mathbb F$ are replaced by another pair, which we denote by $\bbG_\beta$ and $\bbF = \bbK \bbF_0$.  The latter $\bbF$ is a CM field containing a purely imaginary quadratic extension $\bbK$ in which $p$ splits, and $\bbG_\beta = {\rm GU}(\bbD_\beta, *_\beta)$ is a group like $\bbG$ whose signature at the archimedean places is determined by the original $\mu$, as in our section \ref{sec:unit_sim_grps} (especially Lemma \ref{lem:global-group}).  For this unitary group $\bbG_\beta$, we have the ambient group $\widetilde{\bbG}_\beta$, the linear group of units of $\bbD_\beta$ over its center $\bbF$.  We establish the required theory of base change for these unitary groups (Theorem \ref{thm:global-base-change}) and use this to obtain the Galois representation of ${\rm Gal}(\bar{\bbF}/\bbF)$ attached to a cuspidal automorphic representation $\pi$ of $\bbG_\beta(\bbA)$ (Theorem \ref{thm:automorphic-to-Galois}), which by construction satisfies global-local compatibility.  For each relevant $\pi$ we induce the corresponding Galois representation using Shapiro's Lemma to get something almost like a global Langlands parameter, namely an admissible homomorphism $\varphi_{\ell}(\pi)\,:\,{\rm Gal}(\bar{\bbQ}/\bbK) \rightarrow \,^L\bbG_\beta(\bar{\bbQ}_\ell)$ (see the proof of Corollary \ref{cor:reciprocity}).  We restrict this further to ${\rm Gal}(\bar{\bbQ}/\tilde{\bbE})$ where $\tilde{\bbE} = \bbE \bbK$ (composite of the reflex field with $\bbK$), in order to compare it with the Galois representation $\sigma_{\beta, \xi}(\pi_f)$ appearing tautologically in the cohomology of the Shimura variety attached to $\bbG_\beta$.  
Having done all that, we are in a position to use the results of Kottwitz at the good places, and the Chebotarev density theorem, in order to prove that Theorem \ref{thm:main-global-intro} holds for ${\rm Sh}(\bbG_\beta, X_\beta)$, even without using the Langlands-Kottwitz-Scholze method for this special Shimura variety.  The details are given in the proof of Corollary \ref{cor:reciprocity}.\par

For the next step, we run the Langlands-Kottwitz-Scholze method for ${\rm Sh}(\mathbb G_\beta, X_\beta)$. This connects the purely local test function $\phi_{\tau, h}$ to global cohomology groups where the Galois action is now sufficiently pinned-down, and it would permit us to prove the desired properties of $\phi_{\tau, h}$, provided that the Langlands-Kottwitz-Scholze method could be pushed to the point of stabilization. As mentioned above, there are obstructions to converting the resulting Lefschetz trace formula of form \eqref{eq:tr-intro} into an Arthur-Selberg trace formula without knowing the validity of the vanishing property of the test function $\phi_{\tau,h}$ and without being able to construct, even when the vanishing property is known, the global elements $\gamma_0$ in the desired Kottwitz triples $(\gamma_0; \gamma, \delta)$. To surmount these obstructions, our main innovation is to find a replacement  $(\gamma'_0; \gamma, \delta)$ of the usual Kottwitz triples using an auxiliary group $\bbG'$, also constructed from a division algebra with involution of the second kind, which is quasi-split at $p$ and isomorphic to $\bbG_\beta$ away from $p$ and $\infty$. With this modification we can relate the right hand side of \eqref{eq:tr-intro} to a simple trace formula for $\bbG'$. Finally, by comparing simple trace formulas for $\bbG_\beta$ and $\bbG'$, we get the local results Theorems~\ref{thm:vanishing-intro} and \ref{thm:local-matching-intro} and hence we deduce the global result Theorem~\ref{thm:main-global-intro} in full generality.

\subsection{Organization of the article}
In \S\ref{sec:prep-local}, we establish several results in $p$-adic harmonic analysis. After reviewing various notions of transfer of local test functions, we prove a twisted version of the local Jacquet-Langlands correspondence (Theorem \ref{thm:twisted-JL}) using a global method involving twisted trace formulae.  As a consequence we deduce the technical results mentioned above concerning base change of the stable Bernstein center and the permanence of the vanishing property.\par 
In \S\ref{sec:p-div}, we study the deformation spaces of $p$-divisible groups with EL structures and use it to define the local test functions $\phi_{\tau,h}$, generalizing the constructions in \cite{Sch-deformation}. Then we state our local results concerning its vanishing property (Theorem \ref{thm:vanishing-property}) and its relation with certain elements in the stable Bernstein center of the local group (Theorem \ref{thm:main-local}).\par
In \S\ref{sec:shimura}, we first review the construction of Kottwitz type simple Shimura varieties and their integral models with parahoric level structures. Then we state the 
main global result Theorem~\ref{thm:cohomology-isom} giving the description of the cohomology, review the Langlands-Kottwitz-Scholze method, and explain the extra ingredient needed to adapt it to our situation, that is,  the vanishing property Theorem~\ref{thm:vanishing-property} of the local test function $\phi_{\tau,h}$. After that we deduce Theorem~\ref{thm:cohomology-isom} from the main local result Theorem~\ref{thm:main-local} concerning the spectral information of the test function $\phi_{\tau,h}$.\par 
In \S\ref{sec:proof-local-thm}, we prove the local results by embedding the local situation into a global simple Shimura variety, for which we can use several results on automorphic representations to establish the description of its cohomology directly, without using the Langlands-Kottwitz-Scholze method. From there we deduce the local results concerning the test function $\phi_{\tau,h}$ by modifying the Langlands-Kottwitz-Scholze method using an auxiliary group $\bbG'$. 
\subsection{Notations}
Since we work in both local and global setting and will pass back and forth between them, to minimize confusion we distinguish different settings by different fonts. We use Roman letters $F,D,B,G$ etc. to denote objects over the local field $\bbQ_p$. But note that in \S\ref{sec:prep-local} we use the boldface letter $\bfG$ to denote an algebraic group over $\bbQ_p$, whereas we use the Roman letter $G=\bfG(\bbQ_p)$ to denote its set of $\bbQ_p$-points. We use the calligraphic font $\cG$ to denote a group scheme over $\bbZ_p$, while blackboard bold letters $\bbF,\bbD,\bbB,\bbG$ etc. are used to denote objects defined over the global field $\bbQ$. 

\subsection*{Acknowledgements} 
We heartily thank Michael Rapoport, Xu Shen and Wei Zhang for their comments on the manuscript. The research of J.C.~is partially supported by National Natural Science Foundation of China (Grant No.12288201, No.12231001), National Key R\&D Program of China (No.2023YFA1009701), CAS Project for Young Scientists in Basic Research (Grant No.YSBR-033). Research of T.H.~is partially supported by NSF DMS-2200873. Disclaimer: Any opinions, findings, and conclusions or recommendations expressed in this material are those of the author and do not necessarily reflect the views of the National Science Foundation.

\section{Preparations in local harmonic analysis}\label{sec:prep-local}
In this section we let $\bfG$ be a reductive algebraic group over $\bbQ_p$ whose quasi-split inner form $\bfG^*$ is isomorphic to a product of Weil restrictions of general linear groups. Let $r\ge1$ be an integer. Let $\bfG_r:=\mathrm{Res}_{\bbQ_{p^r}/\bbQ_p}\bfG_{\bbQ_{p^r}}$, and let $\bfG_r^*:=\mathrm{Res}_{\bbQ_{p^r}/\bbQ_p}\bfG_{\bbQ_{p^r}}^*$ be the quasi-split inner form of $\bfG_r$. Fix a generator $\sigma$ of the cyclic group $\mathrm{Gal}(\bbQ_{p^r}/\bbQ_p)$ which induces automorphisms of the $\bbQ_p$-algebraic groups $\bfG_r$ and $\bfG_r^*$ that we also denote by $\sigma$. 
We denote $G=\bfG(\bbQ_p)$, $G_r=\bfG_r(\bbQ_p)=\bfG(\bbQ_{p^r})$, and similarly $G^*=\bfG^*(\bbQ_p)$, $G_r^*=\bfG_r^*(\bbQ_p)=\bfG^*(\bbQ_{p^r})$.

\subsection{Transfer of twisted orbital integrals}
For each element $\delta\in G_r$, define its \emph{naive norm} to be the element
\[N_r\delta:=\delta\sigma(\delta)\dotsm\sigma^{r-1}(\delta)\in G_r.\]
Then there exists an element $\cN_r\delta\in G^*$ that is stably conjugate to the naive norm $N_r\delta$. The stable conjugacy class of $\cN_r\delta$ only depends on the stable $\sigma$-conjugacy class of $\delta$ and we call it \emph{the norm of $\delta$ in $G^*$}; see \cite[$\S5$]{Ko-conj}. Similarly, for $\delta\in G_r^*$ we have the naive norm $N_r^*\delta\in G_r^*$ and the norm $\cN_r^*\delta\in G^*$, the latter of which is, a priori, a stable conjugacy class in $G^*$. We say $\delta\in G_r$ (resp. $\delta^*\in G_r^*$) is \emph{$\sigma$-semisimple} if $N_r\delta$ (resp. $N_r\delta^*$) is semisimple.

For the groups and elements considered here, stable ($\sigma$-) conjugacy just coincides with ordinary ($\sigma$-) conjugacy.

\begin{lem}\label{lem:transfer-twisted-conj-class}
\begin{enumerate}
    \item The map $\delta\mapsto\cN_r\delta$ defines an injection from the set of $\sigma$-conjugacy classes of $\sigma$-semisimple elements of $G_r$ to the set of conjugacy classes of semisimple elements of $G^*$. A similar statement holds for $G_r^*$.
    \item For each $\delta\in G_r$ such that $N_r\delta$ is $G_r$-conjugate to a regular semi-simple element in $G$, there exists $\delta^*\in G_r^*$ such that $\cN_r\delta=\cN_r^*\delta^*$.
\end{enumerate}
\end{lem}

\begin{proof}
The first part follows from \cite[Lemma 1.1]{AC} (which is stated for $G_r^*$ but the proof also works for $G_r$). In more detail, by \cite[Cor.\,5.3]{Ko-conj}, the map $\delta \mapsto N_r\delta$ induces an injection from $\sigma$-$\bfG_r(\overline{\bbQ}_p)$ conjugacy classes in $\bfG_r(\overline{\bbQ}_p)$ represented by 
$\bfG_r(\bbQ_p) = \bfG(\bbQ_{p^r})$ to conjugacy classes in $\bfG^*(\overline{\bbQ}_p)$ represented by $\bfG^*(\bbQ_p)$. It is enough to prove that if $x, y \in \bfG(\bbQ_{p^r})$ are $\sigma$-$\bfG_r(\overline{\bbQ}_p)$-conjugate and are $\sigma$-semisimple, then they are $\sigma$-$\bfG(\bbQ_{p^r})$-conjugate (see the end of \cite[$\S5$]{Ko-conj}). Viewing the $\sigma$-centralizer $\bfG_{x\sigma}$ as a $\bbQ_p$-subgroup of $\bfG_r$, it is enough to show that $${\rm ker}(H^1(\bbQ_p,\bfG_{x\sigma}) \rightarrow H^1(\bbQ_p, \bfG_r))$$ is trivial. In fact we will show that $H^1(\bbQ_p,\bfG_{x\sigma})$ is trivial. Recall that $\mathbf H \mapsto H^1(\bbQ_p,\mathbf H)$ is invariant under inner twists of a connected reductive $\bbQ_p$-group $\mathbf H$; see \cite[Thm.\,1.2]{Ko-ellsing}.  By \cite[Lem.\,5.8]{Ko-conj}, $\bfG_{x\sigma}$ is a $\bbQ_p$-inner form of the centralizer $\bfG^*_{\cN_rx}$, and the latter is connected and reductive (using that $\cN_r x$ is semisimple and $\bfG^*_{\rm der} = \bfG^*_{\rm sc}$).  Hence is it enough to prove that $H^1(\bbQ_p, \bfG^*_{\cN_rx})$is trivial. But as $\cN_rx$ is semisimple, the group $\bfG^*_{\cN_rx}$ is the group of units of a semisimple $\bbQ_p$-algebra, and so by a version of Hilbert's Theorem 90, $H^1(\bbQ_p, \bfG^*_{\cN_rx})$ is trivial. This proves the assertion for $G_r$, and the proof for $G^*_r$ is similar.\par
For the second part, we argue as in \cite[Lemma 1.3]{AC}. By construction $G$ (resp. $G^*$) is the unit group of a semisimple $\bbQ_p$-algebra $B$ (resp. $B^*$). Note that $B^*$ is a product of matrix algebras. 
Let $B_r:=B\otimes_{\bbQ_p}\bbQ_{p^r}$ and $B_r^*:=B^*\otimes_{\bbQ_p}\bbQ_{p^r}$. Let $u=N_r\delta=\delta\sigma(\delta)\dotsm\sigma^{r-1}(\delta)$. 
By assumption, there exists $h\in G_r$ such that $huh^{-1}\in G$. Note that $huh^{-1}=N_r(h\delta\sigma(h)^{-1})$ so replacing $\delta$ by $h\delta\sigma(h)^{-1}$ we may assume that $u\in G$. Let $L$ be the centralizer of $u$ in $B$. Then $L_r:=L\otimes_{\bbQ_p}\bbQ_{p^r}$ is the centralizer of $u$ in $B_r$. 
Since $\delta^{-1} u\delta=\sigma(u)=u$, we have $\delta\in L_r$. Choose an embedding of $\bbQ_p$-algebras $L\into B^*$ and let $u^*$ be the image of $u$. Let $\delta^*$ be the image of $\delta$ under the resulting $\bbQ_{p^r}$-algebra embedding $L_r\into B_r^*$. 
Then $u^*\in G^*$ is stably conjugate to $u$ and since $N_{L_r/L}(\delta)=N_r(\delta)=u$ we get $N_r^*(\delta^*)=N_{L_r/L}(\delta^*)=u^*$.
\end{proof}

\begin{defn}\label{st_bc_defn}
For any $\phi\in C_c^\infty(G_r)$, we say that a function $\phi^*\in C_c^\infty(G^*)$ is a \emph{stable base change transfer} of $\phi\in C_c^\infty(G_r)$ if for any $\delta\in G_r$ with semisimple norm $\cN_r\delta =: \gamma^* \in G^*$ we have $e(G_{\delta \sigma})\,TO_{\delta\sigma}(\phi)= e(G^*_{\gamma^*})\,O_{\gamma^*}(\phi^*)$, whereas if a semisimple element $\ga^*\in G^*$ is not conjugate to a norm from $G_r$, then $O_{\ga^*}(\phi^*)=0$. Here $e(G_{\delta \sigma})$ is the Kottwitz sign of the $\sigma$-centralizer of $\delta$, viewed as a group over $\bbQ_p$. (See $\S\ref{sec:kottwitz-sign}$ for some reminders about the Kottwitz sign.)
\end{defn}

\begin{rem}
In making this definition, we choose Haar measures on $G_r$, $G_{\delta \sigma}$, $G^*$, $G^*_{\gamma^*}$, such that the measures on $G_{\delta\sigma}$ and its inner form $G^*_{\gamma^*}$ are compatible. The notion of matching depends on these choices.  
\end{rem}

\begin{rem}
    Since our group $G$ is an inner form of a product of Weil restrictions of $\mathrm{GL}_n$, stable conjugacy coincides with conjugacy due to vanishing of Galois cohomology as in the proof of Lemma \ref{lem:transfer-twisted-conj-class}. This justifies that our notion of transfer is compatible with the general definition as in \cite[\S3.2]{La99}, for example. 
\end{rem}

% partial uncompleted results about: "matching on regular semisimple implies matching on semisimple" has been moved to the separate file "matching.tex". It is not needed in the sequel, but can be added here once the proofs were complete. 
%\input{matching.tex}

The existence of stable base change in general is established in \cite[\S3.3.1]{La99}. It is a special instance of twisted endoscopic transfer.

%{\cc Jingren: The following definition seems not necessary. May be deleted}
%\begin{defn}
%We say that a pair of functions $\phi\in C_c^\infty(G_r)$ and $\phi^\dagger\in C_c^\infty(G_r^*)$ \emph{have matching twisted orbital integrals} if $e(G_{\delta \sigma})\, TO_{\delta\sigma}(\phi)= e(G^*_{\delta^*\sigma^*})\,TO_{\delta^*\sigma^*}(\phi^\dagger)$
%for any {\cb $\sigma^*$-semisimple $\delta^*\in G^*_r$ and $\sigma$-semisimple $\delta \in G_r$ with $\delta^* \leftrightarrow \delta$} {\cm $\cN_r\delta=\cN^*_r\delta^*$ regular semisimple}, and $TO_{\delta^*\sigma^*}(\phi^\dagger)=0$ when $\delta^*$ does not come from $G_r$. {\cg Tom: I added this more general definition; will remove if we don't need it.}
%\end{defn}

The following definition plays a key role in this article.

\begin{defn}\label{defn:vanishing-property}
We say that a function $\phi\in C_c^\infty(G_r)$ has the \emph{vanishing property} if $TO_{\delta\sigma}(\phi)=0$ for any $\sigma$-semisimple element $\delta\in G_r$ such that $N_r\delta$ is not $G_r$-conjugate to an element in $G$.  
\end{defn}

%{\cc Jingren: The following result is not used elsewhere, may be deleted if necessary.}
%\begin{prop}\label{prop:twisted-transfer}
%If $\phi\in C_c^\infty(G_r)$ has the vanishing property, then there exists $\phi^\dagger\in C_c^\infty(G_r^*)$ having matching twisted orbital integrals with $\phi$. 
%\end{prop}
%\begin{proof}
%Let $\phi^*\in C_c^\infty(G^*)$ be the stable base change transfer of $\phi$. Let $\ga^*\in G^*$ be a conjugacy class that is not a norm from $G_r^*$. If $\ga^*$ is not a norm from $G_r$, then $O_{\ga^*}(\phi^*)=0$ by definition of $\phi^*$. On the other hand, if $\ga^*=\cN_r\delta$ for some $\delta\in G_r$ but $\ga^*$ is not a norm from $G_r^*$, then by Lemma~\ref{lem:transfer-twisted-conj-class} the naive norm $N_r\delta$ is not $G_r$-conjugate to any element of $G$ and hence $O_{\ga^*}(\phi^*)=O_{\delta\sigma}(\phi)=0$ by the vanishing property of $\phi$.\par 
%Therefore by \cite[Proposition 3.1(ii)]{AC} there exists $\phi^\dagger\in C_c^\infty(G_r^*)$ with matching (twisted) orbital integrals with $\phi^*$. Then $\phi$ and $\phi^\dagger$ have matching twisted orbital integrals.
%\end{proof}

\subsection{Reduction to matching at regular elements}
It is well-known principle that in order to check the matching of (twisted) orbital integrals on all semisimple elements, it should suffice to check the matching on \emph{regular} semisimple elements. We include a proof of this fact in our setting, following the arguments of \cite[Prop.\,7.2]{Clo90}.
\begin{prop} \label{red_to_reg_matching_prop}
Let $\phi\in C_c^\infty(G_r)$ and $\phi^*\in C_c^\infty(G^*)$. Suppose that for any \emph{regular} semisimple element $\ga^*\in G^*$ we have $O_{\ga^*}(\phi^*)=0$ if $\ga^*$ is not a norm from $G_r$, while if $\ga^*=\cN_r\delta$ we have
    \[O_{\ga^*}(\phi^*)=TO_{\delta\sigma}(\phi).\]
    Then $\phi^*$ is a stable base change transfer of $\phi$ at all semisimple elements $\ga^*$. \textup{(}Note that the usual sign $e(\bfG_{\delta \sigma})$ which would appear is trivial under our assumption that $\ga^*$ is regular semisimple, since $\bfG_{\delta \sigma}$ is a torus.\textup{)}
\end{prop}
\begin{proof} 
We start by recalling some facts about Shalika germs. Suppose $\ga^*=\cN_r\delta$ is a semisimple element in $G^*$, not necessarily regular. By \cite[Lem.\,5.8]{Ko-conj}, once we fix a $\bbQ_p$-inner twisting $\psi: \bfG \rightarrow \bfG^*$ there is an $\bbQ_p$-inner twisting $\bfG_{\delta \sigma} \rightarrow \bfG^*_{\ga^*}$, canonical up to inner automorphisms by $\bfG^*_{\ga^*}(\bar{\bbQ}_p)$. For any maximal $\bbQ_p$-torus ${\mathbf T} \subset \bfG_{\delta\sigma}$ (${\mathbf T}$ need not be $\bbQ_p$-elliptic), we may assume our chosen inner twisting $\psi$ gives a $\bbQ_p$-embedding $\psi: {\mathbf T} \overset{\sim}{\rightarrow} {\mathbf T}^* \hookrightarrow \bfG^*_{\ga^*} \subset \bfG^*$, by \cite[Cor.\,2.2]{Ko-conj}. We can arrange (by conjugating $\psi$ by $\bfG^*(\bar{\bbQ}_p)$) that $\psi(N_r\delta) = \ga^*$.  Then we can further assume (conjugating $\psi$ if necessary by an element in $\bfG^*_{\ga^*}(\bar{\bbQ}_p)$ and using \cite[Cor.\,2.2]{Ko-conj}) that $\psi$ gives a $\bbQ_p$-isomorphism $\psi: {\mathbf T} \overset{\sim}{\rightarrow} {\mathbf T}^* \subset \bfG^*_{\ga^*}$. Then we have $N_r\delta \in T$ and $\psi(N_r\delta) = \ga^*$.\par
Fix a Haar measure $dt$ on $T := {\mathbf T}(\bbQ
_p)$. Let $\{u_j\}$ be a set of representatives of unipotent conjugacy classes in $G_{\delta\sigma}$. It follows from \cite[Proposition 8.1.1]{Ro90} that there is a sufficiently small neighborhood $0\in\Omega\subset\mathrm{Lie}(T)$ such that 
\begin{itemize}
    \item the exponential map defines an isomorphism from $\Omega$ to an open neighborhood of $1$ in $T$;
    \item Let $\Omega'\subset\Omega$ be the open subset consisting of elements $X\in\Omega$ such that $\exp(X)\in T$ is regular. Then $a\Omega'\subset\Omega'$ for any $0\ne a\in\bbZ_p$;
    \item The germ expansion holds on $\Omega'$, i.e.\,for all $0\ne t\in\bbZ_p$ and any $X\in\Omega'$, denote $x_t:=\exp(tX)$, then we have
    \[TO_{x_t\delta\sigma}(\phi)=\sum_j\Gamma_{u_j}^G(x_t\delta)\,TO_{u_j\delta\sigma}(\phi),\quad\forall \phi\in C_c^\infty(G_r)\]
    where the Shalika germ $\Gamma_{u_j}$ is a function on $\exp(\Omega')\delta$ that depends on the Haar measure on $T=G_{x_t\delta\sigma}=(G_{\delta\sigma})_{x_t}$ and the Haar measure on $G_{u_j\delta\sigma}$ used to define the orbital integrals. 
\end{itemize}
By \cite[Proposition 8.1.2(b)]{Ro90}, the function $\Gamma_{u_j}^G$ is homogeneous of degree $\frac{1}{2}d_{u_j}:=\frac{1}{2}\dim(G_{\delta\sigma}/G_{u_j\delta\sigma})$ in the sense that for all $X\in\Omega'$ and $0\ne t\in\bbZ_p$ we have
\[\Gamma_{u_j}^G(\exp(tX)\delta)=|t|^{-\frac{1}{2}d_{u_j}}\Gamma_{u_j}^G((\exp X)\delta).\]

There is a similar Shalika germ expansion 
\[O_{\psi(\exp(rtX))\ga^*}(\phi^*)=\sum_{j'}\Gamma_{u^*_{j'}}^{G^*}(\psi(\exp(rtX))\ga^*)\,O_{u^*_{j'}\ga^*}(\phi^*),\quad\forall  0 \neq t \in \mathbb Z_p.\]

Here $\{u^*_{j'}\}_{j'}$ ranges over the set of unipotent conjugacy classes in $G^*_{\ga^*}$, and the other notation is as before. Like the above, the Shalika germs here enjoy a similar homogeneity property.

Now we prove the vanishing statement. From the definition we verify easily that for any $X\in\Omega'$ and $0\ne t\in\bbZ_p$ we have $N_r(\exp(tX)\delta)=\exp(rtX)N_r\delta$ and hence under our embedding $T\into G^*_{\ga^*}$ we have $\cN_r(\exp(tX)\delta)=\psi(\exp(rtX))\ga^*$. Therefore if the semisimple element $\ga^*\in G^*$ is not a norm from $G_r$, then $\psi(\exp(rtX))\ga^*$ is also not a norm from $G_r$. Then by the definition of stable base change transfer we have $O_{\psi(\exp(rtX))\ga^*}(\phi^*)=0$. Looking at the degree $0$ term in the germ expansion we get that $O_{\ga^*}(\phi^*)=0$.\par
Next we again assume that $\ga^* = \cN_r\delta$ is a norm.  We consider $\psi, {\mathbf T}$, and ${\mathbf T}^*$ as above. We assume moreover that ${\mathbf T}$ is a $\bbQ_p$-elliptic torus in $\bfG_{\delta\sigma}$. Then ${\mathbf T}^*$ is also a $\bbQ_p$-elliptic torus in the inner form $\bfG_{\ga^*}^*$ of $\bfG_{\delta \sigma}$.\par
By Harish-Chandra descent of (twisted) orbital integrals, there exists $\phi_1\in C_c^\infty(G_{\delta\sigma})$ and $\phi_1^*\in C_c^\infty(G_{\ga^*}^*)$ such that for all $X\in\Omega'$ and $0\ne t\in\bbZ_p$ sufficiently close to $0$ we have
\begin{equation}\label{eq:descent-orb-int}
    TO^{G_r}_{\exp(tX)\delta\sigma}(\phi)=O_{\exp(tX)}^{G_{\delta\sigma}}(\phi_1),\quad O_{\psi(\exp(rtX))\ga^*}(\phi^*)=O_{\psi(\exp(rtX))}^{G^*_{\ga^*}}(\phi_1^*).
\end{equation}
See, for example, the proof of \cite[p.114, proof of Proposition 8.1.2]{Ro90}, or \cite[p.23-24]{AC}. \par 
By the descent property of Shalika germs (see \cite[Proposition 8.1.2(a)]{Ro90}) we have 
\begin{equation}\label{eq:germ}
    \Gamma_{1}^G((\exp(t X))\delta)=\Gamma_{1}^{G_{\delta\sigma}}(\exp (tX)),\quad\Gamma_1^{G^*}(\psi(\exp(rtX))\ga^*)=\Gamma_1^{G^*_{\ga^*}}(\psi(\exp(rtX))).
\end{equation}
Comparing the degree $0$ terms of the germ expansions in \eqref{eq:descent-orb-int} we get
\begin{equation}\label{eq:singular-orb-int}
    TO_{\delta\sigma}(\phi)=\phi_1(1),\quad O_{\ga^*}(\phi^*)=\phi_1^*(1).
\end{equation}
On the other hand, by assumption we have $TO^{G_r}_{\exp(tX)\delta\sigma}(\phi)=O_{\psi(\exp(rtX))\ga^*}(\phi^*)$ and hence by \eqref{eq:descent-orb-int} we get
\[O_{\exp(tX)}^{G_{\delta\sigma}}(\phi_1)=O_{\psi(\exp(rtX))}^{G^*_{\ga^*}}(\phi_1^*).\]
Then looking at the degree $0$ terms of the germ expansion and using \eqref{eq:germ}, \eqref{eq:singular-orb-int} we get
\[\Gamma_{1}^{G_{\delta\sigma}}(\exp X)\,TO_{\delta\sigma}(\phi)=\Gamma_1^{G^*_{\ga^*}}(\psi(\exp(rtX)))\,O_{\ga^*}(\phi^*)\]
We note that $\bfG^*_{\ga^*}$ is the quasi-split inner form of $\bfG_{\delta\sigma}$. Now we claim: since ${\mathbf T}$ is elliptic in $\bfG_{\delta\sigma}$, $\Gamma_{1}^{G_{\delta\sigma}}$ and $\Gamma_1^{G^*_{\ga^*}}$ are constant functions of $X$ and they differ by the factor $e(\bfG_{\delta\sigma})$. Then we would get that $O_{\ga^*}(\phi^*)=e(\bfG_{\delta\sigma})TO_{\delta\sigma}(\phi)$.\par
It remains to prove the claim above. Since ${\mathbf T}$ is $\bbQ_p$-elliptic in $\bfG_{\delta \sigma}$, the main theorem of \cite{Ro81} shows that
\[\Gamma_{1}^{G_{\delta\sigma}}(\exp X)=(-1)^{q(G_{\delta\sigma})}d(G_{\delta\sigma})^{-1},\quad\forall X\in\Omega'\] 
where $q(G_{\delta\sigma})$ is the $\bbQ_p$-rank of the derived group of $\bfG_{\delta\sigma}$ and $d(G_{\delta\sigma})$ is the formal degree of the Steinberg representation of $G_{\delta\sigma}$ (with respect to a chosen Haar measure on $G_{\delta\sigma}$). There is a similar result for $\Gamma_1^{G^*_{\ga^*}}$. In \cite{Ko-Tam}, the measures on the groups are determined from Euler-Poincar\'e measures on the adjoint groups, in which case $d(G^*_{\ga^*})=(-1)^{q(G^*_{\ga^*})}$ by \cite{Bo76} (see also \cite[\S7.1]{GR}). Then by \cite[Theorem 1]{Ko-Tam}, if the measures on $G_{\delta\sigma}$ and $G^*_{\ga^*}$ are chosen as in \cite{Ko-Tam}, we have $d(G_{\delta\sigma})=d(G^*_{\ga^*})$. Finally we note that since $\bfG^*_{\ga^*}$ is the quasi-split inner form of $\bfG_{\delta\sigma}$, we have $e(\bfG^*_{\ga^*})=1$ and $e(\bfG_{\delta\sigma})=(-1)^{q(G_{\delta\sigma})-q(G^*_{\ga^*})}$. This completes the proof of the claim, and hence of the proposition.
\end{proof}

\subsection{A reinterpretation of stable base change transfer}
When $\phi^* \in C^\infty_c(G^*)$ is the stable base change transfer of a function $\phi_r \in C^\infty_c(G_r)$, we will write this as $\phi_r \longleftrightarrow \phi^*$. When $\phi^*$ is the usual base-change of a function $\phi^*_r \in C^\infty_c(G^*_r)$, then we write ${\rm BC}(\phi^*_r) = \phi^*$. (By ``usual base-change'', we just mean what is described in Definition \ref{gen_BC_defn} below, when all groups in sight are quasi-split.) We now define a variant of the first relation.

\begin{defn}
For $\phi_r \in C^\infty_c(G_r)$ and $\phi^*_r \in C^\infty_c(G^*_r)$, we write $\phi_r \overset{\sigma-{\rm reg}}{\longleftrightarrow} \phi^*_r$ when the following property holds for every $\sigma$-regular semisimple element $\delta^* \in G^*_r$: 
\begin{itemize}
\item If there is a $\delta \in G_r$ such that $\mathcal N_r\delta = \mathcal N_r \delta^*$, then $TO^{G^*_r}_{\delta^* \sigma}(\phi^*_r) = TO^{G_r}_{\delta \sigma}(\phi_r)$.
\item If there is no such $\delta$, then $TO^{G^*_r}_{\delta^* \sigma}(\phi^*_r) = 0$.
\end{itemize}
\end{defn}

\begin{lem}\label{sbc_reinterpret}
Fix $\phi_r \in C^\infty_c(G_r)$ and $\phi^* \in C^\infty_c(G^*)$. 
\begin{enumerate}
\item[(1)] If $\phi_r \longleftrightarrow \phi^*$, then there exists a function $\phi^*_r \in C^\infty_c(G^*_r)$ such that $\phi_r \overset{\sigma-{\rm reg}}{\longleftrightarrow} \phi^*_r$ and ${\rm BC}(\phi^*_r) = \phi^*$.
\item[(2)] Conversely, if there exists a $\phi^*_r$ with $\phi_r \overset{\sigma-{\rm reg}}{\longleftrightarrow} \phi^*_r$ and ${\rm BC}(\phi^*_r) = \phi^*$, then $\phi_r \longleftrightarrow \phi^*$.
\item[(3)] If $G_r^* = (G^*)^r$, and if $\phi^* \in C^\infty_c(G^*)$ is a Jacquet-Langlands transfer of $\phi \in C^\infty_c(G)$, then $\phi^*$ is a stable base change transfer of $\phi_r := \phi \otimes e \otimes \cdots \otimes e \in C^\infty_c(G^r)$, where $e \in C^\infty_c(G)$ is an idempotent such that $\phi * e = \phi$.
\end{enumerate}
\end{lem}
\begin{proof}
(1). Suppose $\phi_r \longleftrightarrow \phi^*$.  We first claim that $\phi^* = {\rm BC}(\phi^*_r)$ for some $\phi^*_r \in C^\infty_c(G^*_r)$. Suppose $\ga^* \in G^*$ is regular semisimple. If $\ga^*$ is not a norm from $G^*_r$, then it is not a norm from $G_r$ (use Lemma \ref{lem:transfer-twisted-conj-class}(2)) and hence $O^{G^*}_{\ga^*}(\phi^*) = 0$.  This proves the claim thanks to \cite[Prop.\,I.3.1]{AC}.  Next we claim that $\phi_r \overset{\sigma-{\rm reg}}{\longleftrightarrow} \phi^*_r$.  Suppose that $\delta^* \in G^*_r$ is $\sigma$-regular semisimple, and write $\mathcal N_r \delta^* = \ga^* \in G^*$. If there is a $\delta \in G_r$ with $\mathcal N_r \delta = \mathcal N_r \delta^*$, then we have
$$
TO^{G_r}_{\delta \sigma}(\phi_r) = O^{G^*}_{\ga^*}(\phi^*) = TO^{G^*_r}_{\delta^* \sigma}(\phi^*_r),
$$
as desired.  Now suppose that there is no such $\delta$. Then we claim that $\ga^*$ is not a norm from $G_r$, since if $\ga^* = \mathcal N_r \delta$, then by Lemma \ref{lem:transfer-twisted-conj-class}(2), $\delta$ gives rise to $\delta^{**} \in G_r^*$ with $\mathcal N_r \delta = \mathcal N_r \delta^{**}$, and then by injectivity of the $\mathcal N_r$ for $G^*_r$, this forces $\delta^*$ to be $G^*_r$-$\sigma$-conjugate to $\delta^{**}$, which then implies that $\delta^*$ comes from $\delta$, a contradiction. The claim being proved, we see that $TO^{G^*_r}_{\delta^* \sigma}(\phi^*_r) = O^{G^*}_{\ga^*}(\phi^*) = 0$, the first equality following from $\phi^* = {\rm BC}(\phi^*_r)$.  This completes the proof of (1).\par
(2). Suppose that $\phi^*, \phi^*_r, \phi_r$ satisfy the given hypotheses.  We need to verify the condition defining $\phi_r \longleftrightarrow \phi^*$ at any semisimple element $\ga^* \in G^*$.  By Proposition \ref{red_to_reg_matching_prop}, we may assume $\ga^*$ is regular. Suppose $\ga^* = \mathcal N_r \delta$ for $\delta \in G_r$; let $\delta^* \in G^*_r$ be an element associated to $\delta$ as in Lemma \ref{lem:transfer-twisted-conj-class}(2). Then $\mathcal N_r \delta^* = \ga^*$ and our hypotheses give
$$
O^{G^*}_{\ga^*}(\phi^*) = TO^{G^*_r}_{\delta^* \sigma}(\phi^*_r) = TO^{G_r}_{\delta \sigma}(\phi_r).
$$
Now suppose that $\ga^*$ is not a norm from $G_r$. We claim that $O^{G^*}_{\ga^*}(\phi^*) = 0$.  If $\ga^*$ is not a norm from $G^*_r$, we are done using $\phi^* = {\rm BC}(\phi^*_r)$.  If $\ga^* = \mathcal N_r \delta^*$, then this $\delta^*$ does not come from any $\delta \in G_r$, and hence $O^{G^*}_{\ga^*}(\phi^*) = TO^{G^*_r}_{\delta^*\sigma}(\phi^*_r) = 0$, as desired.\par
(3).  Define $\phi^*_r = \phi^* \otimes e^* \otimes \cdots \otimes e^*$, where $e^* \in C^\infty_c(G^*)$ is an idempotent with $\phi^* * e^* = \phi^*$. By the proof of \cite[$\S I.5$]{AC}, we see that $\phi^*_r \overset{\sigma-{\rm reg}}{\longleftrightarrow} \phi_r$ in the obvious sense adapted to functions in $C^\infty_c(G^{*r})$ resp.,\, $C^\infty_c(G^r)$. By \cite[$\S I.5$]{AC} again, we also have ${\rm BC}(\phi^*_r) = \phi^*$, again in the obvious sense for the groups in question.  Then by the same formal argument as in (2), the assertion follows.
\end{proof}

\begin{cor} \label{phi*_is_bc_cor}
If a function $\phi^* \in C^\infty_c(G^*)$ is a stable base change transfer of a function on $G_r$, it is also a stable base change transfer of a function on $G^*_r$.    
\end{cor}

\subsection{Local stable base change}
\begin{thm}\label{thm:twisted-JL}
Let $\pi^*$ be an irreducible tempered representation of $G^*$ and let $\Pi^*$ be the base change lift of $\pi^*$, which is by definition a $\sigma$-stable representation of $G_r^*$ with canonical intertwining operator $I_\sigma^*$, see \cite[$\S I.2$]{AC}.
\begin{enumerate}
    \item If $\Pi^*$ has a Langlands-Jacquet transfer to an irreducible tempered representation $\Pi$ of $G_r$, then $\Pi$ is $\sigma$-stable and we can choose an intertwining operator $I_\sigma$ on $\Pi$ such that for any $\phi\in C_c^\infty(G_r)$ with stable base change transfer $\phi^*\in C_c^\infty(G^*)$ we have
    \[\mathrm{Tr}(\phi I_\sigma | \Pi)=e(\bfG_r)\mathrm{Tr}(\phi^*|\pi^*).\]
    %    for any pair of functions $\phi\in C_c^\infty(G_r)$ and $\phi^\dagger\in C_c^\infty(G_r^*)$ with matching twisted orbital integrals, we have
%    \[\mathrm{Tr}(\phi I_\sigma | \Pi)=e(\bfG_r)\mathrm{Tr}(\phi^\dagger I_\sigma^* | \Pi^*)\]
    \item If $\Pi^*$ does not have a Langlands-Jacquet transfer to $G_r$, then for any $\phi\in C_c^\infty(G_r)$ with stable base change transfer $\phi^*\in C_c^\infty(G^*)$ we have 
    \[\mathrm{Tr}(\phi^*|\pi^*)=0.\]
\end{enumerate}
\end{thm}

\subsubsection{Remark on $\sigma$-stability} \label{sigma-stable-rem} Suppose $\Pi$ is any irreducible tempered representation of $G_r$ with Jacquet-Langlands transfer the irreducible tempered representation $\Pi^*$ of $G_r^*$. Then $\Pi$ is $\sigma$-stable if $\Pi^*$ is $\sigma^*$-stable (the latter being automatic if $\Pi^*$ is the base-change of an irreducible tempered representation $\pi^*$ of $G^*$).  To see this, we interpret $\sigma$-stability in terms of Harish-Chandra characters: $\Pi$ is $\sigma$-stable if and only if $\Theta_{\Pi}(\sigma(\ga_r)) = \Theta_{\Pi}(\ga_r)$, for every regular semisimple $\ga_r \in G_r$; similar statements hold for $\Pi^*$ and $\Theta_{\Pi^*}$ and regular semisimple elements $\ga^*_r \in G^*_r$.  Moreover, the transfer relation $\ga_r \mapsto \ga^*_r$ is equivalent to $\sigma(\ga_r) \mapsto \sigma^*(\ga^*_r)$.  So $\Pi^* \cong \Pi^* \circ \sigma^*$ implies the following for all $\ga \mapsto \ga^*_r$ as above: 
\begin{equation}
\Theta_{\Pi}(\ga_r) = e(\bfG_{F_r}) \Theta_{\Pi^*}(\ga^*_r) = e(\bfG_{F_r}) \Theta_{\Pi^*}(\sigma^* (\ga^*_r)) 
= \Theta_\Pi(\sigma(\ga_r)),
\end{equation}
in other words, it implies $\Pi \cong \Pi \circ \sigma$.

\subsubsection{Proof of part (2)}
We can write $\Pi^*$ as a normalized parabolic induction $\Pi^* = i^{G^*_r}_{P^*_r}(\pi_{M^*_r})$ for some $F_r$-rational parabolic subgroup with Levi decomposition ${\mathbf P}^*_r = {\mathbf M}^*_r {\mathbf N}^*_r$ and with $F_r$-points $P^*_r= M^*_r N^*_r$, and with $\pi_{M^*_r}$ a discrete series representation of $M^*_r$.  If ${\mathbf M}^*_r$ transfers to an $F_r$-Levi subgroup ${\mathbf M}_r$ of $\bfG_{F_r}$ (a Levi factor of an $F_r$-parabolic ${\mathbf P}_r \subset \bfG_{F_r}$ with $F_r$-points $P_r$), then $\pi_{M^*_r}$ has a Langlands-Jacquet transfer $\pi_{M_r}$ on $M_r$. The representation $i^{G_r}_{P_r}(\pi_{M_r})$ is irreducible and tempered by \cite[Cor.\,6.3]{Tad15}, and hence is a Langlands-Jacquet transfer of $\Pi^*$.  So the only obstruction to transferring $\Pi^*$ is the failure of ${\mathbf M}^*_r$ to transfer.

Moreover, note that a tempered irreducible $\pi^* = i^{G^*}_{P^*}(\sigma^*)$ has base-change lift of the form $\Pi^* = i^{G^*_r}_{P^*_r}(\pi_{M^*_r})$, where $\pi_{M^*_r}$ is the base-change lift of $\sigma^*$.

In light of the above, we need to prove the following twisted version of a result of Badulescu (see \cite[Lem.\,3.3]{Bad03}). For $H$ the rational points of $p$-adic group over a $p$-adic field, let $\Pi(H)$ denote the set of isomorphism classes of smooth irreducible representations of $H$.

\begin{lem} \label{Frob-transfer_lem}
Let $\psi \in \mathcal H(G^*)$.  Suppose ${\rm O}_{\gamma^*}(\psi) = 0$ for all semisimple elements $\gamma^* \in G^*$ which are not norms from $G_r$. Then ${\rm Tr}(\psi| \pi^*) = 0$ for all $\pi^* \in \Pi(G^*)$ which are normalized parabolically induced representations $\pi^* = i^{G^*}_{P^*}(\sigma^*)$ where $P^* = M^* N^*$ is a standard parabolic subgroup of $G^*$ such that ${\mathbf M}^*_{F_r} = {\mathbf M}^* \otimes_{F} F_r$ does not transfer to $\bfG_{F_r} = \bfG \otimes_F F_r$, and $\sigma^* \in \Pi(M^*)$.
\end{lem}

\begin{proof}
Let $K^* \subset G^*$ be a hyperspecial maximal compact subgroup in good position with respect to $P^*$, so that an Iwasawa decomposition holds: $G^* = M^* N^* K^*$. Then define the normalized constant term $\psi^{P^*} \in \mathcal H(M^*)$ by the formula 
$$\psi^{P^*}(m^*) = \delta^{1/2}_{P^*}(m^*) \, \int_{N^*} \int_{K^*} \psi(k^{* -1}m^* n^* k^*) \, dk^* \, dn^*,$$ 
where $dk^*$ and $dn^*$ are Haar measures on $K^*$ and $N^*$ normalized by
$$
{\rm vol}_{dk^*}(K^*) = {\rm vol}_{dn^*}(N^* \cap K^*) = 1.
$$
Let $D_{G^*/M^*}(m^*)$ denote the normalized absolute value of ${\rm det}(1- {\rm Ad}(m^{* -1}) \, | \, {\rm Lie}(G^*)/{\rm Lie}(M^*))$. Then 
\begin{align} 
&{\rm Tr}(\psi \, | \, i^{G^*}_{P^*}(\sigma^*)) = {\rm Tr}(\psi^{P^*} \, | \, \sigma^*)  \label{const_term_eq1} \\
&{\rm O}^{G^*}_{m^*}(\psi) = D^{-1/2}_{G^*/M^*}(m^*) \, {\rm O}^{M^*}_{m^*}(\psi^{P^*}), \,\,\, \forall m^* \in G^{* \rm rs} \cap M^* \label{const_term_eq2}
\end{align}
where $G^{* \rm rs} \cap M^*$ is the set of elements of $M^*$ which are regular semisimple in $G^*$ (see e.g.\,\cite[Prop.\,4.3.11; Lem.\,7.5.7]{Lau96}).  

If (\ref{const_term_eq1}) is not zero, then applying the Weyl integration formula to $\psi^{P^*}$, we see that there must exist $m^* \in G^{* \rm rs} \cap M^*$ such that ${\rm O}_{m^*}(\psi^{P^*}) \neq 0$; hence by (\ref{const_term_eq2}), ${\rm O}^{G^*}_{m^*}(\psi) \neq 0$.  Then by assumption $m^*$ is a norm from $G_r$, say of $\delta \in G_r$. 

Using Lemma \ref{lem:transfer-twisted-conj-class}(2), choose $\delta^* \in G^*_r$ such that $N_r \delta^*$ is stably-conjugate to $N_r\delta$; then $m^*$ is $G^*_r$-conjugate to  $N_r\delta^*$; we may thus assume that $N_r \delta^* = m^*$. By replacing $M^*$ and $m^*$ by $G^*_r$-conjugates (which would mean our new $m^*$ might only be in $M^*_r$ instead of $M^*$, which is harmless) the proof of Lemma \ref{lem:transfer-twisted-conj-class}(2) shows that we may assume  the inner twisting ${\mathbf G}_{F_r}  \to {\mathbf G}^*_{F_r}$ takes $\delta$ to $\delta^*$. We get an inner twisting and hence an isomorphism between the $F_r$-tori ${\mathbf G}^*_{F_r, m^*}$ and ${\mathbf G}_{F_r,N_r\delta}$, which are maximal $F_r$-tori in ${\mathbf M}_{F_r}^*$ (resp.,\,${\mathbf G}_{F_r}$). Therefore the $F_r$-split component $A_{{\mathbf M}^*_{F_r}}$ of the center of ${\mathbf M}^*_{F_r}$ transfers to ${\mathbf G}_{F_r}$, which means that ${\mathbf M}^*_{F_r}$ itself transfers to $\bfG_{F_r}$, a contradiction.
\end{proof}

\subsubsection{Remark on a relative version of \textup{(}1\textup{)}} In various reductions coming into the proof of (1), we need to be able to change the group and the field of definition. In fact we will prove a more general version, where the base field $\bbQ_p$ is replaced by any finite extension $E$ thereof. We can consider $E$-inner forms $\bfG, \bfG^*$, giving rising as above to inner forms $\bfG_{r'}$ and $\bfG^*_{r'}$ corresponding to a degree $r'$ unramified extension field $E_{r'}$ of $E$. The statements are then the obvious analogues of (1) above, with $\sigma$ replaced by any generator $\sigma_E$ of $\mathrm{Gal}(E_{r'}/E)$. Note that $\sigma_E$ (not $\sigma$) is then used in the definition of twisted orbital integral and the notion of stable base change, when working relative to the extension $E_{r'}/E$. Although our proof will establish this more general version, to keep the notation simple we shall explain the proof over $E = \bbQ_p$, and after making the reductions in which $\bbQ_p$ is possibly replaced by a larger extension, we will then pretend that we are again working over $\bbQ_p$ and $r$ is once again $r'$.

\subsubsection{Reduction to a totally ramified Weil restriction}
%Jingren: below I rephrase and add details to the paragraph above (with some necessary modifications):\par
First we may assume that $\bfG^*=\mathrm{Res}_{F/\bbQ_p}\mathrm{GL}_n$ where $F$ is a finite extension of $\bbQ_p$. By Shapiro's Lemma, any $\bbQ_p$-inner form of $\bfG^* = \mathrm{Res}_{F/\bbQ_p} \mathrm{GL}_n$ is of the form $\bfG = \mathrm{Res}_{F/\bbQ_p} \bfG_1$ where $\bfG_1$ is an $F$-inner form of $\mathrm{GL}_n$.\par
Let $F/E/\bbQ_p$ be the unique subextension with $F/E$ totally ramified and $E/\bbQ_p$ unramified.  For any field $K$ with $\bbQ_p \subseteq K \subset \overline{\bbQ}_p$, let $\Gamma_K = {\rm Gal}(\overline{\bbQ}_p/K)$ be its absolute Galois group, viewed as a subgroup of $\Gamma_{\bbQ_p}$. By transitivity of Weil restriction, we have
$$
\bfG = \mathrm{Res}_{F/\bbQ_p} \bfG_1 = \mathrm{Res}_{E/\bbQ_p} \Big( \mathrm{Res}_{F/E} \bfG_1 \Big) =: \mathrm{Res}_{E/\bbQ_p} \big(\bfG'\big)
$$
where this defines $\bfG'$. Let $\bfG'^*=\mathrm{Res}_{F/E} \mathrm{GL}_{n,F}$ be the quasi-split inner form of $\bfG'$ over $E$. Then $\bfG^* = \mathrm{Res}_{E/\bbQ_p}\bfG'^*$ is the quasi-split inner form of $\bfG$. \par
Suppose $E=\bbQ_{p^s}$ is the unique unramified extension of $\bbQ_p$ of degree $s\ge1$. Let $d:=\mathrm{gcd}(r,s)$ and write $r=dr', s=ds'$. Let $E_{r'}=\bbQ_{p^{sr'}}=\bbQ_{p^{rs'}}$ be the unique degree $r'$ unramified extension of $E=\bbQ_{p^s}$, which is also the unique degree $s'$ unramified extension of $\bbQ_{p^r}$. Choose $a,b\in\bbZ$ such that $as+br=d$, or in other words $as'+br'=1$.

\begin{lem}
There exists an isomorphism of algebras $E \otimes_{\bbQ_p} \bbQ_{p^r} \cong (E_{\tau'})^d$ under which $1 \otimes \sigma$ corresponds to 
$$
(x_1, \dots, x_d) \mapsto (x_2, \dots, x_d, \tau(x_1))
$$
where $\tau = \sigma^d = (\sigma^s)^a$.
\end{lem}

\begin{proof}
There is an obvious isomorphism $\bbQ_{p^s} \otimes_{\bbQ_p} \bbQ_{p^r} \cong \bbQ_{p^s} \otimes_{\bbQ_{p^d}} \big( \bbQ_{p^d} \otimes_{\bbQ_p} \bbQ_{p_r})$. There is an isomorphism 
\begin{align}
\bbQ_{p^d} \otimes_{\bbQ_p} \bbQ_{p^r} &\overset{\sim}{\rightarrow} (\bbQ_{p^r})^d \label{sigma_goes_over}\\
x \otimes y &\mapsto (\sigma^{-i+1}(x)\,y)_{i= 1}^d \notag
\end{align}
under which ${\rm id} \otimes \sigma$ goes over to the automorphism
$$
(x_1, x_2, \dots, x_d) \mapsto (\sigma(x_2), \sigma(x_3), \dots, \sigma(x_d), \sigma(x_1)).
$$
If we post-compose (\ref{sigma_goes_over}) with the automorphism of the algebra $(\bbQ_{p^r})^d$ given by 
$$
(x_1, x_2, \dots, x_d) \mapsto (x_1, \sigma(x_2), \dots, \sigma^{d-1}(x_d)),
$$
then ${\rm id} \otimes \sigma$ on $\bbQ_{p^d} \otimes_{\bbQ_p} \bbQ_{p^r}$ is identified with the automorphism of $(\bbQ_{p^r})^d$ given by
$$
(x_1, x_2, \dots, x_d) \mapsto (x_2, x_3, \dots, \tau(x_1)).
$$
This map is now $\bbQ_{p^d}$-linear and furthermore since $\sigma^r = {\rm id}$, we see $\sigma^d = (\sigma^s)^a$.  Therefore ${\rm id} \otimes \sigma$ on 
$$
E \otimes_{\bbQ_p} \bbQ_{p^r} = \bbQ_{p^s} \otimes_{\bbQ_{p^d}} (\bbQ_{p^r})^d = (\bbQ_{p^{s'r'd}})^d
$$
is given by the same formula.
\end{proof}
We will view $\tau:=\sigma^d=(\sigma^s)^a$ as a generator of $\mathrm{Gal}(E_{r'}/E)$ (note that since $\sigma^s$ is a generator and $a$ is relatively prime to $r'$, $\tau$ is also a generator).\par 
This leads to the decompositions
\[\bfG_r:=\mathrm{Res}_{\bbQ_{p^r}/\bbQ_p}\bfG_{\bbQ_{p^r}}\cong(\mathrm{Res}_{E/\bbQ_p}\bfG'_{r'})^d,\quad\bfG_r^*\cong(\mathrm{Res}_{E/\bbQ_p}\bfG'^*_{r'})^d\]
where $\bfG'_{r'}:=\mathrm{Res}_{E_{r'}/E}(\bfG'\otimes_EE_{r'})$ and $\bfG'^*_{r'}:=\mathrm{Res}_{E_{r'}/E}(\bfG'^*\otimes_EE_{r'})$. The $\sigma$-actions on the right hand sides are given by the same formula as above. \par 

The $\sigma$-stable representation $\Pi^*$ of $G_r^*=\bfG_r^*(\bbQ_p)$ is decomposed into $\Pi^*=(\Pi'^*)^{\otimes d}$, where $\Pi'^*$ is a $\tau$-stable representation of $\bfG'^*_{r'}(E) = \bfG'^*(E_{r'})$. Let $I_\tau^*$ be the  canonical intertwining operator on $\Pi'^*$. Then the canonical intertwining operator $I_\sigma^*$ for $\Pi^*$ is given by
\[I_\sigma^*(v_1\otimes\dotsm\otimes v_d)=v_2\otimes\dotsm\otimes v_d\otimes I_\tau^*v_1,\quad\forall v_1,\dotsc,v_d\in\Pi'.\] 

For any function $\phi=\phi_1\otimes\dotsm\otimes\phi_d\in C_c^\infty(G_r^*)=C_c^\infty(\bfG'^*(E_{r'}))^{\otimes d}$, let $\phi':=\phi_1*\dotsm*\phi_d$ be the convolution product of the factors. Then by \cite[\S I.5]{AC}, the functions $\phi$ and $\phi'$ have matching twisted orbital integrals and hence they have the same base change transfer $\phi^*\in C_c^\infty(G^*)$. On the other hand, by the Saito-Shintani formula (see for example \cite[Lemma 6.12]{Feng}) we deduce that
\[\mathrm{Tr}(\phi I_\sigma^*|\Pi^*)=\mathrm{Tr}(\phi' I_\tau^*|\Pi'^*).\]
Indeed, to show this we are reduced to the case where for each $1\le i\le d$, $\phi_i$ is the characteristic function of a coset $g_iK$ for $g_i\in\bfG'^*(E_{r'})$ where $K\subset\bfG'^*(E_{r'})$ is a normal open compact subgroup. Let $V^K$ be the space of $K$-invariants in $\Pi'^*$. Then after dividing the volume factor $\vol(K)^d$ the identity above becomes 
\[\mathrm{Tr}((g_1\otimes\dotsm\otimes g_d)I_\sigma^* \mid (V^K)^{\otimes d}) = \mathrm{Tr}(g_1\dotsm g_d I_\tau^* | V^K)\]
and this follows from \cite[Lemma 6.12]{Feng}.\par

In particular, we see that $\Pi'^*$ is the base change lift of $\pi^*$ with respect to the extension $E_{r'}/E$. Moreover $\Pi^*$ has a Langlands-Jacquet transfer to $G_r$ if and only if $\Pi'^*$ has a Langlands-Jacquet transfer to $\bfG'_{r'}(E)=\bfG'(E_{r'})$ by part (2), which we already proved above.\par
Now suppose that the theorem is true for the data $(\bfG'_{r'}(E),\tau\in\mathrm{Gal}(E_{r'}/E))$.\par 
Suppose $\Pi^*$ has a Langlands-Jacquet transfer to $G_r$. Then $\Pi'^*$ has a Langlands-Jacquet transfer $\Pi'$, which is a representation of $\bfG'_{r'}(E)=\bfG'(E_{r'})$. By the validity of part (1) for the data $(\bfG'_{r'}(E),\tau\in\mathrm{Gal}(E_{r'}/E))$, the representation $\Pi'$ is $\tau$-stable and there is an intertwining operator $I_\tau$ such that  
\[\mathrm{Tr}(\phi' I_\tau | \Pi')=e(\bfG'_{r'})\mathrm{Tr}(\phi^*| \pi^*)\]
for all $\phi' \in  C_c^\infty(\bfG'(E_{r'}))$ with stable base change transfer $\phi^*\in C_c^\infty(G^*)$. Let $\Pi:=(\Pi')^{\otimes d}$ be the representation of $G_r=\bfG'_{r'}(E)^d$ and define an endomorphism $I_\sigma$ on $\Pi$ by the formula
\[I_\sigma(v_1\otimes\dotsm\otimes v_d)=v_2\otimes\dotsm\otimes v_d\otimes I_\tau v_1.\] 
Then one checks that $\Pi$ is $\sigma$-stable and $I_\sigma$ is an intertwining operator. Since $\Pi'$ is the Langlands-Jacquet transfer of $\Pi'^*$, it follows that $\Pi = (\Pi')^{\otimes d}$ is the Langlands-Jacquet transfer of $\Pi^* = (\Pi'^*)^{\otimes d}$. For any $\phi=\phi_1\otimes\dotsm \otimes \phi_d\in C_c^\infty(G_r)$ with $\phi':=\phi_1*\dotsm*\phi_d$, by the Saito-Shintani formula again we get 
\[\mathrm{Tr}(\phi I_\sigma|\Pi)=\mathrm{Tr}(\phi' I_\tau|\Pi')=e(\bfG'_{r'})\mathrm{Tr}(\phi^*|\pi^*)\]
where in the first equality we use that $\phi$ and $\phi'$ have the same stable base change transfer $\phi^*$ (by the same reasoning as in \cite[\S I.5]{AC}).\par
From \cite[Corollary on page 295]{Ko-sign} we have the following identities of Kottwitz signs:
\begin{align*}
e(\bfG_r) &= e(\bfG)^r = e(\bfG')^r \\
e(\bfG'_{r'}) &= e(\bfG')^{r'}.
\end{align*}
In other words, we have $e(\bfG_r) = e(\bfG'_{ r'})^d$. In case $e(\bfG_r)\ne e(\bfG'_{r'})$, we must have $e(\bfG'_{r'})=-1$ and $d$ is even. Then $r=r'd$ is also even and we can replace $I_\sigma$ by $-I_\sigma$ to get  
\[\mathrm{Tr}(\phi I_\sigma|\Pi)=e(\bfG_r)\mathrm{Tr}(\phi^*\pi^*)\]
and hence part (1) for the data $(G_r,\sigma\in\mathrm{Gal}(\bbQ_{p^r}/\bbQ_p))$ follows. \par 
We have now reduced the proof of the Theorem for the data $(\bfG_r(\bbQ_p), \sigma \in \mathrm{Gal}(\bbQ_{p^r}/\bbQ_p))$ to the proof of the Theorem for the data $(\bfG'_{r'}(E), \tau \in \mathrm{Gal}(E_{r'}/E))$.  In this way we reduced from the group $\bfG$ over $\bbQ_p$ to the group $\bfG'$ over $E$, which is a Weil-restriction for a {\em totally ramified} extension $F/E$. So we now abuse notation and pretend we are in the situation where $\bfG = \mathrm{Res}_{F/\bbQ_p} \bfG_1$ and $F/\bbQ_p$ is {\em totally ramified}. The arguments below work the same way in the actual situation where we would be working over the base field $E$ instead of $\bbQ_p$.   

\subsubsection{Analysis of $F/\bbQ_p$ totally ramified case}
In view of our reduction above to the case where $F/\bbQ_p$ is totally ramified, we may assume $F \otimes_{\bbQ_p} \bbQ_{p^r} = F_r$ is the unique degree $r$ unramified extension of $F$. 
For the rest of this section we fix the following notations. Let $B=M_s(D)$ be a central simple algebra over $F$ of dimension $n^2$, where $D$ is a division algebra over $F$ of dimension $d^2$, so that $sd=n$. Let $B_r:=B\otimes_FF_r$ and write $B_r=M_{t}(\widetilde{D})$ where $\widetilde{D}$ is a division algebra over $F_r$ of dimension $\tilde{d}^2$ so that $t\tilde{d}=n$ and $\tilde{d}$ divides $d$. Then $G=B^\times$ and $G_r=B_r^\times$ are the unit groups. More precisely, we view $B^\times$ (resp.,\,$B^\times_r$) via Weil restriction as an algebraic group over $\bbQ_p$ (resp.\,$\bbQ_{p^r}$). \par
The proof of (1) will occupy the rest of this subsection. It will involve a global method and will rely on a comparison of twisted trace formulas. We first observe that in the current situation the sign in part (1) is
\[e(\bfG_r)=(-1)^{n-t}.\]
Indeed, this follows from part (6) of the Corollary on p.\,295 of \cite{Ko-sign} and the discussion at the bottom of p.\,296 of \emph{loc.\,cit.} Thus the identity to be proved becomes
\[\mathrm{Tr}(\phi I_\sigma | \pi)=(-1)^{n-t}\mathrm{Tr}(\phi^*|\pi^*)\]
For simplicity, in the proof we will view $G$ and $G_r$ as algebraic groups over $F$ without further comment. 

\subsubsection{Remark on central characters} \label{central_char_rem}

Our next task is to introduce the global method, relying on the Deligne-Kazhdan very simple trace formula. A technical point is that this formula (and its twisted version) is valid only for automorphic forms which transform by a unitary central character under translation of the argument by the center. The groups we work with are not adjoint, but have split connected centers. For these particular groups, by a twisting argument we can assume without loss of generality that the local representations $\pi^*, \Pi, \Pi^*$ we start with, as well as all global automorphic representations, have trivial central characters, and that the functions we encounter are smooth, transform by the trivial character, and have compact support modulo center. (For more details on the twisting argument, see \cite[\S5.3]{Hai09}, in particular \cite[Lem.\,5.3.1(i)]{Hai09}.) \iffalse we are also making use of Lemma \ref{lem:transfer-twisted-conj-class} here to define the relation $\chi \mapsto \chi N_r$ relating characters on the center of $G_r$ with those for $G^*$\fi We will assume these conventions for the rest of this section without further comment. For any nonarchimedean local field $F$, we write $\mathcal H({\bbG}(F))$ for the Hecke algebra of smooth functions on $\bbG(F)$ having compact modulo center support, and which transform by the trivial character of the center.

\subsubsection{Approximation by automorphic representations}\label{sec:global-approx}

Let  $\widetilde{\bbF}/\bbF$ be a cyclic degree $r$ extension of number fields and $v_0$ a place of $\bbF$ such that
\begin{itemize}
    \item $v_0$ is inert in $\widetilde{\bbF}$ and the completion $\widetilde{\bbF}_{v_0}/\bbF_{v_0}$ is isomorphic to $F_r/F$;
    \item All archimedean places of $\bbF$ are complex.
\end{itemize}
We can choose a central division algebra $\bbD$ over $\bbF$ such that
\begin{itemize}
    \item $\bbD_{v_0}\cong B$;
    \item Let $S(\bbD)=\{v_0,v_1,\dotsm,v_m\}$ be the set of places where $\bbD$ is ramified. We may arrange that for all $1\le i\le m$, the place $v_i$ splits completely in $\widetilde{\bbF}$ and $\bbD_{v_i}$ is a division algebra.
\end{itemize}
The existence of $\bbD$ follows from the local-global exact sequence for Brauer groups in global class field theory, see for example \cite[\S21, p.196]{Mum-AV}.\par 

Denote $\widetilde{\bbD}:=\bbD\otimes_\bbF\widetilde{\bbF}$. Then by the second assumption above, $\widetilde{\bbD}$ is a central division algebra over $\widetilde{\bbF}$. Let $\bbG:=\bbD^\times$ be the unit group of $\bbD$ and $\widetilde{\bbG}:=\mathrm{Res}_{\widetilde{\bbF}/\bbF}\widetilde{\bbD}^\times$. Then $\bbG$ is an inner form of $\bbG^*:=\mathrm{GL}_n$ and $\widetilde{\bbG}$ is an inner form of $\widetilde{\bbG}^*:=\mathrm{Res}_{\widetilde{\bbF}/\bbF}\mathrm{GL}_n$. We fix an isomorphism $\widetilde{\bbG}_{v_0}\cong G_r$.\par 
Moreover we choose two non-archimedean places $v_{m+1}, v_{m+2}$ of $\bbF$ not in the set $\{v_0,\dotsm,v_m\}$ such that $v_{m+1}$ splits completely in $\widetilde{\bbF}$ and $v_{m+2}$ is inert in $\bbF$.\par 

Fix a supercuspidal representation $\rho$ of $\mathrm{GL}_n(\bbF_{v_{m+1}})$. For later purposes, choose a matrix coefficient $f_\rho$ of $\rho$ and an idempotent $e_\rho\in {\mathcal H}(\mathrm{GL}_n(\bbF_{v_{m+1}}))$ such that $f_\rho*e_\rho=f_\rho$.\par 

Let $\cA^*$ be the set of $\sigma$-stable cuspidal automorphic representations $\widetilde{\Pi}^*$ of $\widetilde{\bbG}^*(\bbA_\bbF)=\mathrm{GL}_n(\bbA_{\widetilde{\bbF}})$ satisfying the following conditions:
\begin{itemize}
    \item $\widetilde{\Pi}^*_{v_0}$ is $d$-compatible in the sense of \cite{Badu}: there exists a regular semisimple element $\ga^* \in \widetilde{\bbG}^*(\bbF_{v_0})$ which is a transfer of an element in $\widetilde{\bbG}(\bbF_{v_0})$, such that $\Theta_{\widetilde{\Pi}^*_{v_0}}(\ga^*) \neq 0$;
    \item For all $1\le i\le m$, at any place $w$ of $\widetilde{\bbF}$ above $v_i$, $\widetilde{\Pi}^*_w$ is isomorphic to the Steinberg representation;
    \item At any place $w$ above $v_{m+1}$, $\widetilde{\Pi}^*_w$ is isomorphic to the supercuspidal representation $\rho$ fixed above. 
    \item At the unique place of $\widetilde{\bbF}$ above $v_{m+2}$, $\widetilde{\Pi}^*$ is isomorphic to the Steinberg representation.
\end{itemize}
Note that, like any local component of a cuspidal automorphic representation, $\widetilde{\Pi}^*_{v_0}$ is unitary and generic. We remark here that since $\widetilde{\Pi}_{v_0}$ is generic and $\sigma$-stable, it is the base-change of a generic representation of $G^*$, characterized by Harish-Chandra character identities (see \cite[p.\,59-60]{AC} and \cite[Thm.\,9.7]{Zel80}). We will see in Proposition \ref{prop:density} below that the set $\cA^*$ is nonempty.\par 
By \cite[Thm.\,5.1 and Prop.\,5.5]{Badu}, such a $\widetilde{\Pi}^*$ has a global Langlands-Jacquet transfer to a discrete automorphic representation $\widetilde{\Pi}$ of $\widetilde{\bbG}$ and moreover $\widetilde{\Pi}$ is cuspidal. Let $\cA$ be the set of automorphic representations of $\widetilde{\bbG}$ that are global Langlands-Jacquet transfers (in Badulescu's sense) of representations in $\cA^*$. 

\begin{lem} \label{cA_characterization}
   The set $\cA$ consists of $\sigma$-stable cuspidal automorphic representations $\widetilde{\Pi}$ of $\widetilde{\bbG}(\bbA_\bbF)$ such that
\begin{itemize}
    \item For all $1\le i\le m$, at any place $w$ of $\widetilde{\bbF}$ above $v_i$, $\widetilde{\Pi}_w$ is the trivial representation;
    \item At any place $w$ above $v_{m+1}$, $\widetilde{\Pi}_w$ is isomorphic to the supercuspidal representation $\rho$.
    \item At the unique place of $\widetilde{\bbF}$ above $v_{m+2}$, $\widetilde{\Pi}$ is isomorphic to the Steinberg representation.
\end{itemize} 
\end{lem}
\begin{proof}
    Let $\widetilde{\Pi} \in \cA$ be the global Jacquet-Langlands transfer of $\widetilde{\Pi}^*\in\cA^*$. Then clearly $\widetilde{\Pi}$ satisfies the stated conditions at the places $w$ above the places $v_i$, $1\le i\le m+1$. The fact that $\widetilde{\Pi}\in\cA$ is $\sigma$-stable follows using Remark \ref{sigma-stable-rem} along with Flath's article \cite{Fl79}.  In the other direction, \cite[Appendix]{Badu} shows that every cuspidal representation $\widetilde{\Pi}$ arises as a Badulescu-Langlands-Jacquet transfer of a cuspidal representation $\widetilde{\Pi}^*$ whose local factors are all $d$-compatible. If the local factors of $\widetilde{\Pi}$ satisfy the properties above, then $\widetilde{\Pi}^*$ must satisfy the four properties defining $\cA^*$. Further, the $\sigma$-invariance of $\widetilde{\Pi}$ forces, by injectivity of Badulescu's map $\bfG^{-1}$ in \cite[Prop.\,5.5]{Badu}, that $\widetilde{\Pi}^*$ is also $\sigma$-invariant.  Hence $\widetilde{\Pi}$ belongs to $\cA$.
\end{proof}

\begin{rem} 
Note we did not impose a condition at $v_0$ in Lemma \ref{cA_characterization}. In fact $\widetilde{\Pi} \in \cA$ forces $\widetilde{\Pi}_{v_0}$ to be an irreducible unitary representation which is in the image of a $d$-compatible representation of $\widetilde{\bbG}^*(\bbF_{v_0})$ under the map $|\mathrm{LJ}^u|_{v_0}$ of \cite[Prop.\,3.10]{Badu}.  But not every irreducible generic unitary representation of $\widetilde{\bbG}(\bbF_{v_0})$ is in the image of $|\mathrm{LJ}^u|_{v_0}$; see \cite[Lem.\,3.11]{Badu}. Hence this property of $\widetilde{\Pi}_{v_0}$ is highly non-obvious, and ultimately follows from the results of \cite[Appendix]{Badu}, by the argument above.
\end{rem}

\begin{lem}\label{lem:twisted-trace-global}
Let $\widetilde{\Pi}^*$ be a $\sigma$-stable cuspidal automorphic representation of $\widetilde{\bbG}^*$ in the set $\cA^*$. Let $\pi_{v_0}^*$ be an irreducible generic smooth representation of $\bbG^*(\bbF_{v_0})$ whose base change to $\bbG^*(\widetilde{\bbF}_{v_0})=\widetilde{\bbG}^*(\bbF_{v_0})$ is $\widetilde{\Pi}^*_{v_0}$. Let $\widetilde{\Pi}\in\cA$ be the Badulescu-Langlands-Jacquet transfer of $\widetilde{\Pi}^* \in \cA^*$. Then we can normalize the intertwining operator $I_\sigma$ on the $\sigma$-stable representation $\widetilde{\Pi}_{v_0}$ of $\widetilde{\bbG}(\bbF_{v_0})=\bbG(\widetilde{\bbF}_{v_0})$ such that 
    \[\mathrm{Tr}(\phi_{v_0} I_\sigma|\widetilde{\Pi}_{v_0})=e(\bfG_r)\mathrm{Tr}(\phi^*_{v_0}|\pi_{v_0}^*).\]
for every $\phi_{v_0}\in {\mathcal H}(\widetilde{\bbG}(\bbF_{v_0}))$ whose stable base change transfer is $\phi^*_{v_0}\in {\mathcal H}(\bbG^*(\bbF_{v_0}))$.
\end{lem}

\begin{proof}
\iffalse {\cm By \cite[Chapter 3, Theorem 4.2 and Theorem 5.1]{AC}, there exists a cuspidal automorphic representation $\pi^*$ of $\bbG^*$ whose local component at $v_0$ is isomorphic to the given representation $\pi_{v_0}^*$, and whose strong base change lifting (in the sense of \cite[Chapter 3, Definition 1.2]{AC}) to $\widetilde{\bbG}^*$ is isomorphic to $\widetilde{\Pi}^*$.}{\cg Tom: I am worried about the fact that the theorems you cite require $r$ to be prime. I don't think we need this anyway.  We could just rephrase the Lemma, as follows.  Let $\widetilde{\Pi}^*$ be unchanged, but let $\pi^*$ be any cuspidal representation whose base change is $\widetilde{\Pi}^*$.  (We should not need to cite any theorem to see that $\pi^*$ exists: its existence is established during the proof, namely fro   
 m the equation (2.9).) Now just let $\pi^*_{v_0}$ be the $v_0$-component of $\pi^*$.  The assertion should be that there exists a normalization of $I_\sigma$ such that the displayed equation in Lemma 2.3.7 holds, for all matching function $\phi_{v_0}, \phi^*_{v_0}$. }%{\cg Tom: the fact that $\pi^*_{v_0}$ occurs as a factor in a cuspidal global $\pi^*$ seems too strong -- the set of local factors is dense but not all po}  
\par \fi
We consider test functions $\phi=\otimes_w\phi_w\in {\mathcal H}(\widetilde{\bbG}(\bbA_\bbF))={\mathcal H}(\bbG(\bbA_{\widetilde{\bbF}}))$ that satisfy the following conditions (here $w$ runs over all places of $\widetilde{\bbF}$):
\begin{itemize}
    \item For any $1\le i\le m$, fix one place $w_i$ of $\widetilde{\bbF}$ above $v_i$. For any place $w|v_i$ different from $w_i$, $\phi_w=e_{v_i}$ is an idempotent such that $\phi_{w_i}*e_{v_i}=\phi_{w_i}$, where $e_{v_i}$ and $\phi_{w_i}$ are understood as functions on $\bbG(\bbF_{v_i})$ under fixed isomorphisms $\bbF_{v_i}\cong\widetilde{\bbF}_{w}\cong\widetilde{\bbF}_{w_i}$. We assume that the support of $\phi_{w_i}$ is contained in the elliptic regular semisimple locus of $\bbG(\bbF_{v_i})$.
    \item Fix a place $w_{m+1}$ above $v_{m+1}$, we require $\phi_{w_{m+1}}$ to be a matrix coefficient of the supercuspidal representation $\rho\cong\widetilde{\Pi}_{v_{m+1}}\cong\widetilde{\Pi}^*_{v_{m+1}}$ and for any $w|v_{m+1}$ different from $w_{m+1}$, $\phi_w$ is the idempotent $e_\rho$ so that $\phi_{w_{m+1}}*e_\rho=\phi_{w_{m+1}}$.
    \item At the unique place $w_{m+2}$ of $\widetilde{\bbF}$ above $v_{m+2}$, we require $\phi_{w_{m+2}}$ to be a pseudo-coefficient of the Steinberg representation of $\widetilde{\bbG}(\bbF_{v_{m+2}})=\bbG(\widetilde{\bbF}_{w_{m+2}})$.  
    \item For any non-archimedean place $v\notin\{v_i\}_{0\le i\le m+2}$, $\phi_v$ is in the unramified Hecke algebra and is the unit element for all but finitely many such $v$. 
\end{itemize}
For each place $v$ of $\bbF$, let $\phi_v^*\in {\mathcal H}(\bbG^*(\bbF_v))$ be the stable base change transfer of $\otimes_{w|v}\phi_w$ and let $\phi^*:=\otimes_{v}\phi_v^*$. (In particular, at $v = v_0$, the $\phi^*_{v_0}$ and $\phi_{v_0}$ range freely over all matching pairs of functions.) Then we can apply the simple twisted trace formula to the test functions $\phi,\phi^*$. Comparison of the geometric sides leads to the identity
\begin{equation}\label{eq:twisted-trace-formula}
    r\,\mathrm{Tr}(R_{\widetilde{\bbG},\mathrm{cusp}}(\phi) I_\sigma)=
    \mathrm{Tr}(R_{\bbG^*,\mathrm{cusp}}(\phi^*))
\end{equation}
for all pairs $\phi, \phi^*$ we consider. Here $I_\sigma$ on the left hand side refers to the canonical intertwining operator on $L^2$-functions given by the action of $\sigma$ on the argument. To prove (\ref{eq:twisted-trace-formula}) we follow the reasoning of \cite[p.44]{AC}, making use of Lemma \ref{lem:transfer-twisted-conj-class}. Then following the arguments in \cite[bottom of page 55]{AC}, we deduce\footnote{It is for this step that we need to use the inert place $v_{m+2}$. We also use Corollary \ref{phi*_is_bc_cor} in the course of the argument.} that there is a cuspidal automorphic representation $\pi^*$ of $\bbG^*$ satisfying
\begin{equation} \label{one_piece_eq}
\mathrm{Tr}(\phi I_\sigma | \widetilde{\Pi})=\mathrm{Tr}(\phi^*|\pi^*).
\end{equation}

For each $1\le i\le m$, we fix the local intertwining operator $I_{\sigma,v_i}$ on $\widetilde{\Pi}_{v_i}\cong\widetilde{\Pi}_{w_i}^{\otimes r}$ to be given by the cyclic permutation $I_{\sigma, v_i}(x_1 \otimes \cdots \otimes x_r) = x_2 \otimes x_3 \otimes \cdots x_r \otimes x_1$. Let $\phi^*_{w_i}$ be a Jacquet-Langlands transfer of $\phi_{w_i}$. Then by Lemma \ref{sbc_reinterpret}(3) -- rather by its trivial central character variant, see Remark \ref{sbc_reinterpret}--  $\phi^*_{w_i}$ is a stable base change for the local test function $\phi_{v_i}=\phi_{w_i}\otimes e_{v_i}\otimes\dotsm \otimes e_{v_i}$. Moreover, we have
\[\mathrm{Tr}(\phi_{v_i}I_{\sigma,v_i}|\widetilde{\Pi}_{v_i})=\mathrm{Tr}(\phi_{w_i}|\widetilde{\Pi}_{w_i})=(-1)^{n-1}\mathrm{Tr}(\phi_{w_i}^*|\pi^*_{v_i}).\]
The first equality follows from the Saito-Shintani formula. For the second equality we have used that $\widetilde{\Pi}$ is the Badulescu-Langlands-Jacquet transfer of a unique  $\widetilde{\Pi}^* \in \cA^*$, and by construction $\pi^*_{v_i}\cong\widetilde{\Pi}^*_{w_i}$ is the local Jacquet-Langlands transfer of $\widetilde{\Pi}_{w_i}$.\par 
For any place $v$ of $\bbF$ not in the set $\{v_i,0\le i\le m\}$, we have $\widetilde{\bbG}(\bbF_v)\cong\mathrm{GL}_n(\widetilde{\bbF}\otimes_{\bbF}\bbF_{v})$ and we fix the intertwining operator $I_{\sigma,v}$ on $\widetilde{\Pi}_v$ as in \cite{AC}. We claim that at the place $v_{m+2}$ we have 
\[\mathrm{Tr}(\phi_{w_{m+2}}I_{\sigma,w_{m+2}}|\widetilde{\Pi}_{w_{m+2}})=\mathrm{Tr}(\phi_{v_{m+2}}^*|\pi_{v_{m+2}}^*)\]
where $\widetilde{\Pi}_{w_{m+2}}$ is isomorphic to the Steinberg representation of $\bbG(\widetilde{\bbF}_{w_{m+2}})\cong\mathrm{
GL}_n(\widetilde{\bbF}_{w_{m+2}})$ by assumption. 
% Jingren: I agree with the following change
%{\cg Tom: I don't see how we can assert this without already knowing that $\widetilde{\Pi}^*_{w_{m+2}}$ is the base change of $\pi^*_{v_{m+2}}$. I think the order of the argument you wrote has to be changed. My proposal is what follows.}
By our assumption on $\phi_{w_{m+2}}$, its stable base change transfer $\phi^*_{v_{m+2}}$ is a pseudo-coefficient of the Steinberg representation of $\bbG(\bbF_{v_{m+2}})=\mathrm{GL}_n(\bbF_{v_{m+2}})$. Therefore as the factor $\mathrm{Tr}(\phi_{v_{m+2}}^*|\pi_{v_{m+2}}^*)$ appearing in (\ref{one_piece_eq}) is not zero, we must have that $\pi^*_{v_{m+2}}$ is the Steinberg representation. Then the Steinberg representation $\widetilde{\Pi}_{v_{m+2}}$ is its base-change lift, and we get the identity above.

%Jingren: the comment below has been addressed.
%{\cg  Tom: there is a small gap here: we do not know yet that $\widetilde{\Pi}^*_{v_{m+2}}$ is a base change lift of $\pi^*_{v_{m+2}}$.  We can fix this by requiring an extra condition of the function $\phi_{w_{m+2}}$ -- it should be a pseudo-coefficient of the Steinberg representation of $\widetilde{\bbG}(\bbF_{v_{m+2}})$. Then we choose $\phi^*_{v_{m+2}}$ as a pseudo-coefficient of the Steinberg `downstairs'.  Then the displayed identity above forces $\pi^*_{v_{m+2}}$ to base change to $\widetilde{\Pi}^*_{v_{m+2}}$. I will fix this tomorrow.}

On the other hand, at any place $v$ of $\bbF$ not in $\{v_i\}_{1\le i\le m+2}$ such that $\widetilde{\Pi}_v$ is unramified, $I_{\sigma,v}$ is the identity on the one dimensional space of spherical vectors $\widetilde{\Pi}_v^{\widetilde{\bbG}(\cO_{\bbF,v})}$. These choices determine uniquely an intertwining operator $I_{\sigma,v_0}$ on $\widetilde{\Pi}_{v_0}$ so that 
\begin{equation*}
\mathrm{Tr}(\phi_{v_0}I_{\sigma,v_0}|\widetilde{\Pi}_{v_0})=(-1)^{m(n-1)}\mathrm{Tr}(\phi_{v_0}^*|(\pi^*)_{v_0}),
\end{equation*}
where $(\pi^*)_{v_0}$ is the $v_0$-component of $\pi^*$, which despite our notation is {\em a priori} different from the representation $\pi^*_{v_0}$ we started with.

%Jingren: The following comment has been addressed.
%{\cg Tom: note that to make the following argument, we need to move the sign argument to here and reformulate the statement with the sign in (\ref{local_intertwiner_forced}) replaced by $(-1)^{n-t}$.}

Next we show that the sign $(-1)^{m(n-1)}$ differs from the desired sign $e(\bfG_r)=(-1)^{n-t}$ by an $r$-th root of unity.
Recall the notation from \S\ref{sec:global-approx}. In particular we have $n=t\tilde{d}$. First suppose that $n$ is odd, so $t$ is also odd. Then $m(n-1)-(n-t)$ is even and we are done.\par 
We assume from now on that $n$ is even. Recall that $\bbD_{v_0}\cong\widetilde{\bbD}_{v_0}\cong B$. Write its invariant as $\mathrm{inv}(B)=\frac{a}{\tilde{d}}=\frac{at}{n}$ where $a$ is coprime to $\tilde{d}=n/t$. For each $1\le i\le m$, write $\mathrm{inv}(\bbD_{v_i})=\frac{a_i}{n}$ where $a_i$ is coprime to $n$. In particular, $a_i$ is odd for all $1\le i\le m$. The sum of local invariants of $\widetilde{\bbD}$ vanishes, which implies that
\[at+r(a_1+\dotsm+a_m)\equiv 0\mod n.\]
Suppose $t$ is even, then we get that $rm$ is even. If $r$ is odd then $m$ must be even and hence $m(n-1)-(n-t)$ is even. If $r$ is even, then $-1$ is an $r$-th root of unity and we are done. \par
Finally if $t$ is odd, then $\tilde{d}$ is even and $a$ is odd (since it is coprime to $\tilde{d}$). Then the congruence above implies that $rm$ is odd, so in particular $m$ is odd. Then $m(n-1)-(n-t)$ is even and we are done.\par

Consequently after multiplying by an $r$-th root of unity if necessary, we get an intertwining operator $I_{\sigma,v_0}$ on $\widetilde{\Pi}_{v_0}$ such that 

\begin{equation} \label{local_intertwiner_forced}
\mathrm{Tr}(\phi_{v_0}I_{\sigma,v_0}|\widetilde{\Pi}_{v_0})=e(\bfG_r)\mathrm{Tr}(\phi_{v_0}^*|(\pi^*)_{v_0}).
\end{equation}

We regard (\ref{local_intertwiner_forced}) as equation in all matching pairs $\phi_{v_0}, \phi^*_{v_0}$.  By the (twisted) Weyl integration formula and Lemma \ref{sbc_reinterpret}, this implies that $\widetilde{\Pi}^*_{v_0}$ is a base-change lift of $(\pi^*)_{v_0}$. Therefore $\pi^*_{v_0}$ and $(\pi^*)_{v_0}$ differ by a character of $\bbF_{v_0}^\times$. Using that $v_0$ is inert in $\widetilde{\bbF}$, we replace $\pi^*$ by its twist by the corresponding character of $\bbA_{\widetilde{\bbF}}^\times$. By the fact that every component of the global $\phi^*$ is in the image of the base-change map (see Corollary \ref{phi*_is_bc_cor}), this twist does not affect the validity of (\ref{one_piece_eq}). We thereby arrange that the $v_0$-component of $\pi^*$ is exactly the given $\pi^*_{v_0}$, and the Lemma is proved.
\end{proof}

Let $\cA_{v_0}^*$ (resp. $\cA_{v_0}$) be the subset of $\mathrm{Irr}(G_r^*)$ (resp. $\mathrm{Irr}(G_r)$) consisting of $v_0$ components of the automorphic representations in $\cA^*$ (resp. $\cA$).\par

Let $\mathfrak{S}_2(G_r)$ be the set of equivalence classes of pairs $(L,D)$ where $L$ is a Levi subgroup of $\bfG_{F_r}$ defined over $F_r$ and $D$ is an $X_{ur}^*(L)$-orbit of a discrete series representation of $L$. For each $(L,D)\in\mathfrak{S}_2(G_r)$, let $V(L,D)$ be the corresponding component in the $L^2$-Bernstein variety of $G_r$.
%The following question has been addressed.
%{\cg Tom: is local JL correspondence being used to view $V(L,D)$ as a set of representations of $G^*_r$ in the next proposition? I don't quite see how this works, maybe we are using the irreducibility of normalized parabolic inductions proved in \cite{Tad15}, and then LJ of these?} 
By the local Jacquet-Langlands correspondence, we also view $V(L,D)$ as a component in the $L^2$-Bernstein variety of $G_r^*$. We refer to \cite[\S2]{Shin-AJM} for more details.

\begin{prop}\label{prop:density}
For each $\sigma$-stable pair $(L,D)\in\mathfrak{S}_2(G_r^*)$, the set $\cA^*_{v_0}\cap V(L,D)$ is Zariski dense in $V(L,D)$.
\end{prop}
\begin{proof}
Suppose on the contrary that the intersection is not Zariski dense. Then by the twisted version of the trace Paley-Wiener theorem \cite{Ro88} (see also \cite[Prop.\,2.9, Cor.\,2.10]{AC}), there exists a function $f_0\in \mathcal H(G_r^*)$ whose twisted character vanishes on $\cA^*_{v_0}\cap V(L,D)$ and all other Bernstein components, but does not vanish identically on $V(L,D)$. In particular, there exists $\delta_0\in\widetilde{\bbG}^*({\bbF}_{v_0})$ such that 
\[TO_{\delta_0\sigma}(f_0)\ne0\]
We consider test functions $\phi=\otimes_w\phi_w\in \mathcal H(\widetilde{\bbG}^*(\bbA_{\bbF}))$ satisfying the following conditions:
\begin{itemize}
    \item $\phi_{v_0}=f_0$;
    \item For any $1\le i\le m$, at one place $w$ above $v_i$, $\phi_w=f_{\mathrm{St}}$ is a matrix coefficient of the Steinberg representation and at the remaining $r-1$ places above $v_i$, the corresponding local factor of $\phi$ equals to an idempotent $e_{\mathrm{St}}$ in the local Hecke algebra such that $f_{\mathrm{St}}*e_{\mathrm{St}}=f_{\mathrm{St}}$;
    %{\cm \item Let ${w_1,\dotsc,w_r}$ be the places of $\widetilde{\bbF}$ above $v_{m+2}$ and let $\Omega_i:=\mathrm{Supp}(\phi_{w_i})$. Then $\Omega_1\dotsm\Omega_r$ is contained in the set of regular elliptic elements;}{\cg Tom: the place $v_{m+2}$ was already required to be inert, and anyway the regular elliptic support condition follows from the requirements made in the lines below.}
    \item At one place $w$ above $v_{m+1}$, $\phi_{w}=f_\rho$ is a matrix coefficient of the supercuspidal representation $\rho$ and at the remaining $r-1$ places above $v_{m+1}$, the local factor of $\phi$ equals to the idempotent $e_\rho$;
    \item At the unique place $w_{m+2}$ above $v_{m+2}$, $\phi_{w_{m+2}}$ is a pseudo-coefficient of the Steinberg representation.
    \item Let $S$ be a finite set of finite primes of $\bbF$ disjoint from all the places fixed above, and containing an additional prime $v_{m+3}$ of $\bbF$ distinct from $\{ v_i \, | \, 0 \leq i \leq m+2\}$ and totally split $\widetilde{\bbF}$. For every $v \in S$ we consider a pair of matching functions $\phi_v, f_v$ which are supported on the set of ($\sigma$-)regular elements, and assume that $\phi_{v_{m+3}}, f_{v_{m+3}}$ are supported on the ($\sigma$-)elliptic elements. 
\end{itemize}

By \cite[Lemma 2.5]{AC}, we have the simple twisted trace formula
\begin{equation} \label{AC89_simple_TF}
\sum_{\widetilde{\Pi}^*}\mathrm{Tr}(\phi I_\sigma|\widetilde{\Pi}^*)=\sum_\delta c_\delta TO_{\delta\sigma}(\phi)
\end{equation}
where the left hand side runs over $\sigma$-stable cuspidal automorphic representations of $\widetilde{\bbG}^*$, the right hand side runs over semisimple conjugacy classes $\delta$ in $\widetilde{\bbG}^*(\bbF)=\mathrm{GL}_n(\widetilde{\bbF})$ such that $\cN_r\delta$ is regular elliptic in $\bbG^*(\bbF)=\mathrm{GL}_n(\bbF)$ and $c_\delta$ is a certain nonzero volume factor. \par 
Note that any $\widetilde{\Pi}^*$ which contributes to the left hand side belongs to $\cA^*$. By our choice of $\phi$, the left hand side vanishes. We can choose $\delta\in\bbG^*(\widetilde{\bbF})$ satisfying:
\begin{itemize}
    \item At $v_0$, $\delta$ is sufficiently close to $\delta_0$ in $\widetilde{\bbG}(\bbF_{v_0})$ so that $TO_{\delta\sigma}(\phi_{v_0})\ne0$;
    \item For $1\le i\le m$ or $i=m+2$, at any place $w$ above $v_i$, $\delta$ is strongly $\sigma$-regular and $\sigma$-elliptic, and is sufficiently close to $1$ above $v_{m+1}$ so that $TO_{\delta\sigma}(\phi_{v_i})\ne0$ (by our choice of $\phi_{v_i}$, $TO_{\delta\sigma}(\phi_{v_i})$ equals to the character of the Steinberg or the supercuspidal representation $\rho$ at $\cN_r\delta$).
    
\end{itemize}
For functions having given support, the number of terms appearing on the right hand side of (\ref{AC89_simple_TF}) is finite. By enlarging the set $S$ if necessary and by choosing the functions $\phi_v, f_v$ for $v \in S$ appropriately, we can arrange that the only term which appears is our chosen $\delta$ (see \cite[p.\,295]{Clo90} and \cite[p.\,812-813]{KR00}).
Then the right hand side of the simple twisted trace formula has precisely one nonzero term corresponding to $\delta$ and we reach a contradiction. 
\end{proof}

\subsubsection{Finishing the proof of Theorem~\ref{thm:twisted-JL}}
Let $(L^*,D^*)\in\mathfrak{S}(G_r^*)$ be a $\sigma$-stable pair corresponding to $(L,D)\in\mathfrak{S}(G_r)$ under the Jacquet-Langlands transfer. Let $\pi_L$ be an irreducible discrete series representation in the orbit $D$. We know from Proposition~\ref{prop:density} and Lemma \ref{lem:twisted-trace-global} that the set of $\sigma$-stable irreducible generic representations in the component $V(L,D)$ of the $L^2$-Bernstein variety for $G_r$ is Zariski dense. This implies that all irreducible generic representations in $V(L,D)$ are $\sigma$-stable. Indeed, for any $\phi\in C_c^\infty(G_r)$, the trace of $\phi$ and $\phi^\sigma$ define regular function on the variety $V(L,D)$, where $\phi^\sigma(g):=\phi(\sigma^{-1}(g))$. Since they coincide on a Zariski dense subset, they are equal on $V(L,D)$. This means that the character of any irreducible generic representation in $V(L,D)$ is $\sigma$-stable and hence the representation is $\sigma$-stable by Remark \ref{sigma-stable-rem}.\par 
Consequently the $G_r$-conjugacy class of $(L,D)$ is $\sigma$-stable. Let $P$ be a parabolic subgroup of $G_r$ containing $L$. Then there exists $w\in G_r$ such that $w\sigma^{-1}(P)w^{-1}=P$, $w\sigma^{-1}(L)w^{-1}=L$, together with an operator $I_L^\sigma$ on $\pi_L$ such that 
\[I_L^\sigma \pi_L(w\sigma^{-1}(h)w^{-1})=\pi_L(h)I_L^\sigma,\quad\forall h\in P.\]
Define an intertwining operator $I_\sigma$ on the induced representation $\pi=i_P^G\pi_L$ by
\[(I_\sigma f)(g):=I_L^\sigma f(w\sigma^{-1}(g)),\quad\forall g\in G_r.\]
Then we have $I_\sigma\pi(g)=\pi(\sigma(g))I_\sigma$ for all $g\in G_r$. Moreover, the same formula defines an intertwining operator, still denoted $I_\sigma$, on (the Jacquet-Langlands transfers of) any irreducible generic representation in the component $V(L,D)$, since such representations are full induced modules by Zelevinsky's theorem (\cite[9.7]{Zel80}).

For any $\phi\in C_c^\infty(G_r)$, the trace of $\phi I_\sigma$ defines a regular function on the variety $V(L,D)$. Therefore the twisted character identity for any representation in $V(L,D)$ can be reduced to representations in the Zariski dense subset $\cA_{v_0}\cap V(L,D)$ that we proved in Lemma \ref{lem:twisted-trace-global}. Therefore we get
\[\mathrm{Tr}(\phi I_\sigma|\Pi)=e(\bfG_r)\mathrm{Tr}(\phi^*|\pi^*)\] 
for any irreducible tempered representation $\pi^*$ of $G^*$ whose base change $\Pi^*$ has a Langlands-Jacquet transfer $\Pi$ (in the sense of Badulsecu), which is an irreducible admissisble representation of $G_r$.
%Jingren: the following question has been addressed.
%{\cg Tom: I do not know which type of LJ-transfer you mean, for a general admissible representation -- are you suggesting using Badulescu's theory for unitary representations...?}

\iffalse
%{\cg Tom: why can we change the representation at $v_0$ here?  I thought the intertwiner is defined separately for each representation.} 
{\cg Tom: I wrote a question a long time ago at this point, which I am not sure has been addressed. What does the Steinberg representation $\widetilde{\Pi}_{v_0}$ have to do with $\pi^*$? and why is the following displayed equation true? I think this can be completed but the information about generalized Steinberg characters needs to be expanded. I am working on this.} 
For this purpose we take $\pi^*$ to be the Steinberg representation of $G^*$, then $\Pi$ is the Steinberg representation of $G_r$. By Lemma~\ref{lem:twisted-St-character}, there exists an intertwining operator $I_{\sigma}'$ on $\Pi$ such that 
\[\mathrm{Tr}(\phi I_{\sigma}'|\Pi)=(-1)^{n-t}\mathrm{Tr}(\phi^*|\pi^*)\]
for any function $\phi\in C_c^\infty(G_r)$ with stable base change transfer $\phi^*\in C_c^\infty(G^*)$. This implies that $(-1)^{m(n-1)}$ and $e(\bfG_r)=(-1)^{n-t}$ differ by an $r$-th root of unit and we are done.
\fi

\subsection{Transfer of the Bernstein center}
We first recall some standard facts about the (stable) Bernstein center from \cite{Ha14}. In what follows, we work over a $p$-adic field $F$, freely using the analogues of the statements above which were proved for the case $F = \bbQ_p$. In this context $F_r$ denotes the unique degree $r$ unramified extension of $F$ in a fixed algebraic closure of $F$. To simplify notation, we will use regular roman letters to denote both algebraic groups and their $F$-points (bearing in mind $G_r$ is also the set of $F$-points for a Weil-restricted group over $F$).\par
Let ${}^LG=\hat{G}\rtimes W_F$ be the Langlands dual group of $G$. A \emph{semisimple $L$-parameter} for $G$ is a continuous homomorphism
$\la:W_F\to{}^LG$ whose image consists of semisimple elements and whose projection to the second factor is the identity map. The set of $\hat{G}$-conjugacy classes of semisimple $L$-parameters has the structure of an infinite disjoint union of algebraic varieties, which we refer to as the \emph{stable Bernstein variety} of $G$ and denote by $\mathfrak{Y}_G^{\mathrm{st}}$. The \emph{stable Bernstein center} $\mathfrak{Z}_G^{\mathrm{st}}$ is defined to be the ring of regular functions on $\mathfrak{Y}_G^{\mathrm{st}}$. In particular, the stable Bernstein center for $G$ is canonically identified with the stable Bernstein center for its quasi-split inner form $G^*$. For any $Z\in\mathfrak{Z}_G^{\mathrm{st}}$, we sometimes denote by $Z^*$ the corresponding element in $\mathfrak{Z}_{G^*}^{\mathrm{st}}$.\par
There is a morphism of algebraic varieties $\mathfrak{Y}_G^{\mathrm{st}}\to\mathfrak{Y}_{G_r}^{\mathrm{st}}$ defined by restricting a semisimple $L$-parameter to the subgroup $W_{F_r}\subset W_F$. This induces the base change homomorphism of stable Bernstein centers:
\begin{equation}\label{eq:BC-center}
b:\mathfrak{Z}_{G_r}^{\mathrm{st}}\to\mathfrak{Z}_G^{\mathrm{st}}
\end{equation}

In our setting, since $G^*$ is a product of Weil restrictions of general linear groups, the local Langlands correspondence is known and there is a quasi-finite map from the stable Bernstein center $\mathfrak{Z}_G^{\mathrm{st}}$ to the Bernstein center $\mathfrak{Z}_G$ which is an isomorphism if $G=G^*$. A good reference for these facts is \cite{Coh18}. In particular, we can view an element $Z\in\mathfrak{Z}_G^{\mathrm{st}}$ as a distribution on the group $G$ acting on any irreducible smooth representation $\pi$ of $G$ by a scalar $Z(\pi)$. In particular, if $\pi^*$ is the Jacquet-Langlands transfer of $\pi$, then $Z^*(\pi^*)=Z(\pi)$. Moreover, for any $Z\in\mathfrak{Z}_{G_r}^{\mathrm{st}}$ and any tempered irreducible representation $\pi^*$ of $G^*$ we have \[Z(BC(\pi^*))=b(Z)(\pi^*)\]
where $BC(\pi^*)$ is the base change of $\pi^*$ whose existence is established in \cite{AC}.\par
%{\cc Jingren: The following result is not used elsewhere, may be deleted.}
%\begin{prop}\label{prop:twisted-transfer-stable-center}
%If $\phi\in C_c^\infty(G_r)$ and $\phi^\dagger\in C_c^\infty(G_r^*)$ have matching twisted orbital integrals, then $Z*\phi$ and $Z^**\phi^\dagger$ have matching twisted orbital integrals for any $Z\in\mathfrak{Z}^{\mathrm{st}}_{G_r}$.
%\end{prop}
%\begin{proof}
%By Proposition \ref{prop:twisted-transfer}, there exists a function $(Z*\phi)^\dagger\in C_c^\infty(G_r^*)$ having matching twisted orbital integrals with $Z*\phi$. By the twisted analogue of Kazhdan's density theorem (cf. \cite[Proposition 2.7]{AC}){\cg Tom: this reference is only for $\sigma$-regular twisted orbital integrals -- maybe in addition give reference to Kottwitz-Rogawski}, it suffices to show that for any $\sigma$-stable irreducible tempered representation $\pi^*$ of $G_r^*$, we have an identity
%\[\mathrm{Tr}((Z*\phi)^\dagger|\pi^*)=\mathrm{Tr}(Z^**\phi^\dagger|\pi^*)\]
%{\cg Tom: there are missing intertwiners above}
%This is an immediate consequence of Theorem \ref{thm:twisted-JL}.
%\end{proof}

\begin{thm}\label{thm:center-stable-base-change}
Let $\phi\in C_c^\infty(G_r)$ and let $\phi^*\in C_c^\infty(G^*)$ be its stable base change transfer. Then for any $Z\in\mathfrak{Z}^{\mathrm{st}}_{G_r}$,  $b(Z)^**\phi^*$ is a stable base change transfer of $Z*\phi$.
\end{thm}
\begin{proof}
Let $\phi^*$ be a stable base change transfer of $\phi$ and let $(Z*\phi)^*$ be a stable base change transfer of $Z*\phi$. By Kazhdan's density theorem it suffices to show that 
\[\mathrm{Tr}(b(Z)^**\phi^*|\pi^*)=\mathrm{Tr}((Z*\phi)^*|\pi^*)\]
for any irreducible tempered representation $\pi^*$ of $G^*$.\par
Fix such a $\pi^*$. Let $\Pi^*$ be the base change lifting of $\pi^*$ to $G_r^*$ as in \cite[Theorem 6.2]{AC}. If $\Pi^*$ does not have a Langlands-Jacquet transfer to $G_r$, then both sides are $0$ by Theorem \ref{thm:twisted-JL}(2).\par
Now suppose that $\Pi^*$ has a Langlands-Jacquet transfer $\Pi$, which is an irreducible admissible representation of $G_r$. By part (1) of Theorem~\ref{thm:twisted-JL} we know that $\Pi$ is $\sigma$-stable and we can choose the intertwining operator $I_\sigma$ so that 
\begin{align*}
\mathrm{Tr}((Z*\phi)^*|\pi^*)&=e(G_r)\mathrm{Tr}((Z*\phi)I_\sigma|\Pi)=e(G_r)Z(\Pi)\mathrm{Tr}(\phi I_\sigma|\Pi)\\
&=e(G_r) Z(\Pi^*)\mathrm{Tr}(\phi I_\sigma |\Pi) =b(Z)^*(\pi^*)\mathrm{Tr}(\phi^*|\pi^*)=\mathrm{Tr}(b(Z)^**\phi^*|\pi^*).
\end{align*}
\end{proof}

\begin{lem}\label{lem:vanishing-criterion}
Let $\phi\in C_c^\infty(G(F_r))$ and let $\phi^*\in C_c^\infty(G^*(F))$ be its stable base change transfer. Then $\phi$ has the vanishing property if and only if $\mathrm{Tr}(\phi^*|\pi^*)=0$ for any irreducible tempered representation $\pi^*$ of $G^*(F)$ that does not have a Langlands-Jacquet transfer to $G(F)$.
\end{lem}
\begin{proof}
We start with some preliminary remarks. Any irreducible tempered representation $\pi$ of $G(F)$ (resp.\,$\pi^*$ of $G^*(F)$) can be written as a normalized induction $i_{P(F)}^{G(F)}(\sigma)$ (resp.\,$i_{P^*(F)}^{G^*(F)}(\sigma^*)$) of an essentially discrete series representation $\sigma$ (resp.\,$\sigma^*$) of $M(F)$ (resp.\,$M^*(F)$) for a Levi factor $M \subset P$ (resp.\,$M^* \subset P^*$); see \cite{Tad90}.  The property that $\pi^*$ does not have a Langlands-Jacquet transfer to $G(F)$ is equivalent to the condition that $M^*$ does not transfer to $G$ (comp.\,\cite[p.\,105]{Bad03}). Indeed, if $M^*$ does transfer to an $F$-Levi $M \subset G$, then $\sigma^*$ has a discrete series transfer $\sigma$ on $M(F)$, in which case $\pi^*$ is the Jacquet-Langlands transfer of the tempered representation $i_{P(F)}^{G(F)}(\sigma)$ for any $F$-parabolic $P$ having $M$ as $F$-Levi factor (see \cite[Prop.\,7.1(b)]{Bad18}).  Conversely, if $M^*$ does not transfer, then $\pi^*$ cannot be the Jacquet-Langlands transfer of a tempered irreducible representation $\pi = i_{Q(F)}^{G(F)}(\sigma)$, where $Q = LU$ is a Levi factorization of an $F$-parabolic and $\sigma$ is a discrete series on $L(F)$.  If $\pi^*$ were the transfer of such a $\pi$, then we would have $\pi^* = i_{Q^*(F)}^{G^*(F)}(\sigma^*)$ where $\sigma^*$ is a discrete series transfer of $\sigma$, and where $Q^* = L^*U^*$ is a transfer of $Q= LU$. By uniqueness of the discrete series inducing data of tempered representations of $G^*$ (\cite[Thm.\,9.7]{Zel80}), we must have $M = L^*$, contradicting our assumption that $M$ does not transfer. 

In a slightly different direction, Badulescu defined an injective map on representation rings
$$
\mathrm{JL}: R(G) \rightarrow R(G^*),
$$
and termed the condition that $\pi^*$ not lie in the image by $\pi^* \in \mathbf{S}_{G^*, G}$. Badulescu also proved that, for $\phi^* \in \mathcal{H}(G^*(F))$, the property that
$$
{\rm Tr}(\phi^* | \pi^*) = 0, \hspace{.5in} \forall \pi^* \in \mathbf{S}_{G^*, G}
$$
is equivalent to the property that
$$
O_{\gamma^*}(\phi^*)=0, \hspace{.5in} \forall \gamma^* \in G^{*,\rm rs}(F)\backslash G(F) 
$$
that is, for all regular semisimple $\gamma^*$ which do not transfer to $G(F)$;
see \cite[Lem.\,3.3]{Bad03}. This statement is clearly related to what we wish to prove, but it is not quite sufficient.

We now give the proof, which is quite analogous to \cite[Lem.\,8.2]{Coh18}.
First suppose that $\phi$ has the vanishing property. Then the orbital integral of $\phi^*$ vanishes on any regular semisimple conjugacy class $\ga^*$ of $G^*$ that does not transfer to $G$. In fact, if such an element $\ga^*$ is not a norm from $G_r$ then $O_{\ga^*}(\phi^*)=0$ by the definition of stable base change transfer; while if $\ga^*=\cN_r\delta$ for some $\delta\in G_r$ then $\cN_r\delta$ is not $G_r$-conjugate to any element of $G$ (since otherwise $\ga^*$ would transfer to $G$), so that $O_{\ga^*}(\phi^*)=TO_{\delta\sigma}(\phi)=0$ by the vanishing property of $\phi$. By the Weyl integration formula, this implies that $\mathrm{Tr}(\phi^*|\pi^*)=0$ for any tempered irreducible representation $\pi^*$ of $G^*$ that does not have a Langlands-Jacquet transfer to $G$. In fact, by the preliminary discussion above, any such representation $\pi^*$ is a representation induced from a Levi subgroup $M^*\subset G^*$ that does not transfer to $G$, and the character of such a representation has (its regular-semisimple) support consisting of regular semisimple elements $\ga^*$ that are conjugate to elements in $M^*$ and hence do not transfer to $G(F)$. Here we use the observation in \cite[p.\,105]{Bad03} that if $M^*$ does not transfer to $G(F)$, then no element of $G^{*, \rm rs}(F) \cap M^*$ can transfer. (Alternatively, one can use the argument at the end of the proof of Lemma \ref{Frob-transfer_lem}.) \par 
Next we prove the implication in the reverse direction. Assume that $\mathrm{Tr}(\phi^*|\pi^*)=0$ for any representation $\pi^*$ of $G^*$ that does not have a Langlands-Jacquet transfer to $G$. Suppose $\delta\in G(F_r)$ is a $\sigma$-semisimple element such that $N_r\delta\in G(F_r)$ is not conjugate to any element in $G(F)$. There is a semisimple element $\ga^*\in G^*(F)$ that has the same characteristic polynomial as $N_r\delta$. Then the conjugacy class of $\ga^*$ does not transfer to $G(F)$. Note that $\ga^*$ is an elliptic element in an $F$-Levi subgroup $M^*$ of $G^*$, and $M^*$ does not transfer to an $F$-Levi subgroup in $G$ (because elliptic tori transfer to any inner form of $M^*$, see for example \cite[\S10]{Ko-ellsing}). \par
Let $P^*$ be an $F$-parabolic subgroup of $G^*$ with $F$-Levi factor $M^*$.
Any irreducible tempered representation $\pi^*$ of $M^*$ is induced from a discrete series representation of a Levi subgroup $L^*$ of $M^*$. By induction in stages we see that $i_{P^*(F)}^{G^*(F)}\pi^*$ is tempered, and it is irreducible by a result of Bernstein-Zelevinsky. It does not have a Langlands-Jacquet transfer to $G(F)$: this is because $M^*$ does not transfer to $G(F)$; use \cite[Prop.\,7.1(a)]{Bad18}. Hence 
$$\mathrm{Tr}((\phi^*)^{P^*})|\pi^*)=\mathrm{Tr}(\phi^*|i_{P^*(F)}^{G^*(F)}\pi^*)=0$$ by assumption and the character descent formula. This implies that $O^{M^*}_{\ga^*}((\phi^*)^{P^*}))=0$ by Kazhdan's density theorem and hence $O_{\ga^*}(\phi^*)=0$ by the descent formula for orbital integrals. Then we get
\[TO_{\delta\sigma}(\phi)= \pm O_{\ga^*}(\phi^*)=0\]
from which we conclude that $\phi$ has the vanishing property.
\end{proof}

\begin{defn}\label{gen_BC_defn}
Let $\phi\in C_c^\infty(G_r)$. A function $f\in C_c^\infty(G)$ is called a \emph{base change transfer of $\phi$} if for any semisimple element $\ga\in G$, we have $ e(G_\ga)O_\ga(f)=e(G_{\delta \sigma})TO_{\delta\sigma}(\phi)$ if $\ga$ is $G_r$-conjugate to the naive norm $N_r\delta$ of some $\sigma$-semisimple element $\delta\in G_r$, and $O_\ga(f) = 0$ otherwise.

\end{defn}

\begin{lem}\label{lem:base-change-exist}
    If $\phi\in C_c^\infty(G_r)$ has the vanishing property, then it has a base change transfer $f\in C_c^\infty(G)$. 
\end{lem}
\begin{proof}
Let $\phi^*\in C_c^\infty(G^*)$ be a stable base change transfer of $\phi$. Let $\ga^*\in G^*$ be a semisimple element that is not stably conjugate to any element of $G$. If $\ga^*$ is not the norm of any element in $G_r$, then by definition $O_{\ga^*}(\phi^*)=0$. If $\ga^*=\cN_r\delta$ for some $\delta\in G_r$, then by the vanishing property we have $O_{\ga^*}(\phi^*)=\pm TO_{\delta\sigma}(\phi)=0$. Thus $\phi^*$ has a Jacquet-Langlands transfer $f\in C_c^\infty(G)$ by \cite{DKV}. Then $f$ is a base change transfer of $\phi$ by definition.
\end{proof}

\begin{prop}\label{prop:vanishing-property}
If $\phi\in C_c^\infty(G_r)$ has the vanishing property \textup{(}in the sense of Definition~\ref{defn:vanishing-property}\textup{)}, then so does $Z*\phi$ for any $Z\in\mathfrak{Z}^{\mathrm{st}}_{G_r}$.
\end{prop}
\begin{proof}
Let $\phi^*$ be the stable base change transfer of $\phi$ and let $(Z*\phi)^*\in C_c^\infty(G(F))$ be the stable base change transfer of $Z*\phi$. By Lemma~\ref{lem:vanishing-criterion}, it suffices to show that \[\mathrm{Tr}((Z*\phi)^*|\pi^*)=0\]
for any irreducible tempered representation $\pi^*$ of $G^*$ that does not have a Jacquet-Langlands transfer to $G$.\par
Fix such a $\pi^*$. Let $\Pi^*$ be the base change of $\pi^*$ to $G_r^*$ as in \cite[Theorem 6.2]{AC}. If $\Pi^*$ does not have a Langlands-Jacquet transfer to $G_r$, then we are done by Theorem \ref{thm:twisted-JL}.\par 
Now suppose that $\Pi^*$ has a Langlands-Jacquet transfer $\Pi$, which is an irreducible tempered representation of $G_r$. By part (1) of Theorem~\ref{thm:twisted-JL} we know that $\Pi$ is $\sigma$-stable and we choose the intertwining operator $I_\sigma$ as in \emph{loc. cit} so that
\[\mathrm{Tr}((Z*\phi)^*|\pi^*)=\mathrm{Tr}((Z*\phi)I_\sigma|\Pi).\]
This vanishes since the right hand side is a scalar multiple of
\[\mathrm{Tr}(\phi I_\sigma|\Pi)=e(G_r)\mathrm{Tr}(\phi^*|\pi^*)=0.\]
Here the first equality follows from Theorem~\ref{thm:twisted-JL} and the second equality follows from Lemma~\ref{lem:vanishing-criterion}.
\end{proof}
We mention the following Corollary.
\begin{cor}
    Let $\cG$ be a smooth group scheme over $\cO$ with geometrically connected fibers such that the generic fiber of $\cG$ is isomorphic to $G$. Then $Z*1_{\cG(\cO_r)}$ has the vanishing property for any $Z\in\mathfrak{Z}^{\mathrm{st}}_{G_r}$.
\end{cor}
\begin{proof}
    This follows from Proposition \ref{prop:vanishing-property} and a result in \cite{Ko-BC}. More precisely the group $\cG(\Breve{F})$, where $\Breve{F}$ is the completion of the maximal unramified extension of $F$, satisfies the properties (a),(b), (c) on page 240 of \emph{loc.\,cit.} Hence by the discussion after the Corollary on page 244 of \emph{loc.\,cit.} we get that $1_{\cG(\cO_r)}$ has the vanishing property.
\end{proof}
\begin{thm}\label{thm:center-base-change}
If $\phi\in C_c^\infty(G_r)$ satisfies the vanishing property and has a base change transfer $f\in C_c^\infty(G)$, then $b(Z)*f$ is a base change transfer of $Z*\phi$ for any $Z\in\mathfrak{Z}^{\mathrm{st}}_{G_r}$. 
\end{thm}
\begin{proof}
Let $\phi^*\in C_c^\infty(G^*)$ be the stable base change of $\phi$. Then it is a Jacquet-Langlands transfer of $f$,  as follows from the definitions. Thus $b(Z)^**\phi^*$ is a Jacquet-Langlands transfer of $b(Z)*f$, by the case $r=1$ of Theorem \ref{thm:center-stable-base-change}. On the other hand, $b(Z)^**\phi^*$ is a stable base change transfer of $Z*\phi$ by Theorem \ref{thm:center-stable-base-change}.  Keeping in mind that $Z * \phi$ satisfies the vanishing property by Proposition \ref{prop:vanishing-property}, this implies that $b(Z)*f$ is a base change transfer of $Z*\phi$. 
\end{proof}

\section{$p$-divisible groups with EL structure and test functions}\label{sec:p-div}
In this section we follow \cite[\S3,\S4]{Sch-deformation} and generalize some results there to our situation. 

\subsection{EL data}\label{sec:EL-data}
We start with a list of data in the local setting. Let $B$ be a semisimple $\bbQ_p$-algebra and $V$ a finitely generated left $B$-module. We get a semisimple $\bbQ_p$-algebra $C:=\mathrm{End}_B(V)$ and an algebraic group $G$ over $\bbQ_p$ whose points in any $\bbQ_p$-algebra $R$ is defined by $G(R)=(C\otimes_{\bbQ_p}R)^\times$, the group of units of the algebra $C\otimes_{\bbQ_p}R$.\par 
Let $\bar{\mu}$ be a $G(\bar{\bbQ}_p)$-conjugacy class of cocharacters $\mu:\bbG_m\to G_{\bar{\bbQ}_p}$. The local reflex field $E\subset\bar{\bbQ}_p$ is defined as the field of definition of the conjugacy class $\bar{\mu}$. Let $\cO_E\subset E$ be the ring of integers and $\kappa_E$ the residue field. We assume that $\bar{\mu}$ is \emph{minuscule}. In other words, after choosing a representative $\mu$ of $\bar{\mu}$, only weight $0$ or $1$ occurs in the corresponding weight decomposition of $V_{\bar{\bbQ}_p}$. Fix such a representative $\mu$ and write $V_{\bar{\bbQ}_p}=V_1\oplus V_2$ where $\mu(t)$ acts by $t$ (resp. $1$) on $V_1$ (resp. $V_2$) for each $t\in\bar{\bbQ}_p$.\par 

We also fix integral data consisting of a maximal $\bbZ_p$-order $\cO_B\subset B$ and a periodic chain $\cL=\{\Lambda_i,i\in\bbZ\}$ of $\cO_B$-stable $\bbZ_p$-lattices in $V$. These allow us to define a $\bbZ_p$-model $\cG_\cL$ of $G$ whose points in any $\bbZ_p$-algebra $R$ is $\cG_\cL(R):=\mathrm{Aut}_{\cO_B}\cL_R$. Here $\cL_R:=\{\Lambda\otimes_{\bbZ_p}R, \Lambda\in\cL\}$ and $\mathrm{Aut}_{\cO_B}\cL_R$ consists of a family of automorphisms of the $\cO_B\otimes_{\bbZ_p}R$-modules $\Lambda\otimes_{\bbZ_p}R$ for each $\Lambda\in\cL$ that preserves the natural morphisms among the lattices in $\cL$. In particular, $G$ is a product of Weil restrictions of inner forms of general linear groups and $\bfP_\cL:=\cG_\cL(\bbZ_p)$ is a parahoric subgroup of $G(\bbQ_p)$. More generally for each integer $r\ge1$, we denote $\bfP_{\cL,r}:=\cG_\cL(\bbZ_{p^r})$ the corresponding parahoric subgroup of $G(\bbQ_{p^r})$. We let $\cD=(B,V,\cO_B,\cL,\bar{\mu})$ denote this list of data. 

\subsection{Deformation spaces of $p$-divisible groups}
\begin{defn}
    For any scheme $S$ on which $p$ is locally nilpotent, a $p$-\emph{divisible $\cO_B$-module} over $S$ is a pair $(H,i)$ consisting of a $p$-divisible group $H$ over $S$ together with a $\bbZ_p$-algebra homomorphism $i:\cO_B\to\mathrm{End}(H)$. An isogeny or quasi-isogeny between $p$-divisible $\cO_B$-modules is required to commute with the $\cO_B$-action. 
\end{defn}

\begin{defn}
Let $S$ be an $\cO_E$-scheme on which $p$ is locally nilpotent. A \emph{$p$-divisible group with $\cD$-structure} $H_\bullet$ over $S$ consists of the data of
\begin{itemize}
    \item A collection of $p$-divisible $\cO_B$-modules $H_\Lambda$ for each $\Lambda\in\cL$. Denote the Lie algebra of the universal vector extension of $H_\La$ by $M_\Lambda$;
    \item For each inclusion $\Lambda\subset\Lambda'$ of lattices in $\cL$, an isogeny $\rho_{\Lambda',\Lambda}:H_{\Lambda}\to H_{\Lambda'}$
\end{itemize}
satisfying the following conditions
\begin{enumerate}
    \item For any triple of lattices $\Lambda\subset\Lambda'\subset\La''$ we have $\rho_{\La'',\La'}\circ\rho_{\La',\La}=\rho_{\La'',\La}$;
    \item For any inclusion of lattices $\La\subset\La'$, the cokernel of the induced map $M_\La\to M_{\La'}$ is locally on $S$ isomorphic to $(\La'/\La)\otimes_{\bbZ_p}\cO_S$ as $\cO_B\otimes_{\bbZ_p}\cO_S$-modules;
    \item For each $\Lambda$, we have an equality of polynomial functions on $\cO_B$:
    \[\mathrm{det}_{\cO_S}(b;\mathrm{Lie}(H_\La))=\det(b;V_1).\]
    \item For any $b\in B^\times$ that normalizes $\cO_B$ and any $\La\in\cL$, a periodicity isomorphism $\theta_b:H_\Lambda^b\to H_{b\La}$. Here $H_\Lambda^b$ is the $p$-divisible $\cO_B$-module with the same underlying $p$-divisible group as $H_\lambda$ but with the action map $i_\lambda:\cO_B\to\mathrm{End}(H_\La)$ replaced by the composition of $i_\lambda$ with the automorphism of $\cO_B$ defined by $x\mapsto bxb^{-1}$. 
\end{enumerate}
\end{defn}

\begin{rem}
    By the discussion in \cite[\S3.23 c)]{RZ}, condition (3) implies that for each $\La\in\cL$, $M_\La$ is locally on $S$ isomorphic to $\La\otimes_{\bbZ_p}\cO_S$ as an $\cO_B\otimes_{\bbZ_p}\cO_S$-module.%{\cg Tom: we need to check this with our sign conventions. We might need to consider the dual of $M_\Lambda$...}
\end{rem}

Let $\kappa$ be a perfect field of characteristic $p$, viewed as an $\cO_E$-algebra via a homomorphism $\cO_E\to\kappa$. Let $H_\bullet$ be a $p$-divisible group with $\cD$-structure over $\kappa$.

\begin{defn}
The deformation space $\cX_{H_\bullet}$ of $H_\bullet$ is the functor that associates to any Artin local $\cO_E$-algebra $R$ with residue field $\kappa$ the set of isomorphism classes of $p$-divisible groups with $\cD$-structure $\tilde{H}_\bullet$ over $\spec(R)$, together with an isomorphism  $\tilde{H}_\bullet\otimes_R\kappa\xrightarrow{\cong} H_\bullet$, i.e. an isomorphism of $p$-divisible $\cO_B$-modules $\tilde{H}_\La\otimes_R\kappa\xrightarrow{\cong} H_\La$ for all $\La\in\cL$ that commute with the isogenies corresponding to the inclusions among various $\La$'s.
\end{defn}

\begin{thm}
    The functor $\cX_{H_\bullet}$ is representable by a complete noetherian local $\cO_E$-algebra $R_{H_\bullet}$ with residue field $\kappa$. 
\end{thm}
\begin{proof}
    By \cite[Theorem 3.6]{Sch-deformation}, the deformation functor of each single $H_\La$ is representable by a complete noetherian local $\cO_E$-algebra $R_{H_\La}$ with residue field $\kappa$. Then by the periodicity condition (4), the functor $\cX_{H_\bullet}$ is representable by a quotient of the completed tensor product of finitely many such $\cO_E$-algebras.
\end{proof}

Let $X_{H_\bullet}$ be the rigid generic fiber of $\cX_{H_\bullet}$. Then 
$X_{H_\bullet}$ is a rigid analytic space over the discrete valued field $W_{\cO_E}(\kappa)[\frac{1}{p}]$, where $W_{\cO_E}(\kappa):=W(\kappa)\otimes_{\bbZ_p}\cO_E$. Moreover, it is smooth by the same reasoning in \cite[Theorem 3.8]{Sch-deformation}. For each compact open subgroup $K\subset\cG_\cL(\bbZ_p)$, we have an \'etale cover $X_{H_\bullet,K}$ of $X_{H_\bullet}$ parametrizing level-$K$ structures of the universal $p$-adic Tate module.\par 
For each lattice $\La\in\cL$, the covariant Dieudonn\'e module $M_\La$ is an $\cO_B\otimes_{\bbZ_p}W(\kappa)$-module that is free of finite rank as $W(\kappa)$-module. Moreover, the $B\otimes_{\bbQ_p}W(\kappa)[\frac{1}{p}]$-modules  $M_\La[\frac{1}{p}]$ for various $\La$'s are canonically identified by the isogenies among them and we denote this common object by $V_{H_\bullet}$. Now suppose $X_{H_\bullet}\ne\varnothing$; then $V_{H_\bullet}$ is isomorphic to $V\otimes_{\bbZ_p}W(\kappa)$ as $B$-module. We can choose an isomorphism $V_{H_\bullet}\cong V\otimes_{\bbZ_p}W(\kappa)$ such that the $\cO_B$-module $M_\La$ is mapped isomorphically to $\La\otimes_{\bbZ_p}W(\kappa)$ for each $\La\in\cL$. Under this identification, the Frobenius operator $F$ on $V_{H_\bullet}$ has the form $F=p\delta\sigma$ where $\sigma$ denotes the Frobenius automorphism of $W(\kappa)$ and $\delta\in G(W(\kappa)[\frac{1}{p}])$ satisfies $p\La\subset p\delta\La\subset\La$ for all $\La\in\cL$. 
A different choice of isomorphism changes $\delta$ by a $\sigma$-conjugate under $\cG_\cL(W(\kappa))$.\par 
To state a further condition satisfied by $\delta$ we introduce more notations. Let $B(G)$ be the set of $\sigma$-conjugacy classes in $G(W(\bar\kappa)[\frac{1}{p}])$. In \cite[\S6]{Ko-lambda} Kottwitz constructs a map $\kappa_G:B(G)\to X^*(Z(\hat{G})^{\Gamma})$ where $\Gamma=\mathrm{Gal}(\bar{\bbQ}_p/\bbQ_p)$. We choose a representative $\mu$ of the conjugacy class $\bar{\mu}$ and view it as a character of the dual group. Then the restriction $\mu_1\in X^*(Z(\hat{G})^{\Gamma})$ of $\mu$ depends only on the conjugacy class $\bar{\mu}$. 
\begin{prop}\label{prop:delta-condition}
    For any perfect field $\kappa$ of characteristic $p$ equipped with a ring homomorphism $\cO_E\to\kappa$, the association $H_\bullet\mapsto\delta\in G(W(\kappa)[\frac{1}{p}])$ described above defines an injection from the set of $p$-divisible groups with $\cD$-structures $H_\bullet$ over $\kappa$ such that $X_{H_\bullet}\ne\varnothing$ to the set of $\cG_\cL(W(\kappa))$-$\sigma$-conjugacy classes $\delta\in G(W(\kappa)[\frac{1}{p}])$ such that $\kappa_G(p\delta)=-\mu_1$ and $p\La_{W(\kappa)}\subset p\delta\La_{W(\kappa)}\subset\La_{W(\kappa)}$ for all $\La\in\cL$, where $\La_{W(\kappa)}:=\La\otimes_{\bbZ_p}W(\kappa)$.
\end{prop}
\begin{proof}
The injectivity follows from classical Dieudonn\'e theory. The identity $\kappa_G(p\delta)=-\mu_1$ follows from \cite[Proposition 1.20]{RZ}, see also \cite[Proposition 3.9]{Sch-deformation}.
\end{proof}

\subsection{Test functions}
Fix an integer $j\ge1$ and an element $\tau\in\mathrm{Frob}^jI_E\subset W_E$ in the Weil group of $E$ (here $I_E$ is the inertia subgroup of $W_E$ and $\mathrm{Frob}\in W_E$ is a geometric Frobenius element). Set $r=j[\kappa_E:\bbF_p]$ and view $\bbF_{p^r}$ as a degree $j$ extension of $\kappa_E$. Let $h\in C_c^\infty(\cG_\cL(\bbZ_p))$ be a $\bbQ$-valued function. Our goal is to define a test function $\phi_{\tau, h}$ in $C_c^\infty(G(\bbQ_{p^r}))$ that enters the Lefschetz trace formula for the special fiber at $p$ of the relevant Shimura varieties, and also a test function $\phi^{(r)}_h$ that is relevant for the description of semi-simple local factor of the Shimura varieties.\par
Let $H_\bullet$ be a $p$-divisible group with $\cD$-structure over a perfect field $\kappa$ of characteristic $p$ equipped with a ring homomorphism $\cO_E\to\kappa$. Recall that for any compact open subgroup $K\subset\cG_\cL(\bbZ_p)$, we defined a rigid analytic space $X_{H_\bullet,K}$ over the field $k:=W_{\cO_E}(\kappa)[\frac{1}{p}]$. 
Following \cite[Definition 3.11]{Sch-deformation}, we say that $H_\bullet$ \emph{has controlled cohomology} if $X_{H_\bullet,K}$ has controlled $\ell$-adic \'etale cohomology for any normal pro-$p$ open subgroup $K\subset\cG(\bbZ_p)$ and any prime $\ell\ne p$. We refer to \cite[\S2]{Sch-deformation} for the notion of rigid analytic spaces with controlled cohomology.

\begin{defn}
Let $\delta\in G(\bbQ_{p^r})$. If $\delta$ is associated to some $p$-divisible group with $\cD$-structure $H_\bullet$ over $\bbF_{p^r}$ as in Proposition~\ref{prop:delta-condition} such that $H_\bullet$ has controlled cohomology, we define
\[\phi_{\tau,h}(\delta)=\mathrm{tr}(\tau\times h| H^*(X_{H_\bullet,K}\otimes_k\hat{\bar{k}},\bbQ_\ell))\]
and
\[\phi_{h}^{(r)}(\delta)=\mathrm{tr}^{\mathrm{ss}}(\mathrm{Frob}^j\times h| H^*(X_{H_\bullet,K}\otimes_k\hat{\bar{k}},\bbQ_\ell))\]
where $K\subset\cG_\cL(\bbZ_p)$ is any normal pro-$p$ open compact subgroup such that $h$ is $K$-biinvariant and $\ell\ne p$. For any other $\delta\in G(\bbQ_{p^r})$, define $\phi_{\tau,h}(\delta)$ and $\phi_{h}^{(r)}(\delta)$ to be $0$. 
\end{defn}
Here $\mathrm{tr}^{ss}$ denotes semi-simple trace, for whose definition and properties we refer to \cite[\S3]{HN02}. In particular, it follows by definition that for all $\delta\in G(\bbQ_{p^r})$ we have

\begin{equation}\label{eq:test-function-integral}
    \phi_{h}^{(r)}(\delta)=\int_{\mathrm{Frob}^jI_E}\phi_{\tau,h}(\delta)d\tau
\end{equation}
where $d\tau$ is the translate of the Haar measure on $I_E$ having total volume $1$; comp.,\,\cite[$\S8.2$]{HaRi20}. Since the alternating sum $H^*(X_{H_\bullet,K}\otimes_k\hat{\bar{k}},\bbQ_\ell))$ is a virtual continuous $\ell$-adic representation of the Weil group $W_E$, the action of the inertia $I_E$ on its semi-simplification factors through a finite quotient and the above integral reduces to a finite sum multiplied by a nonzero rational number. 
\begin{prop}
    The functions $\phi_{\tau,h}$ and $\phi^{(r)}_h$ are well-defined locally constant $\bbQ$-valued functions with compact support on $G(\bbQ_{p^r})$ and they are independent of $\ell$. Moreover, $\phi_{\tau,h}$ and $\phi^{(r)}_h$ are $\cG_\cL(\bbZ_{p^r})$-$\sigma$-conjugation invariant and their supports are contained in the set of elements $\delta\in G(\bbQ_{p^r})$ such that $\kappa_G(p\delta)=-\mu_1$ and $p\La_r\subset p\delta\La_r\subset\La_r$ for all $\La\in\cL$, where $\La_r:=\La\otimes_{\bbZ_p}\bbZ_{p^r}$. 
\end{prop}
\begin{proof}
    By \eqref{eq:test-function-integral}, it suffices to prove the results for $\phi_{\tau,h}$. The same proof in \cite[Proposition 4.2]{Sch-deformation} shows that $\phi_{\tau,h}$ is well-defined, has value in $\bbQ$ and independent of $\ell$. The local constancy is proved in the same way as \cite[Proposition 4.3]{Sch-deformation}. The second statement follows from Proposition~\ref{prop:delta-condition}. 
\end{proof}

\begin{prop}\label{prop:independence-of-EL-data}
The function $\phi_{\tau,h}$ depends only on the data $(\cG_\cL, \Bar{\mu}, \tau,h)$, and is independent of the choice of EL-data $\cD$ giving rise to them.
\end{prop}
\begin{proof}
We may assume that $B$ is a simple algebra whose center $F$ is a finite extension of $\bbQ_p$. Then we can choose an isomorphism $B\cong M_m(D)$ sending $\cO_B$ to $M_m(\cO_D)$ where $D$ is a division algebra and $\cO_D$ is the unique maximal order in $D$. Let $W=D^m$ be the unique simple left $B$-module. There is an isomorphism $V\cong W^n$ of $B$-modules which induces isomorphisms $C\cong M_n(D^{op})$ and  $G\cong\mathrm{Res}_{F/\bbQ_p}\mathrm{GL}_n(D^{op})$. Let $\epsilon\in\cO_B$ be the idempotent corresponding to the matrix whose $(1,1)$ entry equals to $1$ and all other entries equal to $0$. Let $B'=D$ and $V'=\epsilon V\cong D^n$ viewed as a left $B'$-module. Then the group of $B'$-automorphisms of $V'$ is canonically isomorphic to $G$. Let $\cO_{B'}=\cO_D$ be the unique maximal order and $\cL'=\epsilon\cL:=\{\epsilon\La,\La\in\cL\}$ be the $\cO_{B'}$-lattice chain in $V'$ corresponding to $\cL$. Then the $\cO$-group scheme $\cG_{\cL'}$ constructed from the data $(\cO_{B'},\cL')$ is canonically isomorphic to $\cG_{\cL}$. Let $\cD':=(B',V',\cO_{B'},\cL',\bar{\mu})$. Then there is an equivalence between the category of $p$-divisible groups with $\cD$-structure and those with $\cD'$-structures. Indeed, for any $H_\bullet$ in the former category, $H_\bullet':=\epsilon H_\bullet$ is an object in the latter category. Conversely, for any $H'_\bullet$ in the latter category, $H_\bullet:=(H'_\bullet)^m$ defines an object in the former category. In particular, corresponding objects defined over a perfect field $\kappa$ of characteristic $p$ are parametrized by the same invariant $\delta\in G(W(\kappa)[\frac{1}{p}])$, defined up to $\cG_\cL(W(\kappa))$-$\sigma$-conjugacy and have isomorphic deformation spaces. Therefore the function $\phi'_{\tau,h}$ defined from the data $\cD'$ is the same as the function $\phi_{\tau,h}$ and the claim follows. 
\end{proof}

\subsection{Statement of the main local theorem}\label{sec:main-local-thm-statements}
Choose a representative $\mu$ of the conjugacy class $\bar{\mu}$ of cocharacters of $G_{\bar{\bbQ}_p}$ and view it as a character of the standard maximal torus $\hat{T}$ in $\hat{G}$. Let $r_{-\mu}:\Hat{G}\to\mathrm{GL}(V_{-\mu})$ be the irreducible algebraic representation of $\hat{G}$ with extreme $\hat{T}$-weight $-\mu$. In other words, the extreme weight of $r_{-\mu}$ is the dominant representative in the Weyl group orbit of $-\mu$. By \cite[Lemma 2.1.2]{Ko-twisted}, for any choice of $W_E$-invariant splitting of $\hat{G}$, $r_{-\mu}$ extends uniquely to a representation
\[r_{-\mu}:{}^{L}G_E:=\hat{G}\rtimes W_E\to\mathrm{GL}(V_{-\mu})\] 
such that $W_E$ acts as the identity on the corresponding highest weight space.\par   
Fix $\tau\in\mathrm{Frob}_E^jI_E\subset W_E$ and set $r:=j[\kappa_E:\bbF_p]$ as before. 
\begin{defn}\label{def:z_{tau,mu}}
Let $z_{\tau,-\mu}$ be the element in the stable Bernstein center of $G$ whose value at any semisimple Langlands parameter $\varphi:W_{\bbQ_p}\to{}^{L}G=\hat{G}\rtimes\mathrm{Gal}(\bar{\bbQ}_p/\bbQ_p)$ is given by
\[\tr(\tau|(r_{-\mu}\circ\varphi|_{W_E})|\cdot|_{E}^{-\langle\rho,\mu\rangle})\]
where $\rho$ is half the sum of all positive roots of $G$. By \cite[\S5.3]{Ha14}\footnote{Due to printing errors in the published version, all citations for this article refer to the online version, arXiv:1304.6293.}, this defines a regular function on the stable Bernstein variety. 
\end{defn}
Next we define the semisimple version of the element $z_{\tau,-\mu}$. Let $E_j\subset\bar{\bbQ}_p$ be the degree $j$ unramified extension of $E$. Then we have $\tau\in W_{E_j}\subset W_E$. After identifying $\bbQ_{p^r}$ with the maximal unramified extension of $\bbQ_p$ in $E_j$, we get a natural inclusion $W_{E_j}\subset W_{\bbQ_{p^r}}$ which allows us to view $\tau$ also as an element in $W_{\bbQ_{p^r}}$. By \cite[\S5.4]{Ha14} there is an element $z_{\tau,-\mu}^{(r)}$ in the stable Bernstein center of $G_{\bbQ_{p^r}}$ such that $b(z_{\tau,-\mu}^{(r)})=z_{\tau,-\mu}$. In fact, for any irreducible smooth representation $\Pi$ of $G(\bbQ_{p^r})$ with semisimple $L$-parameter 
\[\varphi_{\Pi}:W_{\bbQ_{p^r}}\to\hat{G}\rtimes W_{\bbQ_{p^r}},\] 
the element $z_{\tau,-\mu}^{(r)}$ acts on $\Pi$ by the scalar
\[\tr(\tau|(r_{-\mu}\circ\varphi_{\Pi}|_{W_{E_j}})|\cdot|_{E}^{-\langle\rho,\mu\rangle}).\]
We also define a closely related element $z_{-\mu}^{(r)}$ in the stable Bernstein center of $G(\bbQ_{p^r})$ that acts on any irreducible smooth representation $\Pi$ of $G(\bbQ_{p^r})$ with semisimple $L$-parameter $\varphi_\Pi$ by the scalar
\[\mathrm{tr}(\mathrm{Frob}_E^j|(r_{-\mu}\circ\varphi_{\Pi}|_{W_{E_j}})^{I_E}|\cdot |_E^{-\langle\rho,\mu\rangle}).\]
Then we have
\[z_{-\mu}^{(r)}=\int_{\mathrm{Frob}_E^jI_E} z_{\tau,-\mu}^{(r)}d\tau\]
where $d\tau$ is the Haar measure of total volume $1$ and the integral is actually a finite sum. When $r=1$, we also simply denote $z_{-\mu}:=z_{-\mu}^{(1)}$.\par 

\begin{rem}
Above we are implicitly using the ring homomorphism from the stable Bernstein center to the usual Bernstein center; see \cite[$\S5.5$]{Ha14}. In order to make sense of the Bernstein center (which requires a notion of normalized parabolic induction) and also in order to make sense of the ring homomorphism, we must fix a choice of $\sqrt{p} \in \bar{\bbQ}_\ell$.  We use the same choice to define the symbol $|\cdot|_{E}^{-\langle \rho, \mu \rangle}$ appearing above.  In this way we see that $z_{\tau, -\mu}$ and $z^{(r)}_{-\mu}$ are independent of the choice of $\sqrt{p}$. See \cite[Thm.\,7.15]{HaRi21}.
\end{rem}

Now we can state our main local results.
\begin{thm}\label{thm:vanishing-property}
The functions $\phi_{\tau,h}$ and $\phi_{h}^{(r)}$ have the vanishing property in the sense of Definition~\ref{defn:vanishing-property}.
\end{thm}
As a consequence, by Lemma~\ref{lem:base-change-exist}, both functions have base change transfers to $C_c^\infty(G(\bbQ_p))$. The next result provides a particular base change transfer with specified spectral information.

\begin{thm}\label{thm:main-local}
    For any $h\in C_c^\infty(\cG_\cL(\bbZ_p))$, the function $z_{\tau,-\mu}*h\in C_c^\infty(G^*(\bbQ_p))$ is a base change transfer of $\phi_{\tau,h}\in C_c^\infty (G(\bbQ_{p^r}))$. Similarly, $z_{-\mu}*h$ is a base change transfer of $\phi_h^{(r)}$.
\end{thm}

These results will be proved by global methods in \S\ref{sec:proof-local-thm}. Here we record the following consequence that will not be used in what follows but is of independent interest. 
\begin{cor}\label{cor:TO-equality}
Suppose that $h\in C_c^\infty(\cG_\cL(\bbZ_p))$ is a base change transfer of a function $\tilde{h}\in C_c^\infty(G(\bbQ_{p^r}))$ that has the vanishing property. Then for any $\delta\in G(\bbQ_{p^r})$ such that $N_r\delta$ is semisimple, we have $TO_{\delta\sigma}(\phi_{\tau,h})=TO_{\delta\sigma}(z_{\tau,-\mu}^{(r)}*\tilde{h})$ and $TO_{\delta\sigma}(\phi_{h}^{(r)})=TO_{\delta\sigma}(z_{-\mu}^{(r)}*\tilde{h})$.
\end{cor}
\begin{proof}
By Theorem \ref{thm:vanishing-property} and Proposition \ref{prop:vanishing-property}, both functions $\phi_{\tau,h}$ and $z_{\tau,-\mu}^{(r)}*\tilde{h}$ have the vanishing property. By Theorem \ref{thm:main-local} and Theorem \ref{thm:center-base-change}, both functions have base change transfer $z_{\tau,-\mu}*h$, in the sense of Lemma \ref{lem:base-change-exist}. Thus the two functions have the same twisted orbital integrals. The identity for $\phi_{h}^{(r)}$ and $z_{-\mu}^{(r)}*\tilde{h}$ is proved similarly.
\end{proof}

\section{Simple Shimura varieties of Kottwitz type}\label{sec:shimura}
In this section we prove our main global result (Theorem \ref{thm:cohomology-isom}) on the description of the cohomology of the simple Shimura varieties studied by Kottwitz in \cite{Ko-simple}. Our result generalizes the main result from \emph{loc. cit.} to many cases of bad reduction. 
\subsection{Basic setup}\label{sec:global-setup}
First we recall the basic setting from \cite{Ko-simple}, cf. also \cite[\S5.2]{Ha05}. 
\subsubsection{Shimura data}\label{sec:Shimura-data}
Let $\bbF$ be a CM field with totally real subfield $\bbF_0$. Denote by $c\in\mathrm{Gal}(\bbF/\bbF_0)$ the complex conjugation. Let $\bbD$ be a central division algebra of dimention $n^2$ over $\bbF$ and $*$ an anti-involution of the second kind on $\bbD$. Let $\bbG$ be the $\bbQ$-group whose points in any commutative $\bbQ$-algebra $R$ are given by
\[\bbG(R)=\{x\in\bbD\otimes_\bbQ R | xx^*\in R^\times\}\]
The map  $x\mapsto xx^*$ defines the similitude homomorphism $c:\bbG\to\bbG_m$ whose kernel $\bbG_0$ is an inner form of (the restriction of scalars of) a unitary group associated to the quadratic extension $\bbF/\bbF_0$.\par 
Let $h_0:\bbC\to\bbD\otimes_\bbQ\bbR$ be an $\bbR$-algebra homomorphism such that $h_0(\bar{z})=h_0(z)^*$ for all $z\in\bbC$. Assume moreover that $x\mapsto h_0(i)^{-1}x^*h_0(i)$ defines a positive involution on $\bbD\otimes_\bbQ\bbR$. Let $h:\mathrm{Res}_{\bbC/\bbR}\bbG_m\to\bbG_{\mathbb R}$ be the \emph{inverse} of the restriction of $h_0$ to $\bbC^\times$ and let $X_\infty:=\bbG(\bbR)/K_\infty$ be the conjugacy class of $h$ in $\bbG(\bbR)$, where $K_\infty$ is the centralizer of $h$ in $\bbG(\bbR)$. Fix an isomorphism $\bbC\otimes_\bbR\bbC\cong\bbC\times\bbC$ where the first (resp. second) copy of $\bbC$ on the right hand side corresponds to the identity (resp. complex conjugation). Restricting $h_{\bbC}$ to the first factor, we obtain a cocharacter $\mu:\bbG_{m,\bbC}\to\bbG_\bbC$.
Then the pair $(\bbG,X_\infty)$ is a Shimura datum. Let $\bbE$ be its reflex field, i.e. the field of definition of the $\bbG(\bbC)$-conjugacy class of $\mu$. Then we get a tower of Shimura varieties $\mathrm{Sh}_K$ for neat compact open subgroups $K\subset\bbG(\bbA_f)$, each of which is a projective smooth algebraic variety over $\bbE$.

\subsubsection{PEL data}\label{sec:PEL-data}
The Shimura variety $\Sh_K$ (more precisely a finite disjoint union of copies of it) can be realized as a moduli space of abelian varieties with extra structures involving polarizations, endomorphisms and level structures. To achieve this we first define the corresponding PEL datum.\par 
Let $\bbB:=\bbD^{\mathrm{op}}$ and $\bbV=\bbD$ the left $\bbB$-module via right multiplication. As explained in \cite[\S5.2.1]{Ha05}, we can find $\xi\in\bbD^\times$ such that $\xi^*=-\xi$ and the involution $\dagger$ on $\bbB=\bbD^{\mathrm{op}}$ defined by $x^\dagger:=\xi x^*\xi^{-1}$ is positive. Then the pairing $(\cdot\,,\cdot):\bbV\times\bbV\to\bbQ$ defined by
\[(x,y):=\mathrm{tr}_{\bbD/\bbQ}(x\xi y^*)\]
gives $\bbV$ the structure of a non-degenerate $\dagger$-Hermitian $\bbB$-module. Moreover, replacing $\xi$ by $-\xi$ if necessary, we may assume that $(\cdot,h_0(i)\cdot)$ is positive definite.\par   
The cocharacter $\mu$ induces a decomposition of $\bbB$-modules $\bbV_\bbC=\bbV_1\oplus\bbV_2$ where $\mu(z)$ acts by $z^{-1}$ (resp. $1$) on $\bbV_1$ (resp. $\bbV_2$) for all $z\in\bbC^\times$. The reflex field $\bbE$ is also the field of definition of the $\bbB$-module $\bbV_1$. 
%{\cg Tom: in \cite[$\S 5.2.2$]{Ha05}, $\mu(z)$ acts by $z^{-1}$ on $\bbV_1$ -- should we change it or will it create a sign problem elsewhere?} {\cg Tom: I think we should use the convention of \cite{Ko-simple}, which is the same of that of \cite{Ha05}.  I was quite careful in \cite{Ha05} to show that the determinant condition with the convention of that paper leads to the $r_{-\mu}$ in the final theorem, as opposed to $r_{\mu}$.}
\begin{rem}
    Our convention on the cocharacter $\mu$ is consistent with \cite{Ko-points}, \cite[$\S 5.2.2$]{Ha05}, \cite{Sch-Shin} and is opposite to \cite{RZ}, \cite{Sch-deformation}.
\end{rem}

\subsubsection{Integral data}\label{sec:integral-data}
Fix a prime $p$ such that $\bbF/\bbF_0$ is split above any prime of $\bbF_0$ over $p$. Let $\fp$ be a prime of $\bbE$ over $p$ and let $\bbE_\fp$ be the $\fp$-adic completion of $\bbE$. To define integral models of the Shimura varieties over $\spec\,\cO_{\bbE_\fp}$ we fix integral PEL data as follows. Let $\cO_\bbD$ be a maximal $\bbZ_{(p)}$-order in $\bbD$ that is stable under the involution $\dagger$ and let $\cO_\bbB:=\cO_\bbD^{\mathrm{op}}$. Denote $B=\bbB\otimes_\bbQ\bbQ_p$, $V=\bbV\otimes_\bbQ\bbQ_p$ and $\cO_B:=\cO_\bbB\otimes_\bbZ\bbZ_p$. Then $\cO_B$ is a maximal $\bbZ_p$-order in $B$. Let $v_1,\dotsc,v_s$ be the places of $\bbF_0$ above $p$. For each $1\le i\le s$, there are two places $w_i,w_i^c$ of $\bbF$ above $v_i$. Let $B_i:=\bbB\otimes_\bbF\bbF_{w_i}$, $V_i:=\bbV\otimes_\bbF\bbF_{w_i}$ and $\cO_{B_i}:=\cO_\bbB\otimes_{\cO_\bbF}\cO_{\bbF_{w_i}}$. Then the involution $\dagger$ induces isomorphisms $\cO_{B_i}^{\mathrm{op}}\cong\cO_\bbB\otimes_{\cO_\bbF}\cO_{\bbF_{w_i^c}}$ and $B_i^{\mathrm{op}}\cong\bbB\otimes_\bbF\bbF_{w_i^c}$. 
The alternating form $(\cdot,\cdot)$ on $\bbV$ induces isomorphisms $\bbV\otimes_\bbF\bbF_{w_i^c}\cong V_i^*$ where $V_i^*$ is the $\bbQ_p$-linear dual of $V_i$. Consequently we get isomorphisms
\[\bbG_0(\bbQ_p)\cong\prod_{i=1}^sB_i^\times,\quad\bbG(\bbQ_p)=\bbG_0(\bbQ_p)\times\bbQ_p^\times.\]

For each $1\le i\le s$, let $\cL_i=\{\La_{i,j},j\in\bbZ\}$ be an $\cO_{B_i}$-lattice chain in $V_i$ where we label the lattices so that $\La_{i,j}\subset\La_{i,j+1}$ for each $j\in\bbZ$. We extend this to  a self-dual multi-chain of $\cO_B$-lattices $\cL$ in $V$ as follows. For each $i$, we extend $\cL_i$ to a $(\cdot\,, \cdot)$-self dual lattice chain $\cL_i^+$ in $V_i \oplus V^*_i$ using the method of \cite[$\S5.2.3$]{Ha05}; this depends explicitly on the choice of the element $\xi$ above.
%{\cc Jingren: I agree with the change.}
%{\cg Tom: I don't think this can be quite right, since the definition of dual does not depend on the choice of the pairing $(\cdot \,, \cdot)$, equivalently, it is independent of the choice of $\xi$, and therefore the dual is not adapted to the actual pairing $(\cdot \, , \cdot)$. I think we should handle this like in \cite[$\S5.2.3$]{Ha05}. I put my suggestion in blue.} 
We let $\cL$ be the set of $\cO_B$-lattices in $V=\prod_{i=1}^s (V_i\oplus V_i^*)$ of the form $\prod_{i=1}^s\La_i^+$ where $\La_i^+\in\cL_i^+$ for each $i=1,\dotsc,s$.\par 
The stabilizer $\bfP_\cL$ of $\cL$ in $\bbG_0(\bbQ_p)$ is a parahoric subgroup and $K_\cL:=\bfP_\cL\times\bbZ_p^\times$ is a parahoric subgroup of $\bbG(\bbQ_p)$. The corresponding $\bbZ_p$-model of $\bbG$ (resp. $\bbG_0$) will be denoted by $\bfG_\cL$ (resp. $\bfG_{\cL,0}$). Finally we fix a neat open compact subgroup $K^p\subset\bbG(\bbA_f^p)$.
\subsection{Integral models}
We will define integral models of Shimura varieties with level $K=K^pK_\cL$ structure.\par  
For any scheme $S$, let $AV_{\cO_{\bbB}}(S)$ be the $\bbZ_{(p)}$-linear category in which:
\begin{itemize}
    \item The objects are pairs $(A,i)$ where $A$ is an abelian scheme over $S$ and $i:\cO_\bbB\to\mathrm{End}(A)\otimes\bbZ_{(p)}$ is a $\bbZ_{(p)}$-algebra homomorphism;
    \item The morphisms between two objects $(A,i)$ and $(A',i')$ form the $\bbZ_{(p)}$-module $\mathrm{Hom}_{\cO_\bbB}(A,A')$ consisting of elements in $\mathrm{Hom}(A,A')\otimes\bbZ_{(p)}$ that commute with $\cO_\bbB$.
\end{itemize}
For each object $(A,i)$ in $AV_{\cO_{\bbB}}(S)$, let $A^\vee$ be the dual abelian scheme and $i^\vee(b):=i(b^\dagger)^\vee$ for all $b\in\cO_\bbB$. Then $(A^\vee,i^\vee)$ is also an object in $AV_{\cO_{\bbB}}(S)$. A \emph{polarization} of $(A,i)$ is an isomorphism $\lambda:A\to A^\vee$ in $AV_{\cO_{\bbB}}(S)$ such that $n\lambda$ is induced by an ample line bundle on $A$ for some positive integer $n$. In particular, $\lambda$ induces the Rosati involution $\iota_\la$ on $\mathrm{End}(A)\otimes\bbZ_{(p)}$ and $i(b^\dagger)=i(b)^{\iota_\la}$ for all $b\in\cO_\bbB$. Note that what we call a \emph{polarization} is called \emph{principal polarization} in \cite[\S6]{RZ}.\par 
To define the moduli problem, recall from \cite[Definition 6.5]{RZ} that an \emph{$\cL$-set of abelian varieties  $A_\bullet=\{A_\La,\La\in\cL\}$ in $AV_{\cO_{\bbB}}(S)$} is a functor from $\cL$ to $AV_{\cO_{\bbB}}(S)$ satisfying a periodicity condition and a condition on the height of the isogenies among the $A_\La$'s. For each $\cL$-set $A_\bullet$, there is a dual $\cL$-set $\tilde{A}_\bullet$ and we have the notion of a polarization $\lambda:A_\bullet\to\tilde{A}_\bullet$, cf. \cite[Definition 6.6]{RZ}. \par 

We define the functor $\mathfrak{M}_{K_\cL K^p}$ on the category of $\cO_{\bbE_\fp}$-schemes that associates to a test scheme $S$ the set of isomorphism classes of the following objects:
\begin{itemize}
    \item An $\cL$-set of abelian varieties $A_\bullet = (A_\bullet, i_\bullet)$ in $AV_{\cO_{\bbB}}(S)$;
    \item A polarization $\la$ of $A_\bullet$;
    \item A $K^p$-level structure on $A_\bullet$, i.e. a $K^p$-orbit of isomorphisms of skew-Hermitian $\bbB$-modules
    \[\bar{\eta}:H_1(A_\bullet,\bbA_f^p)\cong\bbV\otimes\bbA_f^p\mod K^p\]
    that respect the bilinear forms on both sides up to a constant in $(\bbA_f^p)^\times$. More precisely, we consider such isomorphisms $\bar{\eta}: H_1(A_{\bullet, s}, \bbA^p_f) \cong V_{\bbA^p_f}$ for every geometric point $s$ of $S$, and we require that $\bar{\eta}$ is fixed by $\pi_1(S,s)$ (see \cite[p.\,390]{Ko-points}).
\end{itemize}
Moreover we require $A_\bullet$ to satisfy the usual Kottwitz determinant condition, i.e.,\,an equality of polynomial functions
\[\det(b;\mathrm{Lie}A_\La)=\det(b;\bbV_1)\]
for all $\La\in\cL$ and all $b\in\cO_{\bbB}$.

\begin{prop}
The functor $\fM_{K_\cL K^p}$ is representable by a projective scheme over $\spec\,\cO_{\bbE_\fp}$ whose generic fiber is isomorphic to a disjoint union of $|\ker^1(\bbQ,\bbG)|$ copies of $\mathrm{Sh}_{K_\cL K^p}$.
\end{prop}
\begin{proof}
Each lattice $\La\in\cL$ determines a maximal parahoric subgroup $K_\La$ whose corresponding moduli functor $\fM_{K_\La K^p}$ parametrizes a single object in $AV_{\cO_\bbB}(S)$ with polarization and level $K^p$-structure, satisfying the determinant condition. By \cite[Theorem 5.2]{Sch-deformation}, $\fM_{K_\La K^p}$ is representable by a projective scheme $\cO_{\bbE_\fp}$. Note that $\fM_{K_\cL K^p}$ embeds as a closed subscheme of a product of finitely many $\fM_{K_\La K^p}$'s, hence is also a projective scheme over $\cO_{\bbE_\fp}$. The description of the generic fibre of $\fM_{K_\cL K^p}$ follows from \cite[\S8]{Ko-points}.
\end{proof}

\subsection{Description of cohomology}
Let $\xi$ be an irreducible finite dimensional algebraic representation of $\bbG$ that is defined over a number field $L$. Fix a place $\la$ of $L$ above a rational prime $\ell\ne p$, then by the construction in \cite[\S6]{Ko-points} we obtain $\ell$-adic local systems $\cF_{\xi,K}$ (resp.,\,$\cF_{\xi,K_\cL K^p}$) on $\Sh_K$ (resp.,\,the integral model $\fM_{K_\cL K^p}$) together with Hecke correspondences among them. Taking cohomology in degree $i\ge0$, we get $\mathrm{Gal}(\overline{\bbQ}/\bbE)\times\bbG(\bbA_f)$-representations
\[H_\xi^i:=\varinjlim_K H_{\acute{e}t}^i(\Sh_K\otimes_{\bbE}\overline{\bbQ},\cF_{\xi,K})\]
where $\overline{\bbQ}$ is the algebraic closure of $\bbQ$ in $\bbC$. Here $K \subset \bbG(\bbA_f)$ ranges over all neat compact open subgroups of the form $K = K_p K^p$, where $K_p \subset K_\cL$ is a normal compact open subgroup. We also form the virtual $\mathrm{Gal}(\overline{\bbQ}/\bbE)\times\bbG(\bbA_f)$-representation  
\[H_\xi^*:=\sum_i(-1)^iH_\xi^i\]
where the alternating sum is taken in the corresponding Grothendieck group. Then we have a decomposition
\[H_\xi^*=\bigoplus_{\pi_f}\pi_f\otimes\sigma_{\xi}(\pi_f)\]
where $\pi_f$ ranges over irreducible admissible representations of $\bbG(\bbA_f)$ and $\sigma_\xi(\pi_f)$ is a virtual finite dimensional representation of $\mathrm{Gal}(\overline{\bbQ}/\bbE)$.  Now we can state our global main result, which gives a nearly complete description of the restriction of $\sigma_\xi(\pi_f)$ to the local Weil group at $\fp$.
\begin{thm}\label{thm:cohomology-isom}
For any self-dual $\cO_B$-lattice chain $\cL$ in $V$ with associated parahoric subgroup $K_\cL\subset\bbG(\bbQ_p)$ as in \S\ref{sec:integral-data}, we have an equality in the Grothendieck group of $\bbG(\bbA_f^p)\times K_\cL\times W_{\bbE_\fp}$ representations
\[H^*_\xi=\sum_{\pi_f}a(\pi_f)\pi_f\otimes(r_{-\mu}\circ\varphi_{\pi_p}|_{W_{\bbE_\fp}})|\cdot|_{\bbE_\fp}^{-\dim\mathrm{Sh}/2}\]
where $a(\pi_f)\in\bbZ$ is defined in \cite[p.657]{Ko-simple}.
\end{thm}

\begin{rem}
As remarked in \cite[\S5]{Sch-Shin}, we expect this to be an equality in the Grothendieck group of $\bbG(\bbA_f)\times W_{\bbE_\fp}$-representations. Under some extra assumptions the expectation is true. We refer to \cite[Theorem 5.2, Corollary 5.3]{Sch-Shin} for details.
\end{rem}
\iffalse
{\cg Tom: I thought one does not need the Steinberg condition -- Theorem 5.1 of \cite{Badu} seems to have multiplicity one with no conditions, at least for discrete series automorphic forms. It is not clear to me what is needed in general.} {\cc (Jingren: my previous remark is confusing. Actually these conditions are needed for the existence of base change, not strong multiplicity one. I simply refer to Scholze-Shin for the precise statements.)} \fi

The proof of Theorem~\ref{thm:cohomology-isom} will be finished in \S\ref{sec:pseudo-stabilization}, as a consequence of the local results Theorem~\ref{thm:vanishing-property} and Theorem~\ref{thm:main-local}.\par  
First we record its implications for the semisimple local $L$-factors of the Shimura varieties $\mathrm{Sh}_K$ at $\fp$. We introduce more notations, following \cite[\S6.3]{Ha14}. Let $r_\fp:\hat{G}\rtimes W_{\bbQ_p}\to\mathrm{GL}(V_{r_\fp})$ be the algebraic representation defined by 
\[r_\fp:=\mathrm{Ind}^{\hat{G}\rtimes W_{\bbQ_p}}_{\hat{G}\rtimes W_{\bbE_\fp}}r_{-\mu}\]
where $r_{-\mu}:\hat{G}\rtimes W_{\bbE_\fp}\to\mathrm{GL}(V_{-\mu})$ is the representation defined in \S\ref{sec:main-local-thm-statements}. For an irreducible smooth representation $\pi_p$ of $\bbG(\bbQ_p)=G(\bbQ_p)$, let $\varphi_{\pi_p}:W_{\bbQ_p}\to\hat{G}\rtimes W_{\bbQ_p}$ be its \emph{semisimple $L$-parameter}. Then the semisimple local $L$-factor of $\pi_p$ with respect to the representation $r_\fp$ is defined by
\[L^{ss}(s,\pi_p,r_\fp)=\det(\mathrm{Id}-p^{-s}r_\fp(\Phi_p);V_{r_\fp}^{I_{\bbQ_p}})\]
where $I_{\bbQ_p}\subset W_{\bbQ_p}$ is the inertia subgroup and $\Phi_p\in W_{\bbQ_p}$ is a geometric Frobenius element.
\begin{cor}
Let $K\subset\bbG(\bbA_f)$ be any sufficiently small compact open subgroup  which factorizes in the form $K = K_p K^p$ specified above. Then the semisimple local Hasse-Weil factor of $\mathrm{Sh}_K$ at the place $\fp$ of $\bbE$ is given by 
\[\zeta_\fp^{\mathrm{ss}}(\mathrm{Sh}_K,s)=\prod_{\pi_f}L^{\mathrm{ss}}(s-\frac{\dim\mathrm{Sh}_K}{2},\pi_p,r_\fp)^{a(\pi_f)\dim\pi_f^K}.\]
\end{cor}
\begin{proof}
We denote $E=\bbE_\fp$ and $d:=\dim\mathrm{Sh}_K$. Let $E_0$ be the maximal unramified extension of $\bbQ_p$ in $E$. Write $K=K_pK^p$ where $K^p\subset\bbG(\bbA_f^{p})$ and $K_p$ is contained in a parahoric subgroup $K_\cL$ in $\bbG(\bbQ_p)$ associated to some $\cO_B$-stable lattice chain $\cL$ in $V$. Recall from \cite[Definition 9.4]{Ha05} that
\[\log\zeta_\fp^{\mathrm{ss}}(\mathrm{Sh}_K,s)=\sum_{j=1}^\infty\mathrm{Tr}^{ss}(\Phi_\fp^j,H^*(\mathrm{Sh}_{K,\overline{\bbQ}},\overline{\bbQ}_\ell))\frac{N\fp^{-js}}{j}\]
where the semisimple trace $\mathrm{Tr}^{ss}$ is (by definition) the trace of $\Phi_\fp^j$ on the $I_{E}$-invariant subspace of the \emph{semisimplification} of the corresponding $W_{E}$-representations.\par  
We take $\xi$ to be the trivial representation and apply Theorem~\ref{thm:cohomology-isom}. Then for any $\tau\in\Phi_\fp^jI_{E}\subset W_{E}$, where $\mathrm{\Phi}_\fp\in W_{E}$ is a geometric Frobenius element, we get that
\[\mathrm{Tr}(\tau,H^*(\mathrm{Sh}_{K,\overline{\bbQ}},\overline{\bbQ}_\ell))=\sum_{\pi_f}a(\pi_f)\dim(\pi_f^K)\mathrm{Tr}(r_{-\mu}\circ\varphi_{\pi_p}(\tau))(N\fp)^{-jd/2}.\]

Integrating both sides over the coset $\tau\in\Phi_\fp^jI_{E}$ (against the Haar measure on $I_{E}$ with total volume $1$) and taking the sum over $j$ we get
\begin{equation}\label{eq:log-ss-zeta}
    \log\zeta_\fp^{\mathrm{ss}}(\mathrm{Sh}_K,s-\frac{d}{2})=\sum_{\pi_f}a(\pi_f)\dim(\pi_f^K)\sum_{j=1}^\infty\mathrm{Tr}^{ss}(\varphi_{\pi_p}(\Phi_\fp^j),r_{-\mu})\frac{(N\fp)^{-js}}{j}.
\end{equation}
The semisimple trace on the right hand side of \eqref{eq:log-ss-zeta} coincides with the trace on the inertia invariant subspace since $\varphi_{\pi_p}$ is the \emph{semisimple} $L$-parameter so that the representation $r_{-\mu}\circ\varphi_{\pi_p}|_{W_{E}}$ is such that inertia acts through a finite quotient. In other words we have
\[\mathrm{Tr}^{ss}(\varphi_{\pi_p}(\Phi_\fp^j),r_{-\mu})=\mathrm{Tr}(\varphi_{\pi_p}(\Phi_\fp^j)|r_{-\mu}^{I_{E}}).\]
On the other hand, recall from \cite[Definition 9.8]{Ha05} that
\[\log L^{\mathrm{ss}}(s,\pi_p,r_\fp)=\sum_{r=1}^\infty\mathrm{Tr}^{ss}(\varphi_{\pi_p}(\Phi_p^r);V_{r_\fp})\frac{p^{-rs}}{r}.\]
By \cite[Lemma 6.3.1]{Ha14} we have 
\[\mathrm{Tr}^{ss}(\varphi_{\pi_p}(\Phi_p^r);V_{r_\fp})=\begin{cases}
    [E_0:\bbQ_p]\mathrm{Tr}^{ss}(\varphi_{\pi_p}(\Phi_\fp^j),r_{-\mu}^{E_0})&\text{ if }r=j[E_0:\bbQ_p]\\
    0&\text{ if }[E_0:\bbQ_p]\nmid r
\end{cases}\]
where $r_{-\mu}^{E_0}:=\mathrm{Ind}_{\Hat{G}\rtimes W_E}^{\Hat{G}\rtimes W_{E_0}}r_{-\mu}$.
Since the extension $E/E_0$ is totally ramified, we have \[I_{E_0}/I_E=\frac{\Hat{G}\rtimes W_{E_0}}{\Hat{G}\rtimes W_{E}}\]
from which we deduce that $(r_{-\mu}^{E_0})^{I_{E_0}}=r_{-\mu}^{I_E}$ and hence
\[\mathrm{Tr}^{ss}(\varphi_{\pi_p}(\Phi_\fp^j),r_{-\mu}^{E_0})=\mathrm{Tr}(\varphi_{\pi_p}(\Phi_\fp^j),(r_{-\mu}^{E_0})^{I_{E_0}})=\mathrm{Tr}(\varphi_{\pi_p}(\Phi_\fp^j),r_{-\mu}^{I_E})=\mathrm{Tr}^{ss}(\varphi_{\pi_p}(\Phi_\fp^j),r_{-\mu}).\]
Therefore we get
\begin{equation}\label{eq:log-ss-L}
    \log L^{\mathrm{ss}}(s,\pi_p,r_\fp)=\sum_{j=1}^\infty\mathrm{Tr}^{ss}(\varphi_{\pi_p}(\Phi_\fp^j),r_{-\mu}^{E_0})\frac{p^{-j[E_0:\bbQ_p]s}}{j}=\sum_{j=1}^\infty\mathrm{Tr}^{ss}(\varphi_{\pi_p}(\Phi_\fp^j),r_{-\mu})\frac{(N\fp)^{-js}}{j}.
\end{equation}
Now the result follows by comparing \eqref{eq:log-ss-zeta} and \eqref{eq:log-ss-L}.
\end{proof}

\subsection{Counting points}
Let $E=\bbE_\fp$. Fix an integer $j\ge1$ and an element $\tau\in\mathrm{Frob}_E^jI_E\subset W_E$. Denote $r:=j[\kappa_E:\bbF_p]$ where $\kappa_E$ is the residue field of $E$ and set $k_r:=\bbF_{p^r}$, which is a degree $j$ extension of $\kappa_E$. Let $h\in C_c^\infty(K_\cL)$ and $f^p\in C_c^\infty(\bbG(\bbA_f^p))$. Our goal is to calculate $\mathrm{Tr}(\tau\times hf^p|H_\xi^*)$ by the Grothendieck-Lefschetz trace formula. 
We mostly follow the exposition of \cite{Sch-deformation}. Without loss of generality, we may assume that $f^p$ is the characteristic function of a double coset $K^pg^pK^p$ divided by the volume of $K^p$ for a neat compact open subgroup $K^p\subset\bbG(\bbA_f^p)$, and $h$ is the characteristic function of $K_pg_pK_p$ divided by the volume of $K_p$ for some normal open subgroup $K_p\subset K_\cL$ and $g_p\in K_\cL$.

\subsubsection{Frobenius-Hecke correspondence}
We will analyze the cohomological correspondence on $\Sh_{K_pK^p}$ that gives rise to the action of $\tau\times hf^p$ on $H^*_\xi$. The first step is to consider the induced correspondence on the ``base level" $\Sh_{K_\cL K^p}$, which extends to the integral model $\fM_{K_\cL K^p}$, and describe its fixed points in the special fiber. 
More precisely the action of $hf^p$ is defined by a correspondence
\begin{equation}\label{eq:Hecke-corr-base-level}
    \fM_{K_\cL K^p}\xleftarrow{p_1}\fM_{K_\cL K^p_{g^p}}\xrightarrow{p_2}\fM_{K_\cL K^p}
\end{equation}

where 
\begin{itemize}
    \item $K^p_{g^p}:=K^p\cap g^pK^p(g^p)^{-1}$;
    \item $p_1$ is the composition of the natural projection $\fM_{K_\cL K^p_{g^p}}\to\fM_{K_\cL,g^pK^p(g^p)^{-1}}$ with the isomorphism $\fM_{K_\cL,g^pK^p(g^p)^{-1}}\xrightarrow{\cong}\fM_{K_\cL K^p}$ induced by (``right multiplication by") $g^p$;
    \item $p_2$ is the natural projection.
\end{itemize}
Our convention is slightly different from both \cite{Sch-deformation} and \cite{Ko-points}.

\subsubsection{Fixed points of correspondence}\label{sec:fixed-points-of-correspondence}
Let $\mathrm{Fix}_{j,\cL}(g^p)$ denote the set of fixed points of $\tau\times hf^p$ in $\fM_{K_\cL K^p_{g^p}}(\overline{\bbF}_p)$. The notation is meant to suggest that it is independent of $g_p\in K_\cL$ and the choice of $\tau$ in the coset $\mathrm{Frob}_E^jI_E$. To describe it we consider the following correspondence of the special fiber
\begin{equation}\label{eq:Frob-Hecke-corr}
    \fM_{K_\cL K^p}\otimes\kappa_E\xleftarrow{\mathrm{Fr}_E^j\circ p_1}\fM_{K_\cL K^p_{g^p}}\otimes\kappa_E\xrightarrow{p_2}\fM_{K_\cL K^p}\otimes\kappa_E
\end{equation}
where $p_1, p_2$ are the special fibers of the corresponding maps in \eqref{eq:Hecke-corr-base-
level} and
$\mathrm{Fr}_E$ denotes the geometric Frobenius endomorphism of $\fM_{K_\cL K^p}\otimes\kappa_E$. Then $\mathrm{Fix}_{j,\cL}(g^p)$ is the fixed point set of the above correspondence, which consists of points $(A_\bullet,\la,\bar{\eta})$ in $\fM_{K_\cL K^p_{g^p}}(\overline{\bbF}_p)$ such that $(A_\bullet,\la,\bar{\eta})$ and $\sigma^r(A_\bullet,\la,\overline{\eta g^p})$ define the same point in $\fM_{K_\cL K^p}(\overline{\bbF}_p)$. Here $\sigma^r(\cdot)$ means base change along $p^r$-th power map of $\overline{\bbF}_p$.
Then for each $i\in\bbZ$, there exists an isomorphism $u_i:\sigma^r(A_i)\to A_i$ in $AV_{\cO_\bbB}(\overline{\bbF}_p)$ sending $\sigma^r(\overline{\eta g^p})$ to $\bar{\eta}$ and compatible with the isogenies among the $A_i$'s. Moreover, there is a number $c_0\in\bbZ_{(p)}^\times\cap\bbR_{>0}$ such that $u_0^*\la_0=c_0\sigma(\la_0)$. Then $(A_0,u_0,\la_0)$ is a $p^rc_0$-polarized virtual $\cO_{\bbB}$-abelian variety over $\bbF_{p^r}$ up to $\bbZ_{(p)}$-isogeny in the sense of \cite{Ko-points}. The inverse of the Frobenius endomorphisms on $A_j$ define a conjugacy class $\gamma=(\gamma_l)_{l\ne p}\in\bbG(\bbA_f^p)$ and a $\sigma$-conjugacy class $\delta=(\delta_1,\delta_0)\in\bbG(\bbQ_{p^r})=G(\bbQ_{p^r})\times\bbQ_{p^r}^\times$ such that $\delta_0\sigma(\delta_0)\dotsm\sigma^{r-1}(\delta_0)=p^rc_0$.\par 
Via $u_i$, the $p$-divisible groups $A_i[p^\infty]$ over $\overline{\bbF}_p$ descend to $\bbF_{p^r}$ and the resulting $p$-divisible $\cO_\bbB$-module decomposes according to the isomorphism $\cO_{\bbB}\otimes\bbZ_p\cong\cO_B\times\cO_B^{op}$. Let $H_i$ be the factor of $A_i[p^\infty]$ corresponding to $\cO_B$. Then $H_i$ is a $p$-divisible $\cO_B$-module over $\bbF_{p^r}$. The isogenies between $A_i$'s induce isogenies between $H_i$'s so that they form a $p$-divisible group with $(\cO_B,\mu_1,\cL)$-structure over $\bbF_{p^r}$, where $\bbF_{p^r}$ is given the strucuture of $\cO_E$ algebra by viewing it as a degree $j$ extension of the residue field $\kappa_E$. Here $\mu_1$ is the coweight of $\bbG_{0,\overline{\bbQ}_p}$ defined as follows: we choose a field isomorphism $\overline{\bbQ}_p\cong\bbC$ and interpret $\mu$ as a coweight of $\bbG_{\overline{\bbQ}_p}$. Then $\mu_1$ is the first factor of $\mu$ under the canonical isomorphism $\bbG_{\overline{\bbQ}_p}\cong\bbG_{0,\overline{\bbQ}_p}\times\bbG_{m,\overline{\bbQ}_p}$.

\subsubsection{Fixed points in an isogeny class}\label{sec:fixed-points-in-isogeny-class}
Let $(A,u,\la)$ be a $p^rc_0$-polarized virtual $\cO_\bbB$-abelian variety over $\bbF_{p^r}$. Let $I_{(A,u,\la)}$ be the $\bbQ$-algebraic group of self quasi-isogenies of $A$ compatible with $u,\la$. \par 
Let $\mathrm{Isog}(A,u,\la)$ be the subset of $\mathrm{Fix}_{j,\cL}(g^p)$ consisting of points $(A_\bullet,\la,\bar{\eta})$ lying in the closure of the generic fiber\footnote{ Note that the results of Pappas-Rapoport \cite{PR03} show that the parahoric-level integral models we use, sometimes called naive, need not be flat.} whose associated polarized virtual $\cO_{\bbB}$-abelian variety (described as above) is isomorphic to $(A,u,\la)$.

\begin{prop}\label{prop:isog-injection}
There is an injection
\[\mathrm{Isog}(A,u,\la)\into I_{(A,u,\la)}(\bbQ)\backslash(Y_p\times Y^p)\]
where
\begin{itemize}
    \item $Y_p:=\{g\in\bbG(\bbQ_{p^r})/K_{\cL,r}~|~ pg\Lambda\subset p\delta\sigma g\Lambda \subset g\Lambda,\forall \Lambda \in \cL\}$
     \item $Y^p:=\{x\in\bbG(\bbA_f^p)/K^p_{g^p}~|~x^{-1}\ga x\in g^pK^p\}$.
\end{itemize}
\end{prop}
%{\cg Tom: I think Scholze also forgets the $g^p$ subscript in his Prop. 6.5.  In addition, I think there is a sign mistake somewhere here, and also in Scholze: the condition $x^{-1} \gamma x \in g^p K^p$ is not well-defined, independent of the choice of representation $x$ for the coset $x K^p_{g^p}$, the way $K^p_{g^p}$ is currently defined.  I think the correction is the following: we define the Hecke correspondence as above, except that we should define $K^p_{g^p} = K^p \cap g^p K^p (g^p)^{-1}$,  and $\mathrm{Fr}_E^j \circ p_1$ should be the canonical projection followed by the ``right multiplication by $g^p$'' isomorphism followed by Frobenius.  We should define $p_2$ as the canonical projection (not Frobenius twisted). Now we define the map induced on cohomology (as usual, ``from left", but in any case this doesn't affect the set of fixed points). Then a fixed point corresponds to a $K^p_{g^p}$-level structure with $\sigma^r(\overline{\eta g^p}) = \bar{\eta}$ as $K^p$-level structures.  This translates to the condition $ \gamma^{-1} x g^p K^p = x K^p$, or $x^{-1} \gamma x \in g^pK^p$.  This is what Scholze gets in the end, but I think he makes this small mistake along the way. I make my suggestion below.}

\begin{proof}
The proof is similar to \cite[Proposition 6.5]{Sch-deformation}, but we take into account our slightly different convention for the Frobenius-Hecke correspondence. A fixed point for the Frobenius-Hecke correspondence is a point $(A_\bullet, i_\bullet, \la, \bar{\eta}) =: (A_\bullet, \la, \bar{\eta}) \in \mathfrak{M}_{K_\cL\, K^p_{g^p}}(\overline{\bbF}_p)$ such that there is an isomorphism $(A_\bullet,\la, \bar{\eta}) \cong \sigma^r(A_\bullet, \la, \overline{\eta g^p})$ in $\mathfrak{M}_{K_\cL\, K^p}(\overline{\bbF}_p)$.  The condition that $\sigma^r(\overline{\eta g^p}) = \bar{\eta}$ as $K^p$-level structures translates to the condition $ \gamma^{-1} x g^p K^p = x K^p$, or $x^{-1} \gamma x \in g^pK^p$. The rest of the argument is as in \cite[Proposition 6.5]{Sch-deformation}.
\end{proof}

\begin{cor}\label{cor:trace-coarse-expansion}
We have the following equality
\[\mathrm{Tr}(\tau\times hf^p| H_\xi^*)=\sum_{(A,u,\la)}
\mathrm{vol}(I(\bbQ)\backslash I(\bbA_f))\,O_{\gamma}(f^p)\,TO_{\delta\sigma}(\phi_{\tau,h})\,\mathrm{tr}\,\xi(\gamma_\ell)\]
where the sum runs over isogeny classes of polarized virtual $\cO_\bbB$-abelian varieties $(A,u,\la)$ over $\bbF_{p^r}$ and $I=I_{(A,u,\la)}$ is the automorphism group of $(A,u,\la)$.
\end{cor}
\begin{proof}
Let $\pi:\Sh_{K_p K^p}\to\Sh_{K_\cL K^p}$ be the natural projection. By the proper base change theorem, we have
\[H^*(\Sh_{K_p K^p}\otimes_{\bbE}\overline{\bbE},\cF_\xi)\cong H^*(\Sh_{K_\cL K^p}\otimes_{\bbE}\overline{\bbE},\pi_*\cF_\xi)\cong H^*(\fM_{K_\cL K^p}\otimes_{\kappa_E}\overline{\bbF}_p,R\psi\pi_*\cF_\xi).\]
Under this isomorphism, the action of $\tau\times hf^p$ is defined from the composition
\[(\mathrm{Fr}_E^j\circ p_1)^*R\psi\pi_*\cF_\xi=p_1^*(\mathrm{Fr}_E^j)^*R\psi\pi_*\cF_\xi\xrightarrow{\tau}p_1^*R\psi\pi_*\cF_\xi\xrightarrow{g_pg^p} p_2^!R\psi\pi_*\cF_\xi\]
By \cite[Proposition 5.5]{Sch-deformation}, there is a canonical isomorphism $R\psi\pi_*\cF_\xi\cong \cF_\xi\otimes R\psi\pi_*\bbQ_\ell$, under which the second map above factorizes as the tensor product of the cohomological correspondence on $\cF_\xi$ induced by $g^p$ and the cohomological correspondence on $R\psi\pi_*\bbQ_\ell$ induced by $g_p$. For a more thorough discussion of these points, we refer the reader to the forthcoming work of the second author with Rong Zhou and Yihang Zhu \cite{HZZ}.\par 
Then the (naive) local trace of $\tau\times hf^p$ on $R\psi\pi_*\cF_\xi$ at a fixed point $x=(A_\bullet,\la,\bar{\eta})$ is given by
\begin{equation}\label{eq:naive-local-term}
    \mathrm{Tr}(\tau\times hf^p|(R\psi\pi_*\cF_\xi)_x)=\mathrm{Tr}(\xi(\ga_\ell))\mathrm{Tr}(\tau\times hf^p|(R\psi\pi_*\bbQ_\ell)_x).
\end{equation}
By the version of Lefschetz trace formula in \cite[Theorem 2.3.2 (b)]{Var07} and \eqref{eq:naive-local-term}, we get
\[\mathrm{Tr}(\tau\times hf^p| H_\xi^*)=\sum_{x\in\mathrm{Fix}_{j,\cL}(g^p)}\mathrm{Tr}(\tau\times hf^p|(R\psi\pi_*\cF_\xi)_x).\]
We break the sum according to the partition $\mathrm{Fix}_{j,\cL}(g^p)=\bigsqcup\mathrm{Isog}(A,u,\la)$ into isogeny classes labeled by isomorphism classes of polarized virtual $\cO_\bbB$-abelian varieties $(A,u,\la)$ over $\bbF_{p^r}$.\par 

Following the discussion in \cite[\S6.6]{Sch-deformation},  The contribution of $\mathrm{Isog}(A,u,\la)$ to the summation above is equal to 
\[\mathrm{vol}(I(\bbQ)\backslash I(\bbA_f))\,O_{\gamma}(f^p)\,\,TO_{\delta\sigma}(\phi_{\tau,h})\,\mathrm{tr}\,\xi(\gamma_\ell).\]
Indeed, this follows from \eqref{eq:naive-local-term} and Proposition~\ref{prop:isog-injection} by noting that if $(g,x)\in Y_p\times Y^p$ satisfies 
\[\phi_{\tau,h}(g^{-1}\delta\sigma(g))\ne0,\]
then $(g,x)$ lies in $\mathrm{Isog}(A,u,\la)$. Let us explain this last assertion.  We denote by $(A_\bullet, \lambda, \bar{\eta})$ a point in $\mathrm{Fix}_{j,\cL}(g^p)$ giving rise to the isogeny class $(A,u, \la)$ and the pair $(\ga, \delta)$; we may assume that it maps to the pair $(1, 1) \in Y_p \times Y^p$ under the injection of Proposition \ref{prop:isog-injection}. Now the data $(g,x)$ such that $\phi_{\tau, h}(g^{-1}\delta \sigma(g)) \neq 0$ give us an $\cL$-set of polarized abelian varieties $A'_\bullet$, and a quasi-isogeny $A_\bullet \rightarrow A'_\bullet$ which we use to transport the $\mathcal O_{\bbB}$-action $i$ on $A_\bullet$ to an action $i'$ on $A'_\bullet$. Similarly, the $K^p_{g^p}$-level structure is transported using $x$. By construction we get $(A'_\bullet, \lambda', \bar{\eta}')$. Once we know it satisfies the Kottwitz determinant condition, we will know it is a point in our moduli problem, and thus a point in  $\mathrm{Isog}(A,u, \la)$, which maps to $(g, x)$ under the injection of Proposition \ref{prop:isog-injection}. But in this situation, the determinant condition is reflected exactly by the relative position of the lattices required in the definition of $Y_p$, as explained in \cite[$\S$11.1.1]{Ha05}.
\end{proof}

\subsection{Pseudo-stabilization}\label{sec:pseudo-stabilization}
\subsubsection{Kottwitz triples}
To proceed further, we first recall the notion of Kottwitz triple as stated in \cite[Definition 5.6]{Sch-deformation}. 
\begin{defn}\label{def:Kott-triple}
Let $j\in\bbZ_{\ge1}$ and set $r:=j[\kappa_{\bbE_\fp}:\bbF_p]$. A degree $j$ Kottwitz triple $(\ga_0;\ga,\delta)$ consists of 
\begin{itemize}
    \item A semisimple stable conjugacy class $\ga_0\in\bbG(\bbQ)$ that is elliptic in $\bbG(\bbR)$,
    \item a conjugacy class $\ga\in\bbG(\bbA_f^p)$ that is stably conjugate to $\ga_0$,
    \item a $\sigma$-conjugacy class $\delta\in\bbG(\bbQ_{p^r})$ whose naive norm $N_r\delta=\delta\sigma(\delta)\dotsm\sigma^{r-1}(\delta)$ is stably conjugate to $\ga_0$, such that $\kappa_{\bbG_{\bbQ_p}}(p\delta)=\mu^{\#}$ in $X^*(Z(\hat{\bbG})^{\Gamma(p)})$, where $\Gamma(p):=\mathrm{Gal}(\overline{\bbQ}_p/\bbQ_p)$.
\end{itemize}
\end{defn}

\begin{prop}\label{prop:Kott-triple}
Let $x\in\mathrm{Fix}_{j,\cL}(g^p)$ with associated polarized virtual $\cO_\bbB$-abelian variety $(A,u,\la)$ over $\bbF_{p^r}$, giving rise to conjugacy class $\ga\in\bbG(\bbA_f^p)$ and $\sigma$-conjugacy class $\delta\in\bbG(\bbQ_{p^r})$. Suppose that the naive norm $N_r\delta=\delta\sigma(\delta)\dotsm\sigma^{r-1}(\delta)\in\bbG(\bbQ_{p^r})$ is $\bbG(\bbQ_{p^r})$-conjugate to an element in $\bbG(\bbQ_p)$. Then there exists a unique semisimple conjugacy class $\ga_0\in\bbG(\bbQ)$ such that $(\ga_0,\ga,\delta)$ is a degree $j$-Kottwitz triple. 
\end{prop}
\begin{proof}
This is claimed in \cite[Lemma 5.3]{shen-1}, but the proof given there contains a gap, and in fact does not use the strong hypothesis on $\delta$.   %{\cg Tom: this proof of Shen seems to have a gap; he does not even use the hypothesis that the conjugacy class of $N\delta$ contains a rational element, instead he claims to use that the conjugacy class is rational (which is always the case, without assumptions.}{\cb Tom: I think I know how to fill this gap, it really relies on the ``no global endoscopy result'' that holds for these Shimura varieties.  I will add this soon.} Notation Test: $\mathfrak{K}(I_0/\bbQ)$.\par 
We will prove this by generalizing the argument of Kottwitz in \cite[\S14]{Ko-points}. Let $I$ be the reductive $\bbQ$-group of automorphisms of $(A,u,\la)$. Kottwitz chose an arbitrary maximal $\bbQ$ torus $T\subset I$ and showed that it can be embedded in $G$, and he then took $\ga_0$ to be the inverse of the Frobenius endomorphism $\pi_A$. To make the Kottwitz argument work in our situation, we make the further assumption that $T_{\bbQ_p}$ is elliptic in $I_{\bbQ_p}$. One can see that such $\bbQ$-tori exist by \cite[Lemma 1.2.2]{KPS} (and the references there). The assumption that the conjugacy class of $N_r\delta$ contains an element defined over $\bbQ_p$ implies that $I_{\bbQ_p}$ is an inner form of a subgroup of $G_{\bbQ_p}$. Then we use the fact that an elliptic torus in $I_{\bbQ_p}$ transfers to any of its inner forms to remove the assumption in \cite[\S14]{Ko-points} that $G_{\bbQ_p}$ is quasi-split. More precisely, this assumption is used at two places: first on p.420, paragraph above Lemma 14.1, to show that the centralizer $N$ of $T$ in $\mathrm{End}_B(A)$ can be embedded in $C$ locally at $p$; then on p.421 where one shows that $T_0$ transfers to $A_\alpha$ at $p$.\par 
Let us elaborate more on the first point. We let $\tilde{I}=\mathrm{End}_B(A)^\times$ (resp.,\,$\tilde{G}=C^\times$) be the $\bbQ$-group of units of the semisimple algebra $\mathrm{End}_B(A)$ (resp.,\,$C=\mathrm{End}_B(V)$). Then $I$ (resp.,\,$G$) is a closed subgroup of $\tilde{I}$ (resp.,\,$\tilde{G}$). By assumption there is an element $\ga_p\in G(\bbQ_p)$ that is stably conjugate to $N_r\delta$. Then the centralizer $G_{\ga_p}$ (resp.,\,$\tilde{G}_{\ga_p}$) is an inner form of $I_{\bbQ_p}=G_{\delta\sigma}$ (resp.,\,$\tilde{I}_{\bbQ_p}=\tilde{G}_{\delta\sigma}$). Moreover, the class of $\tilde{G}_{\ga_p}$ in $H^1(\bbQ_p,\tilde{I}_{\mathrm{ad}})$ is equal to the image of the class of $G_{\ga_p}$ under the natural map $H^1(\bbQ_p,I_{{\rm ad},\bbQ_p})\to H^1(\bbQ_p,\tilde{I}_{{\rm ad},\bbQ_p})$; see the proof of \cite[Lem.\,5.8]{Ko-conj}. Let $\tilde{T}=\mathrm{Res}_{N/\bbQ}\bbG_m$ be the centralizer of $T$ in $\tilde{I}$. Let $T_{\mathrm{ad}}$ (resp.,\,$\tilde{T}_{\mathrm{ad}}$) be the image of $T$ (resp.,\,$\tilde{T}_{\mathrm{ad}}$) in the adjoint groups $I_{\mathrm{ad}}$ (resp.,\,$\tilde{I}_{\mathrm{ad}}$). Then we have a commutative diagram
\[\xymatrix{
H^1(\bbQ_p,T_{\mathrm{ad}})\ar[r]\ar[d] & H^1(\bbQ_p,I_{\mathrm{ad}})\ar[d] \\
H^1(\bbQ_p,\tilde{T}_{\mathrm{ad}})\ar[r] & H^1(\bbQ_p,\tilde{I}_{\mathrm{ad}}).
}\]
Since $T_{\bbQ_p}$ is elliptic in $I_{\bbQ_p}$, it transfers to the inner form $G_{\ga_p}$ and hence the class of $G_{\ga_p}$ lies in the image of the upper horizontal arrow. From the commutativity of the diagram we see that the class of $\tilde{G}_{\ga_p}$ lies in the image of the lower horizontal arrow, which means that $\tilde{T}_{\bbQ_p}$ transfers to the inner form $\tilde{G}_{\bbQ_p}=C_{\bbQ_p}^\times$. \iffalse {\cc (Jingren: The following sentence is not quite right, $\tilde{T}_{\bbQ_p}=\mathrm{Res}_{N_{\bbQ_p}/\bbQ_p}\bbG_m$ splits over $N_{\bbQ_p}$ only when $N_{\bbQ_p}/\bbQ_{p}$ is a Galois extension. I added the modified argument in the next paragraph)} Since $\tilde{T}_{\bbQ_p}=\mathrm{Res}_{N_{\bbQ_p}/\bbQ_p}\bbG_m$ splits over $N_{\bbQ_p}$, the existence of an embedding $\tilde{T}_{\bbQ_p}\to\tilde{G}_{\bbQ_p}$ implies that $\tilde{G}_{\bbQ_p}$ splits over $N_{\bbQ_p}$. Thus $C\otimes_{\bbQ} N_{\bbQ_p}$ is a matrix algebra over $N_{\bbQ_p}$ and we deduce that there is an algebra embedding $N_{\bbQ_p}\to C_{\bbQ_p}$ by the discussion in \cite[p.420]{Ko-points}. \fi \par 
Now we get an embedding $\tilde{T}_{\bbQ_p}\to\tilde{G}_{\bbQ_p}$ which on $\bbQ_p$-points gives rise to a group homomorphism $\varphi:N_{\bbQ}^\times\to C_{\bbQ_p}^\times$. We want to show that $\varphi$ extends to an embedding of $\bbQ_p$-algebras $N_{\bbQ_p}\to C_{\bbQ_p}$. Recall that $N$ is a finite product of fields (cf.\,\cite[p.420]{Ko-points}). Take an element $\ga\in N_{\bbQ_p}^\times$ such that $N_{\bbQ_p}=\bbQ_p[\ga]$, and note that $\ga$ is automatically a regular semisimple element in $\tilde{I}(\bbQ_p)$: its centralizer in $\tilde{I}$ is precisely the maximal torus $\tilde{T}$, thanks to the double centralizer theorem applied to $N$ in ${\rm End}_B(A)$. Let $f_\ga(X)\in\bbQ_p[X]$ be the characteristic polynomial of $\ga$, viewed as an element in the semisimple $\bbQ_p$-algebra $\mathrm{End}_B(A)_{\bbQ_p}$. Then $f_\ga(X)$ is also the minimal polynomial of $\ga$ and we have $N_{\bbQ_p}=\bbQ_p[X]/(f_\ga(X))$. By the definition of transfer of semisimple elements, $f_\ga(X)$ is the characteristic polynomial of $\varphi(\ga)\in C_{\bbQ_p}$ and hence $f_\ga(\varphi(\ga))=0$ by the Cayley-Hamilton theorem. Thus $\varphi$ extends to an embedding of $\bbQ_p$-algebras $N_{\bbQ_p}=\bbQ_p[X]/(f_\ga(X))\to C_{\bbQ_p}$ and we are done.
\end{proof}

Conversely, the construction in \cite[\S18]{Ko-points} shows that every Kottwitz triple comes from a polarized virtual $\cO_\bbB$-abelian variety, keeping in mind that in our situation the invariant $\alpha(\ga_0;\ga,\delta)$ is automatically trivial (see \cite{Ko-simple};  one checks that the arguments in \cite[\S18]{Ko-points} go through without the assumption made in \cite[\S5]{Ko-points} that $\bbB_{\bbQ_p}$ is a product of matrix algebras over unramified extension of $\bbQ_p$). Altogether, combined with the vanishing property Theorem~\ref{thm:main-local} we can follow the procedure in \cite[\S19]{Ko-points} and convert the sum in Corollary~\ref{cor:trace-coarse-expansion} to a sum over equivalence classes of degree $j$ Kottwitz triples:
\begin{equation}\label{eq:trace-Kott-triple}
    \mathrm{Tr}(\tau\times hf^p| H_\xi^*)=\sum_{(\ga_0;\ga,\delta)}c(\ga_0;\ga,\delta)\,O_\ga(f^p)\,TO_{\delta\sigma}(\phi_{\tau,h})\,\mathrm{tr}\,\xi(\ga_0)
\end{equation}
where $c(\ga_0;\ga,\delta)$ is the product of $\mathrm{vol}(I(\bbQ)\backslash I(\bbA_f))$ and the order of the finite group
\[\ker[\ker^1(\bbQ,I_0)\to\ker^1(\bbQ,\bbG)].\]
Note that we do not need to multiply the right hand side of (\ref{eq:trace-Kott-triple}) by the order of $\ker^1(\bbQ,\bbG)$, as this group is trivial for our choice of $\bbG$. (See, for example \cite[proof of Proposition 3.2]{Sch-LK-simple}.) By the discussion in \cite[\S4]{Ko-lambda} and the fact that in our situation $|\mathfrak{K}(I_0/\bbQ)|=1$ in the notation of \emph{loc.cit.} (cf.\cite{Ko-simple}), we have
\begin{equation}\label{eq:c-formula}
    c(\ga_0;\ga,\delta)=\tau(\bbG)~\mathrm{vol}(A_\bbG(\bbR)^0\backslash I(\bbR))^{-1}.
\end{equation}

\subsubsection{Pseudo-coefficients}\label{sec:pseudo-coefficients}
Recall from \cite{Ko-simple} that the algebraic representation $\xi$ of $\bbG(\bbC)$ determines a discrete series $L$-packet of $\bbG(\bbR)$. We let $f^{\bbG}_{\xi,\infty}$ be $(-1)^{q(\bbG)}$ times the sum of pseudo-coefficients for representations in this $L$-packet, divided by the cardinality of the $L$-packet. Here $q(\bbG)$ is half the real dimension of the symmetric space associated to $\bbG(\bbR)$. By \cite[Lemma 3.1]{Ko-simple}, for all semisimple elements $\ga_\infty\in\bbG(\bbR)$, we have $SO_{\ga_\infty}(f^{\bbG}_{\xi,\infty})=0$ unless $\ga_\infty$ is elliptic, in which case 
\begin{equation}\label{eq:SO-infty}
    SO_{\ga_\infty}(f^{\bbG}_{\xi,\infty})=\mathrm{tr\xi(\ga_\infty)}\mathrm{vol}(A_{\bbG}(\bbR)^0\backslash I(\bbR))^{-1}e(I_{\ga_\infty}^c)
\end{equation}
where $I_{\ga_\infty}^c$ is the inner form of the centralizer of $\ga_\infty$ in $\bbG(\bbR)$ that is anisotropic modulo the center of $\bbG$ and $e(I_{\ga_\infty}^c)$ is its Kottwitz sign whose definition we recall next. Actually we have $I_{\ga_\infty}^c=I_\bbR$. Indeed, $I_\bbR$ is an inner form of $\bbG_{\bbR, \gamma_0}$ (see \cite[p.\,423]{Ko-points}) and we point out below that $I_\bbR$ is anisotropic mod the center of $G_\bbR$.

\subsubsection{Kottwitz sign}\label{sec:kottwitz-sign}
In \cite{Ko-sign}, Kottwitz defines a sign $e(H)\in\{\pm1\}$ for any connected reductive group $H$ over a local field, which equals to $1$ if $H$ is quasi-split. Moreover, by the main theorem of \emph{loc. cit.}, if $H$ is a reductive group over a global field $F$, then there is a product formula $\prod_ve(H_v)=1$, where the product is over all places of $F$.\par 
In our situation, for any Kottwitz triple $(\ga_0;\ga,\delta)$ there is a reductive group $I$ over $\bbQ$ such that 
\begin{itemize}
    \item for any prime $l\ne p$, $I_l\cong\bbG_{\ga_l}$, the centralizer of $\ga_l$ in $\bbG(\bbQ_l)$;
    \item $I_p\cong\bbG_{\delta\sigma}$, the $\sigma$-centralizer of $\delta$ in $\bbG(\bbQ_{p^r})$;
    \item $I_\bbR$ is anisotropic mod center.
\end{itemize}
The existence of $I$ is proved in \cite[\S2]{Ko-lambda} (the condition $\alpha(\ga_0;\ga,\delta)=1$ in \emph{loc. cit.} is automatic in our situation, cf. \cite{Ko-simple}). In fact, $I$ is the automorphism group of a polarized virtual $\cO_\bbB$-abelian variety whose Kottwitz triple is equivalent to $(\ga_0;\ga,\delta)$. Let $e(\ga):=\prod_{l\ne p}e(\bbG_{\ga_l})$ (product over all finite primes different from $p$) and $e(\delta):=e(\bbG_{\delta\sigma})$. Then by the product formula, we have 
\[e(\ga)e(\delta)=e(I_\bbR).\] 

\subsubsection{Finishing the proof of Theorem \ref{thm:cohomology-isom}}
By Theorem~\ref{thm:main-local} we have
\begin{equation}
    e(\delta)TO_{\delta\sigma}(\phi_{\tau,h})=SO_{\ga_p}(f_{\tau,h})
\end{equation}
where $\ga_p\in\bbG(\bbQ_p)$ is stably conjugate to $\cN\delta=\delta\sigma(\delta)\dotsm\sigma^{r-1}(\delta)$ and $f_{\tau,h}:=z_{\tau,-\mu}*h$.\par
Combining all the discussions above we get
\[\mathrm{Tr}(\tau\times hf^p|H_\xi^*)=\tau(\bbG)\sum_{\ga_0}SO_{\ga_0}(f_{\tau,h}f^pf_{\xi,\infty}^{\bbG})\]
where the sum runs over \emph{stable} conjugacy classes in $\bbG(\bbQ)$ (which are automatically semisimple since $\bbG$ is anisotropic mod center). Then by \cite[Lemma 4.1]{Ko-simple} and the simple trace formula for $\bbG$ we get 
\begin{equation}
    \begin{split}
        \mathrm{Tr}(\tau\times hf^p|H_\xi^*)&=\sum_{\pi}m(\pi)\,\mathrm{tr}\,\pi(f_{\tau,h}f^pf_{\xi,\infty}^{\bbG})=\sum_{\pi_f}a(\pi_f)\,\mathrm{tr}\,\pi_f(f_{\tau,h}f^p)\\
        &=\sum_{\pi_f}a(\pi_f)\,\mathrm{tr}(\tau|(r_{-\mu}\circ\varphi_{\pi_p}|_{W_{\bbE_\fp}})|\cdot|^{-\dim\mathrm{Sh}/2})\,\mathrm{tr}(hf^p|\pi_f)
    \end{split}
\end{equation}
where the first sum runs over automorphic representations $\pi$ of $\bbG$, $m(\pi)$ denotes the automorphic multiplicity, while the second and third sums run over irreducible admissible representations of $\bbG(\bbA_f)$. The second equality uses the definition of $a(\pi_f)$ from \cite[p.657]{Ko-simple}:
\begin{equation}\label{eq:pi-f}
    a(\pi_f)=\sum_{\pi_\infty}m(\pi_f\otimes\pi_\infty)\,\mathrm{tr}\,\pi_\infty(f^{\bbG}_{\xi,\infty}).
\end{equation}

The third equality uses the definition of the element $z_{\tau,-\mu}$ in the stable Bernstein center of $\bbG(\bbQ_p)$. This finishes the proof of Theorem~\ref{thm:cohomology-isom}.

\subsection{Lefschetz number via central functions}\label{sec:Lefschetz-number-central-function}
In this section we record some results of independent interest. We express the Lefschetz number in \eqref{eq:trace-Kott-triple} and its semisimple variant in terms of elements in the Bernstein center instead of the function $\phi_{\tau,h}$.\par 
As usual, we fix an integer $j\ge1$, an element $\tau\in\mathrm{Frob}_E^jI_E\subset W_E$ and let $r=j[\kappa_E:\bbF_p]$. Recall from \S\ref{sec:main-local-thm-statements} the element $z_{\tau,-\mu}^{(r)}$ in the (stable) Bernstein center of $\bbG(\bbQ_{p^r})$. 

\begin{prop}\label{prop:tr-central-function}
Suppose that $h\in C_c^\infty(\bfG_\cL(\bbZ_p))$ is a base change transfer of a function $\tilde{h}\in C_c^\infty(\bbG(\bbQ_{p^r}))$ that satisfies the vanishing property in the sense of Definition \ref{defn:vanishing-property}. Then we have the following formula for the Lefschetz number
\[\mathrm{Tr}(\tau\times hf^p| H_\xi^*)=\sum_{(\ga_0;\ga,\delta)}c(\ga_0;\ga,\delta)\,O_\ga(f^p)\,TO_{\delta\sigma}(z_{\tau,-\mu}^{(r)}*\tilde{h})\,\mathrm{tr}\,\xi(\ga_0)\]
where the sum runs over equivalence classes of degree $j$-Kottwitz triples.
\end{prop}
\begin{proof}
This is an immediate consequence of Equation \eqref{eq:trace-Kott-triple} and Corollary \ref{cor:TO-equality}. 
\end{proof}

In \cite{Ko-BC}, for certain pairs of compact open subgroups $K_p\subset\bbG(\bbQ_p)$, $K_{p^r}\subset\bbG(\bbQ_{p^r})$ satisfying the axioms in \emph{loc. cit.} p.240, Kottwitz shows that the idempotent $e_{K_p}=\frac{1}{\mathrm{vol}(K_p)}1_{K_p}$ is a base change transfer of the idempotent $e_{K_{p^r}}=\frac{1}{\mathrm{vol}(K_{p^r})}1_{K_{p^r}}$ and that the latter has the vanishing property (see the discussion on \emph{loc. cit.} p.244, the paragraph above \S2). For example, if $\mathcal{G}$ is a smooth connected group scheme over $\bbZ_p$ with generic fibre $\bbG_{\bbQ_p}$, then $K_p=\mathcal{G}(\bbZ_p)$ and $K_{p^r}=\mathcal{G}(\bbZ_{p^r})$ satisfy the axioms of Kottwitz.\par 
The following consequence verifies a conjecture of the second author and Kottwitz (cf. \cite[Conjecture 6.1.1]{Ha14}) in our situation. It involves the element $z^{(r)}_{-\mu}=\int_{\mathrm{Frob}_E^jI_E}z_{\tau,-\mu}^{(r)}d\tau$ in the (stable) Bernstein center of $\bbG(\bbQ_{p^r})$ defined in \S\ref{sec:main-local-thm-statements}.

\begin{cor} \label{cor:Lefschetz-number-central-function}
Let $K_p\subset\bbG(\bbQ_p)$ and $K_{p^r}\subset\bbG(\bbQ_p)$ be compact open subgroups satisfying Kottwitz's axiom as above. Let $K^p\subset\bbG(\bbA_f^p)$ be a sufficiently small compact open subgroup and $f^p\in C_c^\infty(K^p\backslash\bbG(\bbA_f^p)/K^p)$ a test function bi-invariant under $K^p$. Then we have
\[\mathrm{Tr}^{\mathrm{ss}}(\mathrm{Frob}_\fp^j\times f^p| H^*(\Sh_{K_pK^p}\otimes_\bbE\bar{\bbQ}, \mathcal{F}_{K^pK_p})=\sum_{(\ga_0;\ga,\delta)}c(\ga_0;\ga,\delta)\,O_\ga(f^p)\,TO_{\delta\sigma}(z^{(r)}_{-\mu}*e_{K^{p^r}})\,\mathrm{tr}\xi(\gamma_0)\]
where the sum runs over equivalence classes of degree $j$-Kottwitz triples. 
\end{cor}
\begin{proof}
This follows by integrating the equation in the previous Proposition \ref{prop:tr-central-function} over $\tau\in\mathrm{Frob}_E^jI_E$ and taking $h=e_{K_p}$ (so that $\tilde{h}$ can be taken to be $e_{K^{p^r}}$ by Kottwitz's result in \cite{Ko-BC}). 
\end{proof}

\section{Proof of the local results}\label{sec:proof-local-thm}
In this section we prove Theorem~\ref{thm:vanishing-property} and Theorem~\ref{thm:main-local} by a local-global argument. It suffices to prove the local result for the group $G=\mathrm{Res}_{F/\bbQ_p}\mathrm{GL}_{m}(D)$ where $F$ is a finite extension of $\bbQ_p$ and $D$ a division algebra with center $F$. Then $G$ is an inner form of the quasi-split group $G^*=\mathrm{Res}_{F/\bbQ_p}\mathrm{GL}_n$ where $n=m(\dim_FD)^{\frac{1}{2}}$. By Proposition \ref{prop:independence-of-EL-data}, we may assume $B:=M_m(D)^{\mathrm{op}}$ and $V=M_m(D)$ viewed as a left $B$-module via right multiplication. \iffalse {\cg Tom: this does not seem to use Proposition \ref{prop:independence-of-EL-data}, which allows one to pass from $B = M_m(D)^{\rm op}$ to $B = D^{\rm op}$.  Wouldn't having $m = 1$ make a lot of what follows slightly simpler?}\fi Then $C\cong M_m(D)$ acts on $V$ by left multiplication.\par 
Let $\cO_D$ be the unique maximal order in $D$ and take $\cO_B:=M_m(\cO_D)^{\mathrm{op}}$. As in \S\ref{sec:EL-data} we are given a chain of $\cO_B$-lattices $\cL$ in $V$ giving rise to a $\bbZ_p$-model $\cG_\cL$ of $G$ so that $\cG_\cL(\bbZ_p)$ is a parahoric subgroup of $G(\bbQ_p)$.\par
Let $\Phi_F:=\mathrm{Hom}_{\bbQ_p}(F,\bar{\bbQ}_p)$. 
We can choose a representative $\mu$ of the conjugacy class $\bar{\mu}$ of the form
\begin{equation}\label{eq:local-cocharacter}
    \xymatrix@R=1pt{
    \mu:\bbG_{m,\bar{\bbQ}_p}\ar[r] & G_{\bar{\bbQ}_p}=\prod_{\Phi_F}\mathrm{GL}_{n,\bar{\bbQ}_p}\\
    t\ar@{|->}[r] & (\mu_{\tau}(t))_{\tau\in\Phi_F}
    }
\end{equation}
where for each $\tau\in\Phi_F$, $\mu_{\tau}(t)=\mathrm{diag}(t,\dotsc,t,1,\dotsc,1)$ in which $t$ occurs $a_\tau$ times. The conjugacy class $\bar{\mu}$ is then uniquely determined by the tuple of integers $(a_\tau)_{\tau\in\Phi_F}$.

\subsection{Globalization}\label{sec:global-data}
\subsubsection{CM fields}\label{sec:CM-field}
Starting from the local objects above, we are going to define some global objects, in particular a global field $\bbF$ and a group $\bbG$ over it, among others; we warn the reader that despite the common notation, these objects will not be the same as the ones defined in the Shimura variety context (see section \ref{sec:Shimura-data}). \par
Let $\bbF_0$ be a totally real field such that there is a unique place $v$ of $\bbF_0$ above $p$ and the $v$-adic completion $\bbF_{0,v}$ is isomorphic to the local field $F$. The existence of such an $\bbF_0$ is guaranteed by \cite[Lemme 8.1.2]{Far04}. Let $\bbK$ be an imaginary quadratic field in which $p$ is split and let $\bbF=\bbF_0\bbK$ be the CM field with maximal totally real subfield $\bbF_0$. The complex conjugations on $\bbK$, $\bbF$ and $\bbC$ will all be denoted by the letter ``$c$". Fix a place $w$ of $\bbF$ above $v$. Then $v$ splits in $\bbF$ as $v=ww^c$. We fix field isomorphisms $\bbF_w\cong\bbF_{0,v}\cong F$.\par 
Fix a field embedding $\tau_\bbK:\bbK\into\bbC$ and let $\Phi=\{\varphi\in\mathrm{Hom}_{\bbQ}(\bbF,\bbC)|\varphi_\bbK=\tau_\bbK\}$ be the CM-type of $\bbF$ determined by $\tau_\bbK$. In particular, we have a decomposition $\mathrm{Hom}_{\bbQ}(\bbF,\bbC)=\Phi\sqcup c\Phi$.\par 
For each place $x$ of $\bbQ$, we fix a field isomorphism $\iota_x:\bbC\xrightarrow{\sim}\bar{\bbQ}_x$. Then the isomorphism $\iota_p$ induces a bijection $\Phi\xrightarrow{\sim}\Phi_F=\mathrm{Hom}_{\bbQ}(F,\bar{\bbQ}_p)$ which permits us to associate to any $\tau\in\Phi$ an integer $a_\tau\in\{0,\dotsm,n\}$ occurring in the definition of the local cocharacter $\mu$ in \eqref{eq:local-cocharacter}.

\subsubsection{Division algebras with involution}\label{sec:division-algebra}
We write the invariant of the local division algebra $D$ in the form $\mathrm{inv}(D)=\frac{a}{n}$ for some $a\in\bbZ$.\par 
Let $\bbD$ be a central division algebra of dimension $n^2$ over $\bbF$ satisfying the following conditions:
\begin{enumerate}
    \item The opposite algebra $\bbD^{\rm op}$ is isomorphic to  $\bbD\otimes_{\bbK,c}\bbK$ (where both sides are viewed as algebras over $\bbK$).
    \item $\bbD_w\cong B^{\mathrm{op}}=M_m(D)$;
    \item At any place $x\ne w,w^c$ of $\bbF$ that is split over $\bbF_0$, $\bbD_x$ is either split or a division algebra and there is at least one such place $x$ such that $\bbD_x$ is a division algebra;
    \item At any place $x\ne w,w^c$ of $\bbF$ that is not split over $\bbF_0$, $\bbD_x$ is split;
    \item If $n$ is even, then $a+[\bbF_0:\bbQ]n/2+\sum_{\tau\in\Phi}a_\tau$ is congruent modulo 2 to the number of places $v'\ne v$ of $\bbF_0$ such that $\bbD$ is ramified at some place of $\bbF$ lying above $v'$.
\end{enumerate}

Let $\bbB:=\bbD^{\mathrm{op}}$ and let $\bbV= \bbB$ be viewed as a left $\bbB\otimes_{\bbF}\bbB^{\mathrm{op}}$-module via left and right multiplication. That is, for $b_1, b_2, x \in \bbB$, the action is given by $(b_1 \otimes b_2)(x) = b_1 x b_2$, the product in $\bbB$.  Note that this is equivalent to the set-up in $\S\ref{sec:PEL-data}$, although the conventions look superficially different.

\iffalse
{\cg Tom: I think we should interchange the symbols $\bbD$ and $\bbB$ in what follows, to make it more parallel to what we do in section 4, and what is done in \cite{Ko-lambda}.  Namely, in section 4 we started with $(\bbD, *)$ where $*$ is an involution of the 2nd kind which is NOT positive. The Shimura variety comes from $(\bbG, X)$ where $\bbG = {\rm GU}(\bbD, *)$.  But the PEL data and the moduli problem from section 4 involves $(\bbB, \dagger)$, where $\dagger$ is an involution of 2nd type which IS positive.  Now, here we start with $(\bbD, \dagger)$ where $\dagger$ is POSITIVE.  It would be better to say we start with $(\bbB, \dagger)$ here where $\dagger$ is POSITIVE, and then we construct $(\bbD, *_\beta)$, where $*_\beta$ is NOT necessarily positive. The FORMER is what is used to define the eventual moduli problem and PEL data, so it should be denoted $(\bbB, \dagger)$ (not $(\bbD, \dagger)$). We should define the Shimura group as $\bbG_\beta = {\rm GU}(\bbD, *_\beta)$. I will make my suggestions in this direction below.}
\fi

We choose a \emph{positive} involution $\dagger$ on $\bbB$ of the second kind (i.e. $\dagger$ restricts to complex conjugation on the center $\bbF$ and for any $x\in\bbB^\times$ we have $\mathrm{Tr}_{\bbB/\bbQ}(xx^\dagger)>0$). This is possible by condition (1) and (4), see the discussion on \cite[page 51-52]{HT01}. The condition in (3) that $\bbD$ is a division algebra over at least one place $x$ of $\bbF$ that splits over $\bbF_0$ will be needed in the proof of Theorem~\ref{thm:global-base-change}: it comes from an assumption in \cite[Th\'eor\`eme A.5.2]{La99}. 

\subsubsection{Unitary similitude groups} \label{sec:unit_sim_grps}
Above we consider the pair $(\bbB, \dagger)$ where $\dagger$ is a positive involution of the second type on $\bbB = \bbD^{\rm op}$.  We next need to reverse the procedure used in sections $\S\ref{sec:Shimura-data}, \ref{sec:PEL-data}$, and construct from this a suitable involution $*_\beta$ on $\bbD$ (which is not necessarily positive).

For any invertible element $0\ne\beta\in \bbB$ such that $\beta^\dagger=-\beta$ we can associate the following objects:
\begin{itemize}
    \item An involution of the second kind $*_\beta$ on $\bbB$ defined by
    \[b^{*_\beta}:=\beta b^\dagger \beta^{-1}\quad\forall b\in\bbB.\]
    \item A non-degenerate alternating pairing $(\cdot,\cdot)_\beta$ on $\bbV$ defined by
    \[(x_1,x_2)_\beta:=\mathrm{tr}_{ \bbB/\bbQ}(x_1 \beta x_2^\dagger)  \quad\forall x_1,x_2\in\bbV.\]
    Then for all $b_1,b_2\in \bbB$ we have
    $(b_1x_1b_2, x_2)_\beta=(x_1,b_1^\dagger   x_2b_2^{*_\beta})_\beta$. 
    \item A $\bbQ$-algebraic group $\bbG_\beta$ defined as the subgroup of $\mathrm{Aut}_\bbB\bbV$ preserving the form $(\cdot,\cdot)_\beta$ up to a scalar in $\bbQ$. Since $\mathrm{End}_{\bbB}(\bbV)=\bbD$, for any $\bbQ$-algebra $R$ we have 
\[\bbG_\beta(R)=\{g\in\bbD\otimes_\bbQ R \,| \, xx^{*_\beta}\in R^\times\}.\]
\end{itemize}
The map $g\mapsto gg^{*_\beta}$ defines a homomorphism of $\bbQ$-algebraic groups $c:\bbG_\beta\to\bbG_m$. Let $\bbG_{\beta,0}\subset\bbG_\beta$ be its kernel. In other words, for any $\bbQ$-algebra $R$ we have
\[\bbG_{\beta,0}(R)=\{g\in\bbB^{op}\otimes_\bbQ R \,| \, xx^{*_\beta}=1\}.\]
Let $\bbG_\beta^{\mathrm{ad}}$ be the adjoint group of $\bbG_{\beta,0}$ (and also $\bbG_\beta$). Sometimes we also view $\bbG_{\beta,0}$ and $\bbG_\beta^{\mathrm{ad}}$ as algebraic groups defined over $\bbF_0$.

\begin{lem}\label{lem:global-group}
There exists $0\ne\beta\in\bbB^{\dagger=-1}$ such that
\begin{itemize}
    \item At any finite prime that is non-split in $\bbK$, $\bbG_\beta$ and $\bbG_{\beta,0}$ are quasi-split;
    \item For any infinite place $\tau:\bbF_0\into\bbR$, the pairing $(\cdot\,, \cdot\,)_\beta$ on \[\bbV\otimes_\bbQ\bbR\cong\prod_{\tau\in\mathrm{Hom}_{\bbQ}(\bbF_0,\bbR)}\bbV\otimes_{\bbF_0,\tau}\bbR\] 
    has invariant $(a_\tau,n-a_\tau)$ at the factor corresponding to $\tau$.
\end{itemize}
\end{lem}
\begin{proof}
We follow the proof of \cite[Lemma I.7.1]{HT01}. First pick any $0\ne\beta_0\in\bbB^{\dagger=-1}$. In this proof, we view the associated groups $\bbG_{\beta_0,0}$ and $\bbG_{\beta_0}^{\mathrm{ad}}$ as algebraic groups over $\bbF_0$. We aim to twist all the objects defined by $\beta_0$ by a cohomology class in $H^1(\bbF/\bbF_0,\bbG_{\beta_0}^{\mathrm{ad}})$ so that the conditions in the lemma are satisfied. Concretely, such a cohomology class is represented by an element $0\ne\alpha\in\bbB$ with $\alpha^{*_{\beta_0}}=\alpha$ and the twisted objects will be associated to $\beta=\alpha\beta_0$. In particular, $\bbG_{\alpha\beta_0,0}$ is the inner form of $\bbG_{\beta_0,0}$ classified by $[\alpha]\in H^1(\bbF_0,\bbG_{\beta_0}^{\mathrm{ad}})$.\par
When $n$ is odd, the natural map of pointed sets
\[H^1(\bbF_0,\bbG_{\beta_0}^{\mathrm{ad}})\to\bigoplus_{x}
H^1(\bbF_{0,x},\bbG_{\beta_0}^{\mathrm{ad}}),\]
is surjective. Here the direct sum is over all places of $\bbF_0$, Hence there is no local-global obstruction to finding $\beta$ with prescribed local conditions.\par 
From now on we assume that $n$ is even. Then there is an exact sequence of pointed sets:
\begin{equation}\label{eq:obs-exact-seq}
    H^1(\bbF_0,\bbG_{\beta_0}^{\mathrm{ad}})\to\bigoplus_{x}
H^1(\bbF_{0,x},\bbG_{\beta_0}^{\mathrm{ad}})\xrightarrow{\mathrm{Obs}}\bbZ/2\bbZ\to1
\end{equation}
where $\mathrm{Obs}$ is the sum of local obstruction maps $\mathrm{Obs}_x:H^1(\bbF_{0,x},\bbG_{\beta_0}^{\mathrm{ad}})\to\bbZ/2\bbZ$ whose definition we recall next.\par 
If $x$ is a finite place of $\bbF_0$ that splits into two different places $y,y^c$ in $\bbF$, then $\mathrm{Obs}_x$ is the composition of the isomorphisms \[H^1(\bbF_{0,x},\bbG_{\beta_0}^{\mathrm{ad}})\cong H^1(\bbF_y,\bbG_{\beta_0}^{\mathrm{ad}})\cong H^2(\bbF_y,\mu_n)\cong\bbZ/n\bbZ\]
with the natural projection $\bbZ/n\bbZ\to\bbZ/2\bbZ$.\par 
If $x$ is a finite place of $\bbF_0$ that does not split in $\bbF$, then $\mathrm{Obs}_x$ is the natural isomorphism $H^1(\bbF_{x,0},\bbG_{\beta_0}^{\mathrm{ad}})\cong\bbZ/2\bbZ$.\par 
If $x$ is an infinite place of $\bbF_0$, then $H^1(\bbF_{x,0},\bbG_{\beta_0}^{\mathrm{ad}})$ is identified with the set of unordered pairs of non-negative integers $(a_x,b_x)$ with $a_x+b_x=n$. Then the map $\mathrm{Obs}_x$ sends such a pair $(a_x,b_x)$ to $a_x-a_{x,0}$ mod 2, where $(a_{x,0},n-a_{x,0})$ is the signature of $\bbG_{\beta_0}$ at $x$. \par 
For any place $x$ of $\bbF_0$, let $u_x\in H^1(\bbF_{x,0},\bbG_{\beta_0}^{\mathrm{ad}})$ be the class of the quasi-split inner form of $\bbG_{\beta_0,0}$. Since $u_x$ comes from the global cohomology class for the quasi-split inner form of $\bbG_{\beta_0,0}$ over $\bbF_0$, we have $\sum_x\mathrm{Obs}_x(u_x)=0$.\par 
Let $\mathrm{Spl}_{\bbF/\bbF_0}$ be the set of (finite) places of $\bbF_0$ that are split in $\bbF$.
Let $s$ be the number of places in $\mathrm{Spl}_{\bbF/\bbF_0}$, different from $v$, above which $\bbD$ is ramified. By our assumption, $s\ge1$ and for any place $y$ of $\bbF$ above such a place $x$ of $\bbF_0$, $\bbD_y$ is a division algebra so that $\mathrm{Obs}_x(u_x)=1$ since $n$ is even. Also we recall that by assumption $\mathrm{Inv}(\bbD_w)=\frac{a}{n}$ so that $\mathrm{Obs}_v(u_v)=a$. Thus we get
\[\sum_{x\in\mathrm{Spl}_{\bbF/\bbF_0}}\mathrm{Obs}_x(u_x)=a+s\]
For any infinite place $x$, $u_x$ corresponds to the pair $(\frac{n}{2},\frac{n}{2})$ and hence $\mathrm{Obs}_x(u_x)=a_{x,0}-\frac{n}{2}=a_{x,0}+\frac{n}{2}$.  Therefore we get
\begin{equation}\label{eq:quasi-split-invariant}
    a+s+[\bbF_0:\bbQ]n/2+\sum_{\tau|\infty}a_{\tau,0}+\sum_{x\notin\mathrm{Spl}_{\bbF/\bbF_0},x\nmid\infty}\mathrm{Obs}_x(u_x)=0.
\end{equation}

To finish the proof we need to show the existence of a global class $[\alpha]\in H^1(\bbF_0,\bbG_{\beta_0}^{\mathrm{ad}})$ whose image in $H^1(\bbF_{0,x},\bbG_{\beta_0}^{\mathrm{ad}})$ equals to
\begin{itemize}
    \item $0$ if $x\in\mathrm{Spl}_{\bbF/\bbF_0}$;
    \item $u_x$ if $x$ is a finite place of $\bbF_0$ not in $\mathrm{Spl}_{\bbF/\bbF_0}$;
    \item $(a_x,n-a_x)$ if $x$ is an infinite place.
\end{itemize}
By the exact sequence \eqref{eq:obs-exact-seq} it suffices to show that
\[\sum_{\tau|\infty}(a_\tau+a_{\tau,0})+\sum_{x\notin\mathrm{Spl}_{\bbF/\bbF_0},x\nmid\infty}\mathrm{Obs}_x(u_x)=0.\]
But this follows from \eqref{eq:quasi-split-invariant} and assumption (5) in \S\ref{sec:division-algebra}. 
\end{proof}

\subsubsection{The PEL data $(\bbB, \dagger, \bbV, (\cdot,\cdot), h_0)$}\label{PEL_data_beta}
From now on in section $\S5$, we fix an element $\beta$ as in Lemma~\ref{lem:global-group} and we often simply write $*$, $\bbG$, $\bbG_0$ instead of $*_\beta$, $\bbG_\beta$, $\bbG_{\beta,0}$. We also fix an isomorphism $\bbB_w\cong M_m(D)^{\mathrm{op}}$ which induces isomorphisms $\bbG_{0,\bbQ_{p}}\cong \mathrm{Res}_{F/\bbQ_p}G$ and $\bbG_{\bbQ_p}\cong\mathrm{Res}_{F/\bbQ_p} G\times\bbG_m$, where here $G = {\rm GL}_m(D)$ is regarded as a group over $F \cong \bbF_{w}$. Let us recall the global setup of \S\ref{sec:global-setup} in this special case.\par  

Let $\cO_\bbB$ be a $\bbZ_{(p)}$-order in $\bbB$ such that $\cO_{\bbB}\otimes_{\bbZ_{(p)}}\cO_{\bbF_w}\cong\cO_B$. Then we have an isomorphism 
\[\cO_\bbB\otimes\bbZ_p\cong\cO_B\times\cO_B^{op}=M_m(\cO_D^{op})\times M_m(\cO_D)\]
where the first factor corresponds to $w$ and the second factor corresponds to $w^c$.\par 
Recall that in this setting $\bbV = \bbB$ is endowed with a symmetric pairing $\langle x, y \rangle = {\rm tr}_{\bbB/\bbQ}(x y^\dagger)$. In producing the positive involution $\dagger$ we may assume we started with an involution of the 2nd type $*$ on $\bbD := \bbB^{\rm op}$ and then found, using the argument mentioned in $\S\ref{sec:Shimura-data}$, an element $0 \neq \beta_0 \in \bbB^{*=-1}$ such that $x \mapsto x^\dagger = \beta_0 x^* \beta_0^{-1}$ is a positive involution on $\bbB$.  Then we use $\beta_0$ to define the $\bbQ$-valued alternating form $(\cdot, \cdot)$ on $\bbV = \bbB$, by the formula $(x, y) = {\rm tr}_{\bbB/\bbQ}(x\beta_0 y^*)$. From $\beta_0$, we then construct $\beta$ as in Lemma \ref{lem:global-group}.

From the lattice chain $\cL$ in $V\cong\bbV_{\bbF_w}$, we obtain a $(\cdot, \cdot)$-self-dual chain $\cL^+_i$ of $\cO_\bbB$-lattices in $\bbV_{\bbQ_p}=\bbV_{F_w}\oplus\bbV_{F_{w^c}}\cong V\oplus V^*$ following the method from $\S\ref{sec:integral-data}$. 

\subsubsection{The Shimura data}
We now consider a Shimura data $(\bbG, X)$, where as above $\bbG = \bbG_\beta = {\rm GU}(\bbD, *_\beta)$ and where $X$ comes from a choice of $*$-homomorphism $h_0: {\rm Res}_{\bbC/\bbR}(\bbG_m) \rightarrow \bbG_{\bbR}$, which may be constructed as in \cite[p.\,91-92]{HT01} using our choice of CM-type $\Phi$ we fixed above. The PEL data for this Shimura variety is given by the tuple $(\bbB, \dagger, \bbV = \bbB, (\cdot, \cdot), h_0)$. Then this choice of Shimura data gives rise to a conjugacy class of cocharacters

\[\xymatrix@R=1pt{
\mu_\bbC:\bbC^\times\ar[r] & \bbG(\bbC)\cong\prod_{\Phi}\mathrm{GL}_n(\bbC)\times\bbC^\times\\
z\ar@{|->}[r] & ((\mu_{\bbC,\tau}(z))_{\tau\in\Phi},z).
}\]
We can choose a representative $\mu_\bbC$ such that for any $\tau\in\Phi$,  $\mu_{\bbC,\tau}(z)=\mathrm{diag}(z,\dotsc,z,1,\dotsc,1)$ where $z$ occurs $a_\tau$ times (see Lemma \ref{lem:global-group}). The global reflex field $\bbE$ is the field of definition of the conjugacy class $\mu_{\bbC}$.\par 
The isomorphism $\iota_p:\bbC\cong\bar{\bbQ}_p$ determines a finite place $\fp$ of $\bbE$ above $p$. The completion $\bbE_\fp$ is isomorphic to the local reflex field $E$ and we fix a field isomorphism $\bbE_\fp\cong E$. Then the isomorphism $\iota_p$ takes the conjugacy class $(\mu_{\bbC,\tau}(z))_{\tau\in\Phi}$ of cocharacters of $\bbG_0(\bbC)$ to the local conjugacy class $\bar{\mu}$ of $\bbG_0(\overline{\bbQ}_p)$.

\subsection{Base change of automorphic representations}
\subsubsection{The Langlands dual group}
Let $\bbG_\bbK$ be the base change of $\bbG$ to $\bbK$ and denote $\widetilde{\bbG}:=\mathrm{Res}_{\bbK/\bbQ}\bbG_\bbK$. The complex conjugation $c\in\mathrm{Gal}(\bbK/\bbQ)$ induces an automorphism $\theta$ of $\widetilde{\bbG}$ defined over $\bbQ$ and $\bbG$ is the closed subgroup of $\widetilde{\bbG}$ consisting of $\theta$-fixed elements. There is a natural isomorphism 
\begin{equation}\label{eq:G-tilde-description}
    \widetilde{\bbG}\cong \bbB^{{\rm op},\times}\times\mathrm{Res}_{\bbK/\bbQ}\bbG_m
\end{equation}
whose composition with the embedding $\bbG\into\widetilde{\bbG}$ is given on points by $g\mapsto(g, gg^*)$ where the second factor lies in $\bbG_m\subset\mathrm{Res}_{\bbK/\bbQ}\bbG_m$. Here and throughout this subsection, we abbreviate by writing $g^*$ in place of $g^{*_\beta}$.
Using this isomorphism, the automorphism $\theta$ of $\widetilde{\bbG}$ can be described by $\theta(g,\la):=(\la^cg^{-*},\la^c)$ where $c$ denotes the automorphism of $\mathrm{Res}_{\bbK/\bbQ}\bbG_m$ induced by complex conjugation in $\mathrm{Gal}(\bbK/\bbQ)$.\par 
The Langlands dual group ${}^L\bbG=\bbG^\vee\rtimes\mathrm{Gal}(\bar\bbQ/\bbQ)$ is described explicitly as follows. Using the CM-type $\Phi$, its identity component can be described by $\bbG^\vee=\prod_{\Phi}\mathrm{GL}_n(\bbC)\times\bbC^\times$ and the action of $\sigma\in\mathrm{Gal}(\bar\bbQ/\bbQ)$ on $\bbG^\vee$ is defined by $\sigma((g_\varphi)_{\varphi\in\Phi},z)=((h_\varphi)_{\varphi\in\Phi},z')$ where
\[h_\varphi=\begin{cases}
g_{\sigma^{-1}\varphi}\quad &\text{ if }\sigma^{-1}\varphi\in\Phi\\
w_n{}^tg_{c\sigma^{-1}\varphi}^{-1}w_n^{-1}\quad &\text{ if }\sigma^{-1}\varphi\notin\Phi
\end{cases}\]
(here $w_n$ is the $n\times n$ matrix with entries $(w_n)_{ij}=(-1)^{i}\delta_{n+1-i,j}$) and  \[z'=z\prod_{\varphi\in\Phi,\sigma^{-1}\varphi\notin\Phi}\det g_\varphi.\]
Similarly, the Langlands dual group of $\widetilde{\bbG}$ is ${}^L\widetilde{\bbG}=\widetilde{\bbG}^\vee\rtimes\mathrm{Gal}(\bar{\bbQ}/\bbQ)$ where \[\widetilde{\bbG}^\vee=\prod_{\mathrm{Hom}_{\bbQ}(\bbF,\bbC)}\mathrm{GL}_n(\bbC)\times\bbC^\times\times\bbC^\times\]
and the action of $\mathrm{Gal}(\bar{\bbQ}/\bbQ)$ factors through the quotient $\mathrm{Gal}(\bbK/\bbQ)$ with the complex conjugation acting by 
\[c((g_\varphi)_\varphi,z_1,z_2)=((g_{c\varphi})_\varphi,z_2,z_1)\]
Define an $L$-embedding $\eta_{BC}:{}^L\bbG\to{}^L\widetilde{\bbG}$ by 
\[((g_\varphi)_{\varphi\in\Phi},z)\rtimes\sigma \mapsto ((h_\varphi)_{\varphi\in\Phi\sqcup c\Phi}, z,z\prod_{\varphi\in\Phi}\det(g_\varphi))\rtimes\sigma\]
where 
\[h_\varphi=\begin{cases}g_\varphi\quad&\text{ if }\varphi\in\Phi\\
w_n{}^tg_{c\varphi}^{-1}w_n^{-1}\quad&\text{ if }\varphi\notin\Phi.
\end{cases}\]
Our goal is to define base change of automorphic representations of $\bbG$ to $\widetilde{\bbG}$ for the $L$-embedding $\eta_{BC}$.

\subsubsection{Local base change at split primes}
Let $x$ be a prime of $\bbQ$ that splits in $\bbK$ as $x=yy^c$. The canonical isomorphism $\bbQ_x\cong\bbK_y$ allows us to view $\bbQ_x$ as a $\bbK$-algebra and from \eqref{eq:G-tilde-description} we obtain an isomorphism
\begin{equation}\label{eq:G(Q_x)-description-split-prime}
    \bbG(\bbQ_x)\cong(\bbB_y^{{\rm op}})^\times\times\bbQ_x^\times.
\end{equation}

Let $\pi$ be an irreducible admissible representation of $\bbG(\bbQ_x)$. Then we can decompose it as $\pi\cong\pi_y\otimes\chi_x$ where $\pi_y$ is an irreducible admissible representation of $(\bbB_y^{op})^\times$ and $\chi_x$ is a character of $\bbQ_x^\times$. Let $\omega_\pi$ be the central character of $\pi$ and decompose its restriction to $\bbK_x^\times\cong\bbK_y^\times\times\bbK_{y^c}^\times$ as $\omega_{\pi,y}\otimes\omega_{\pi,y^c}$. Similarly, let $\omega_{\pi_y}$ be the central character of $\pi_y$.
\begin{lem}
With notations as above, we have $\chi_x=\omega_{\pi,y^c}$ and $\omega_{\pi,y}=\omega_{\pi_y}\omega_{\pi,y^c}$, all viewed as characters of $\bbQ_x^\times$ via the isomorphisms $\bbQ_x\cong\bbK_y\cong\bbK_{y^c}$. In particular, we have $\pi\cong\pi_y\otimes\omega_{\pi,y^c}$. 
\end{lem}
\begin{proof}
Any $(a_1,a_2)\in\bbK_x^\times\cong\bbK_y^\times\times\bbK_{y^c}^\times$ embedded in the center of $\bbG(\bbQ_x)$ acts on the representation $\pi$ by the scalar $\pi(a_1,a_2)=\omega_{\pi,y}(a_1)\omega_{\pi,y^c}(a_2)$. On the other hand, under the isomorphism \eqref{eq:G(Q_x)-description-split-prime}, $(a_1,a_2)$ is sent to $(a_1,a_1a_2)\in (\bbB_y^{op})^\times\times\bbQ_x^\times$ and hence from the decomposition $\pi\cong\pi_y\otimes\chi_x$ we get $\pi(a_1,a_2)=\omega_{\pi_y}(a_1)\chi_x(a_1a_2)$. Consequently, after the identifications $\bbQ_x\cong\bbK_y\cong\bbK_{y^c}$ we obtain
\[\omega_{\pi,y}(a_1)\omega_{\pi,y^c}(a_2)=\omega_{\pi_y}(a_1)\chi_x(a_1a_2),\quad\forall a_1,a_2\in\bbQ_x\]
Let $a_1=1$ we get $\chi_x=\omega_{\pi,y^c}$ and let $a_2=1$ we get $\omega_{\pi,y}=\omega_{\pi_y}\omega_{\pi,y^c}$.
\end{proof}
We define the local base change of $\pi$ to be the representation
\[BC(\pi)=\pi_y\otimes\pi_y^\#\otimes(\omega_{\pi,y^c}\circ c)\otimes(\omega_{\pi,y}\circ c)\]
of
\[\widetilde{\bbG}(\bbQ_x)=\bbG(\bbK_x)\cong (\bbB_y^{op})^\times\times(\bbB_{y^c}^{op})^\times\times\bbK_y^\times\times\bbK_{y^c}^\times\]
where $\pi_y^\#(g):=\pi_y(g^{-*})$.\par 
Let $w_1,\dotsc,w_s$ be the primes of $\bbF$ above $y$. Then we have a decomposition $\bbB_y=\prod_{i=1}^s\bbB_{w_i}$ and $\pi_y=\otimes_{i=1}^s\pi_{w_i}$ where $\pi_{w_i}$ is a representation of $(\bbB_{w_i}^{{\rm op}})^\times$. Consequently corresponding to the decomposition
\[\widetilde{\bbG}(\bbQ_x)\cong\prod_{i=1}^s(\bbB_{w_i}^{{\rm op}})^\times\times\prod_{i=1}^s(\bbB_{w_i^c}^{{\rm op}})^\times\times\bbK_y^\times\times\bbK_{y^c}^\times\]
we can write
\[BC(\pi)=\otimes_{i=1}^sBC(\pi)_{w_i}\otimes\otimes_{i=1}^sBC(\pi)_{w_i^c}\otimes\chi_{BC(\pi),y}\otimes \chi_{BC(\pi),y^c}\]
where $BC(\pi)_{w_i}:=\pi_{w_i}$, $BC(\pi)_{w_i^c}:=\pi_{w_i}^{\#}$, $\chi_{BC(\pi),y}:=\omega_{\pi,y^c}\circ c$ and $\chi_{BC(\pi),y^c}:=\omega_{\pi,y}\circ c$.

\subsubsection{Unramified local base change at inert primes}
Now let $x$ be a prime of $\bbQ$ such that
\begin{itemize}
    \item $x$ is inert in $\bbK$ and unramified in $\bbF$;
    \item $(\bbB_x^{op},*)\cong(M_n(\bbF_x),\dagger)$ where $g^\dagger:=w_n{}^tg^cw_n^{-1}$.
\end{itemize}
The isomorphism $\iota_x:\bbC\xrightarrow{\sim}\bar{\bbQ}_x$ induces an embedding from the Weil group $W_{\bbQ_x}$ into $\mathrm{Gal}(\bar{\bbQ}/\bbQ)$ (here $\bar{\bbQ}$ is the algebraic closure of $\bbQ$ in $\bbC$). We use $\iota_x$ to identify the CM-type $\Phi$ with the subset of $\mathrm{Hom}_{\bbQ_x}(\bbF_x,\bar{\bbQ}_x)$ consisting of homomorphisms $\varphi$ such that $\varphi|_{\bbK}=\iota_x\circ\tau_\bbK$.\par
The group $\bbG(\bbQ_x)$ is unramified and contains a standard maximal torus $T_x$ and Borel subgroup $B_x$ defined over $\bbQ_x$, consisting of diagonal and upper triangular matrices. The torus $T_x$ can be described as
\[T_x=\{(t_1,\dotsc,t_n;t_0)\in(\bbF_x^\times)^n\times\bbQ_x^\times | t_0=t_it_{n+1-i}^c,\forall i=1,\dotsc,n\}.\]
Similarly, $\widetilde{\bbG}_x\cong\mathrm{GL}_n(\bbF_x)\times\bbK_x^\times$ contains the diagonal maximal torus $\widetilde{T}_x\cong(\bbF_x^\times)^n\times\bbK_x^\times$ and standard Borel subgroup $\widetilde{B}_x=B_n(\bbF_x)\times\bbK_x^\times$ where $B_n(\bbF_x)$ is the group of upper triangular matrices in $\mathrm{GL}_n(\bbF_x)$. The automorphism $\theta$ on $\widetilde{\bbG}$ induces automorphisms of $\widetilde{T}_x$ and $\widetilde{B}_x$ so that $T_x=\widetilde{T}_x^\theta$ and $B_x=\widetilde{B}_x^\theta$ are identified as the fixed point subgroups.\par
For any smooth irreducible unramified representation $\pi$ of $\bbG(\bbQ_x)$, there is an unramified charcater $\chi$ of $T_x$ such that $\pi$ is the unique unramified subquotient of the normalized parabolic induction $i_{B_x}^{\bbG(\bbQ_x)}\chi$. The character $\chi$ corresponds to an element $\alpha=((\alpha_{1,\varphi},\dotsc,\alpha_{n,\varphi})_{\varphi\in\Phi};\alpha_0)\in\bbT^\vee=\prod_{\varphi\in\Phi}(\bbC^\times)^n\times\bbC^\times$. More precisely, if $t=(t_1,\dotsc,t_n;t_0)\in T_x$, then 
\[\chi(t)=\alpha_0^{\mathrm{val}_x(t_0)}\prod_{\varphi\in\Phi}\prod_{i=1}^n(\alpha_{i,\varphi})^{\mathrm{val}_x(\iota_x\circ\varphi(t_{i}))}\] 
%{\cb: \mbox{Should $t_{i, \varphi}$ simply be $t_i$?}} Jingren: Yes, I fixed it.
where $\mathrm{val}_x$ denotes the $x$-adic valuation on $\bar{\bbQ}_x$.\par 
The Langlands parameter for $\pi$ is the conjugacy class of unramified $L$-homomorphism $\varphi_\pi:W_{\bbQ_x}\to\bbG^\vee\rtimes W_{\bbQ_x}$ such that $\varphi_\pi(\mathrm{Frob}_x)=\alpha\rtimes\mathrm{Frob}_x$ for any geometric Frobenius $\mathrm{Frob}_x\in W_{\bbQ_x}$. The base change $BC(\pi)$ is the unramified representation of $\widetilde{\bbG}_x\cong\mathrm{GL}_n(\bbF_x)\times\bbK_x^\times$ whose Langlands parameter is $\eta_{BC}\circ\varphi$. In other words, $BC(\pi)$ is the unique irreducible unramified subquotient of the normalized induction $i_{\widetilde{B}_x}^{\widetilde{\bbG}_x}\tilde{\chi}$ where $\tilde{\chi}$ is the character of $\widetilde{T}_x$ corresponding to $\eta_{BC}(\alpha)\in\bbT^\vee$. More explicitly, for any $(t_1,\dotsc,t_n;t_0)\in(\bbF_x^\times)^n\times\bbK_x^\times$, we have
\[\tilde{\chi}(t_1,\dotsc,t_n;t_0)=\chi(t_0^ct_1/t_n^c,\dotsc,t_0^ct_n/t_1^c;t_0t_0^c).\]
Let $w_1,\dotsc,w_s$ be the primes of $\bbF$ above $x$. For each $i=1,\dotsc,s$, let $\Phi_{w_i}:=\{\varphi\in\Phi,\varphi|_{\bbF_{w_i}}\ne0\}$ so that $\Phi=\sqcup_{i=1}^s\Phi_{w_i}$. Corresponding to the decomposition $\bbF_x=\prod_{i=1}^s\bbF_{w_i}$ we have an isomorphism
\[\widetilde{\bbG}_x\cong\prod_{i=1}^s\mathrm{GL}_n(\bbF_{w_i})\times\bbK_x^\times\]
which induces a decomposition
\[BC(\pi)=\otimes_{i=1}^s BC(\pi)_{w_i}\otimes\chi_{BC(\pi)}\]
where 
\begin{itemize}
    \item For each $i=1,\dotsc,s$, $BC(\pi)_{w_i}$ is an unramified representation of $\mathrm{GL}_n(\bbF_{w_i})$ with Satake parameter 
    \[(\prod_{\varphi\in\Phi_{w_i}}\alpha_{1,\varphi}\alpha_{n,\varphi}^{-1},\dotsc,\prod_{\varphi\in\Phi_{w_i}}\alpha_{n,\varphi}\alpha_{1,\varphi}^{-1}),\]
    where the $j$-th coordinate is $\prod_{\varphi\in\Phi_{w_i}}\alpha_{j,\varphi}\alpha_{n+1-j,\varphi}^{-1}$ for any $1\le j\le n$;
    \item $\chi_{BC(\pi)}$ is the unramified character of $\bbK_x^\times$ sending a uniformizer to $\alpha_0^2\prod_{i=1}^n\prod_{\varphi\in\Phi}\alpha_{i,\varphi}$.
\end{itemize}

\subsubsection{Global base change}
Let $\xi$ be an irreducible algebraic representation of $\bbG(\bbC)$ on a finite dimensional $\bbC$-vector space. Recall that we have a decomposition
\[\widetilde{\bbG}(\bbC)=\bbG(\bbK\otimes_\bbQ\bbC)\cong\bbG(\bbC)\times\bbG(\bbC)\]
where the first factor corresponds to the embedding $\tau_\bbK$ and the second factor corresponds to $\tau_\bbK^c$. With this identification, we define $\xi_\bbK$ to be the representation $\xi\otimes\xi$ of $\widetilde{\bbG}(\bbC)$.

\begin{thm}\label{thm:global-base-change}
Let $\pi$ be an irreducible automorphic representation of $\bbG(\bbA)$ whose archimedean component $\pi_\infty$ is cohomological for $\xi$. Let $\omega_\pi$ denote the central character of $\pi$. Then there is a unique irreducible automorphic representation $\mathrm{BC}(\pi)=(\Pi,\omega)$ of $\widetilde{\bbG}(\bbA)=(\bbB^{op}\otimes_\bbQ\bbA)^\times\times\bbA_\bbK^\times$ such that
\begin{enumerate}
    \item $\omega=\omega^c_\pi|_{\bbA_\bbK^\times}$ and $\omega_\infty^c=(\xi|_{\bbK_\infty^\times})^{-1}$
    \item For any place $x$ of $\bbQ$ that splits in $\bbK$, $\mathrm{BC}(\pi)_x=\mathrm{BC}(\pi_x)$
    \item For all but finitely many places $x$ of $\bbQ$ that are inert in $\bbK$, we have $\mathrm{BC}(\pi)_x=\mathrm{BC}(\pi_x)$;
    \item $\Pi_\infty$ is cohomological for $\xi_\bbK$;
    \item If $\omega_\Pi$ is the central character of $\Pi$, then $\omega_\Pi|_{\bbA_\bbK^\times}=\omega^c\omega^{-1}$;
    \item $\Pi^\#\cong\Pi$ where $\Pi^\#(g):=\Pi(g^{-*})$.
\end{enumerate}
\end{thm}
\begin{proof}
This is basically Theorem VI.2.1 of \cite{HT01}, which improves upon Theorem A.5.2 of \cite{La99}. We briefly review the argument here.\par 
First one reduces to base change for the unitary group $\bbG_1$ instead of the unitary similitude group $\bbG$. More precisely, one shows that there exists an automorphic representation $\pi'$ of $\bbG_1(\bbA)$ such that 
\begin{itemize}
    \item For any place $x$ of $\bbQ$ that splits in $\bbK$, $\pi'_x\cong\pi_x|_{\bbG_1(\bbQ_x)}$
    \item For all but finitely many places $x$ of $\bbQ$ that are inert in $\bbK$ and unramified in $\bbF$, $\pi'_x$ is the unique unramified subquotient of $\pi_x|_{\bbG_1(\bbQ_x)}$.
\end{itemize}
Then one applies Theorem A.5.2 of \cite{La99} (which is slightly generalized in the proof of Theorem VI.2.1 of \cite{HT01}) to $\pi'$ to obtain $\Pi$. The crucial ingredient in this result is a comparison of a simple twisted trace formula for $\widetilde{\bbG}_1:=\mathrm{Res}_{\bbK/\bbQ}\bbG_{1,\bbK}=\bbB^{op,\times}$ and a simple trace formula for $\bbG_1$. The comparison goes by equating both trace formulas with the stable trace formula for the quasi-split inner form $\bbG_1^*$ with suitable choice of test functions. More precisely, we choose $\phi=\otimes_v\phi_v\in C_c^\infty(\widetilde{\bbG}_1)$ and $f=\otimes_vf_v\in C_c^\infty(\bbG_1)$ satisfying:
\begin{itemize}
    \item For each place $v$, $\phi_v$ and $f_v$ have matching stable (twisted) orbital integrals;
    \item $\phi_\infty$ is the Lefschetz function assosciated to $\tilde{\xi}$ and the automorphism $\theta$ on $\widetilde{\bbG}_{1,\infty}$ and $f_\infty$;
    \item $f_\infty$ is, up to a positive constant multiple, the Euler-Poincar\'e function associated to $\xi$.
\end{itemize}
It follows from Theorem A.1.1 and Corollaire 1.2 of \cite{La99} that $\phi_\infty$ and $f_\infty$ have matching stable (twisted) orbital integrals and the unstable twisted orbital integrals of $\phi_\infty$ vanishes. We remark that in \cite{La99} it is assumed that $\xi$ is trivial and this assumption is removed in \cite{CL2} (which also corrects some sign errors in \cite{La99}). The vanishing of unstable twisted orbital integrals of $\phi_\infty$ implies that the twisted trace formula for $\widetilde{\bbG}_1$ equals a stable trace formula for $\bbG_1^*$. From the assumption that $\bbB_x$ is a division algebra for some place $x$ of $\bbF$ that splits over $\bbF_0$, it follows by an observation of Kottwitz that the trace formula of $\bbG_1$ equals the stable trace formula for $\bbG_1^*$. Combining these two comparisons one obtains the sought-after comparison of simple (twisted) trace formulas of $\widetilde{\bbG}_1$ and $\bbG_1$ as in Theorem A.3.1 of \cite{La99}.\par 
Using this comparison one proceeds as in the proof of Theorem VI.2.1 of \cite{HT01} to get the automorphic representation $\Pi$ satisfying the desired properties. We remark that property (2) is not explicitly stated in \cite{La99} but follows from the generalization of its arguments in \cite{HT01}.
\end{proof}

\subsection{Description of Cohomology}
The following result establishes the existence of Galois representations (satisfying suitable local-global compatibility conditions) associated to cohomological automorphic representations of $\bbG$.
\begin{thm}\label{thm:automorphic-to-Galois}
Let $\pi$ be an irreducible automorphic representation of $\bbG(\bbA)$ such that $\pi_\infty$ is cohomological for $\xi$. There exists an $n$-dimensional $\ell$-adic Galois representation $\sigma_\ell(\pi)$ of $\mathrm{Gal}(\bar{\bbF}/\bbF)$ such that 
\begin{enumerate}
    \item For any place $w$ of $\bbF$ that splits over $\bbF_0$, the restriction of $\sigma_\ell(\pi)$ to the Weil group of $\bbF_w$ is isomorphic to the $\ell$-adic representation associated to the irreducible smooth representation $\pi_w$ of $\bbB_w^{op,\times}$ under the local Langlands correspondence.
    \item For all but finitely many places $w$ of $\bbF$ such that $w$ is inert over $\bbF_0$ and unramified over $\bbQ$ and $\pi$ is unramified at $x=w|_\bbQ$, the restriction of $\sigma_\ell(\pi)$ to the Weil group of $\bbF_w$ is the unramified representation corresponding to the irreducible smooth representation $BC(\pi_x)_w$ of $\mathrm{GL}_n(\bbF_{w})$.
\end{enumerate}
\end{thm}
\begin{proof}
Consider the base change $BC(\pi)=(\Pi,\omega)$ from Theorem~\ref{thm:global-base-change}. Let $JL(\Pi)$ be the automorphic representation of $\mathrm{GL}_n(\bbA_\bbF)$ associated to $\Pi$ by the global Jacquet-Langlands correspondence established in \cite{Badu} (note that $\Pi$ is automatically cuspidal). Then $JL(\Pi)$ is conjugate self-dual and the archimedean component $JL(\Pi)_\infty$ is regular algebraic. Moreover, by our choice of $\bbB$, there is a place $x$ of $\bbF$ that splits over $\bbF_0$ such that $\bbB_x$ is a division algebra, hence $JL(\Pi)_x$ is square-integrable. Thus we can apply Theorem VII.1.9 of \cite{HT01} to $JL(\Pi)$. In the notation of \emph{loc.cit.} we obtain an $\ell$-adic Galois representation $\sigma_\ell(\pi):=R_\ell(JL(\Pi))$ that satisfies the conditions. 
\end{proof}

We will apply the above to get an understanding of the cohomology of the Shimura variety $Sh(\bbG_\beta, X)$ described in $\S\ref{PEL_data_beta}$. This is another simple Shimura variety of the type considered by Kottwitz, but unlike the original general set-up we started with, the nature of the group $\bbG = \bbG_\beta$ in this special case allows us to exploit the Galois representations mentioned above.\par
\iffalse
{\cg Tom: I think we might wish to introduce $\beta$ into the notation here, to emphasize that now we are exploiting the special features of the Shimura variety $Sh(\bbG_\beta, X)$.  For example, we could write $H^*_{\beta, \xi}$ instead of $H^*_\xi$ everywhere from here until the end of the section 5.3.} {\cc (Jingren: I have included $\beta$ in the notation at several places.)}
\fi

Let $H^*_{\beta,\xi}:=H^*(Sh(\bbG_\beta, X)_{\overline{\bbQ}},\cL_\xi)$. We have a decomposition \[H^*_{\beta,\xi}=\oplus_{\pi_f}\pi_f\otimes\sigma_{\beta,\xi}(\pi_f)\] 
where $\sigma_{\beta,\xi}(\pi_f)$ is a virtual representation of $\mathrm{Gal}(\overline{\bbQ}/\bbE)$.

\begin{cor}\label{cor:reciprocity}
In the setting of \S\ref{sec:global-data}, there is an isomorphism of virtual $W_{E} = W_{\bbE_\fp}$ representations
\[\sigma_{\beta,\xi}(\pi_f)|_{W_{\bbE_\fp}}\cong a(\pi_f)(r_{-\mu}\circ\varphi_{\pi_p}|_{W_{\bbE_\fp}})|\cdot|^{-\dim\Sh(\bbG_\beta,X)/2}.\]
\end{cor}
\begin{proof}
Let $\widetilde{\bbE}$ be the composite of $\bbE$ and the imaginary quadratic field $\bbK$ in $\bbC$ where we use the embedding $\tau_\bbK$ fixed in \S~\ref{sec:CM-field} to view $\bbK$ as a subfield of $\bbC$. Since $p$ is split in $\bbK$, there is a place $\tilde{\fp}$ of $\widetilde{\bbE}$ above $\fp$ and an isomorphism  $\widetilde{\bbE}_{\tilde{\fp}}\cong\bbE_\fp$. 
Using Theorem~\ref{thm:automorphic-to-Galois} and the arguments in the proof of \cite[Th\`eorem\'e A.7.2]{Far04}, one can construct an $\ell$-adic representation of $\mathrm{Gal}(\overline{\bbQ}/\widetilde{\bbE})$ that agrees with $\sigma_{\beta,\xi}(\pi_f)$ at all unramified places, then we can conclude by the Chebotarev density theorem. See also \cite[proof of Theorem 8.1]{Sch-Shin}. \par
Let us explain the arguments in more detail. Let $\sigma_\ell(\pi)$ be the $n$-dimensional representation of $\mathrm{Gal}(\bar{\bbF}/\bbF)$ from Theorem \ref{thm:automorphic-to-Galois}. Under the Shapiro isomorphism, $\sigma_\ell(\pi)$ corresponds to a (conjugacy class of) homomorphism(s):
\[\varphi_\ell(\pi):\mathrm{Gal}(\overline{\bbQ}/\bbK)\to\mathrm{GL}_n(\overline{\bbQ}_\ell)^{\mathrm{Hom}_{\bbK}(\bbF,\overline{\bbQ})}\rtimes\mathrm{Gal}(\overline{\bbQ}/\bbK)=\mathrm{GL}_n(\overline{\bbQ}_\ell)^\Phi\rtimes\mathrm{Gal}(\overline{\bbQ}/\bbK).\]
Let $\omega_\pi$ be the central character of $\pi$ and let $\omega=\omega_\pi^c|_{\bbA_{\bbK}^\times}$ be as in Theorem \ref{thm:global-base-change}. Then $\omega$ is an algebraic Hecke character and let $\sigma_\ell(\omega):\mathrm{Gal}(\overline{\bbQ}/\bbK)\to\bar{\bbQ}_\ell^\times$ be the associated (one-dimensional) $\ell$-adic Galois representation. Combining the restriction of $\varphi_\ell(\pi)$ and $\sigma_\ell(\omega)$ to $\mathrm{Gal}(\overline{\bbQ}/\widetilde{\bbE})$ we get a homomorphism 
\[\tilde{\varphi}_\ell(\pi):\mathrm{Gal}(\overline{\bbQ}/\widetilde{\bbE})\to{}^LG_\ell=(\mathrm{GL}_n(\overline{\bbQ}_\ell)^{\Phi}\times\overline{\bbQ}_\ell^\times)\rtimes\mathrm{Gal}(\overline{\bbQ}/\widetilde{\bbE}).\]
By construction, for any finite prime $\tilde{\fq}$ of $\widetilde{\bbE}$ with residue characteristic $q$, we have \[\tilde{\varphi}_\ell(\pi)|_{\widetilde{\bbE}_{\tilde{\fq}}}\cong\varphi_{\pi_q}|_{\widetilde{\bbE}_{\tilde{\fq}}}\]
Consider the virtual representation $R(\pi)$ of $\mathrm{Gal}(\overline{\bbQ}/\widetilde{\bbE})$ defined by 
\[R(\pi):=a(\pi_f)(r_{-\mu}\circ\tilde{\varphi}_\ell(\pi))|\cdot|^{-\dim\Sh(\bbG_\beta,X)/2}.\]
For any prime $\tilde{\fq}$ of $\widetilde{\bbE}$ with residue characteristic $q$ such that $\pi_q$ is unramified and the Shimura variety has good reduction at $q$, we have $\sigma_{\beta,\xi}(\pi_f)|_{\widetilde{\bbE}_{\tilde{\fq}}}\cong R(\pi)|_{\widetilde{\bbE}_{\tilde{\fq}}}$ by Kottwitz \cite{Ko-simple}. By the Chebotarev density theorem we get $\sigma_{\beta,\xi}(\pi_f)|_{\mathrm{Gal}(\overline{\bbQ}/\widetilde{\bbE})}\cong R(\pi)$ (at least when restricted to all local Weil groups and after semisimplifying the Weil group actions). In particular restricting to the Weil group of $\widetilde{\bbE}_{\tilde{\fp}}\cong\bbE_\fp$ we get the desired isomorphism.
\end{proof}

Recall from Definition~\ref{def:z_{tau,mu}} the element $z_{\tau,-\mu}$ in the stable Bernstein center of $\bbG(\bbQ_p)$.

\begin{cor}\label{cor:trace-G}
Let $r\ge1$ be an integer and $\tau\in\mathrm{Frob}_E^rI_E\subset W_E$, where $E\cong\bbE_\fp$ is the local reflex field. For any $h\in C_c^\infty(\cG_{\cL}(\bbZ_p))$ and $f^p\in C_c^\infty(\bbG(\bbA_f^p))$, we have 
\[\mathrm{tr}(\tau\times hf^p| H_{\beta,\xi}^*)=\sum_\pi m(\pi)\mathrm{Tr}(f_{\tau,h}f^p f_{\xi,\infty}^\bbG|\pi)\]
where $f_{\tau,h}:=z_{\tau,-\mu}*h$ and the sum ranges over automorphic representations of $\bbG$.
\end{cor}
\begin{proof}
By Corollary~\ref{cor:reciprocity}, the left hand side is equal to 
\[\sum_{\pi_f}a(\pi_f)\mathrm{Tr}(f_{\tau,h}f^p|\pi_f).\]
This is equal to the right hand side by the definition of $a(\pi_f)$, see \eqref{eq:pi-f}.
\end{proof}

\subsection{Fixed points via generalized Kottwitz triples}
We would like to proceed from Corollary~\ref{cor:trace-coarse-expansion} and find an expression of the right hand side in terms of a trace formula,  without knowing the vanishing property of the test function $\phi_{\tau,h}$. For this purpose, we will associate to each fixed point of the Frobenius-Hecke correspondence a substitute of the Kottwitz triple defined using an auxiliary group $\bbG'$ that is constructed in a similar way to $\bbG$.
\subsubsection{An auxiliary group}\label{sec:auxiliary-group}
Let $\bbB'$ be the division algebra over $\bbF$ that is split at $w,w^c$ and isomorphic to $\bbB=\bbD^{\mathrm{op}}$ at all other finite places of $\bbF$. Choose a \emph{positive} involution $\dagger'$ of the second kind on $\bbG'$. Analogously to the construction of $\bbG$ in Lemma \ref{lem:global-group}, we can find an element $0\ne\beta'\in\bbB'$ that, together with the positive involution $\dagger'$, gives rise to an involution $*'$, a $\dagger'$-Hermitian pairing on $\bbV'=\bbB'$, and a group $\bbG'$ such that
\begin{itemize}
    \item $\bbG'_{\bbQ_p}\cong\mathrm{Res}_{F/\bbQ_p}\mathrm{GL}_n\times\bbG_m$ is the quasi-split inner form of $\bbG_{\bbQ_p}$;
    \item For any finite prime $\ell\ne p$, we have $\bbG'_{\bbQ_\ell}\cong\bbG_{\bbQ_\ell}$;
\end{itemize}
Note that we do not put any condition at the archimedean places.

\subsubsection{Generalized Kottwitz triples}
\begin{defn}\label{def:generalized-Kott-triple}
Let $j\in\bbZ_{\ge1}$ and set $r:=j[\kappa_{\bbE_\fp}:\bbF_p]$. A degree $j$ \emph{generalized Kottwitz triple} $(\ga_0';\ga,\delta)$ consists of 
\begin{itemize}
    \item A semisimple stable conjugacy class $\ga_0'\in\bbG'(\bbQ)$ that is elliptic in $\bbG'(\bbR)$,
    \item a conjugacy class $\ga\in\bbG'(\bbA_f^p)=\bbG(\bbA_f^p)$ that is stably conjugate to $\ga_0'$,
    \item a $\sigma$-conjugacy class $\delta\in\bbG(\bbQ_{p^r})$ whose naive norm $N_r\delta=\delta\sigma(\delta)\dotsm\sigma^{r-1}(\delta)$ is stably conjugate to $\ga_0'$, such that $\kappa_{\bbG_{\bbQ_p}}(p\delta)=-\mu^\#$ in $X^*(Z(\hat{\bbG})^{\Gamma(p)})$, where $\Gamma(p):=\mathrm{Gal}(\overline{\bbQ}_p/\bbQ_p)$.
\end{itemize}
\end{defn}
\begin{rem}
    Recall that $\bbG(\bbQ_{p^r})\cong(\bbD_w\otimes_{\bbQ_p}\bbQ_{p^r})^\times\times\bbQ_{p^r}^\times$. We call the image of $N_r\delta$ in $\bbQ_{p^r}^\times$ its similitude factor. On the other hand, the first factor of $N_r\delta$ has a characteristic polynomial which is a monic polynomial of degree $(\dim_{\bbF}\bbD)^{\frac{1}{2}}$ with coefficients in $F\otimes_{\bbQ_p}\bbQ_{p^r}$ (recall that $\bbD_w$ has center $\bbF_w\cong F$). We refer to this simply as the characteristic polynomial of $N_r\delta$. Similarly, for any element in $\bbG'(\bbQ_p)\cong\mathrm{GL}_n(F)\times\bbQ_p^\times$ we have its characteristic polynomial (defined by the first factor) and its similitude factor. In the definition of generalized Kottwitz triples, the condition that $N_r\delta\in\bbG(\bbQ_{p^r})$ is stably conjugate to $\ga_0'\in\bbG'(\bbQ_p)$ simply means that they have the same characteristic polynomial and similitude factor.
\end{rem}

Let $x=(A_\bullet,\la,\bar{\eta})\in\mathrm{Fix}_{j,\cL}(g^p)$ be a fixed point in the closure of the generic fiber as in \S\ref{sec:fixed-points-of-correspondence}. Then we have associated to $x$ a polarized virtual $\cO_\bbB$-abelian variety $(A,u,\la)$ over $\bbF_{p^r}$, giving rise to a
conjugacy class $\ga\in\bbG(\bbA_f^p)=\bbG'(\bbA_f^p)$ and a $\sigma$-conjugacy class $\delta\in\bbG(\bbQ_{p^r})$.

\begin{lem}\label{lem:generalized-Kott-triple}
There exists a unique semi-simple stable conjugacy class $\ga_0'\in\bbG'(\bbQ)$ such that $(\ga_0';\ga,\delta)$ forms a degree $j$ generalized Kottwitz triple. 
\end{lem}
\begin{proof}
The arguments in \cite[\S14]{Ko-points} go through in our situation with slight modifications. \par 
First we construct a candidate $\ga_0'\in\bbG'(\bbQ)$. Starting with a maximal semisimple subalgebra $N$ of $\mathrm{End}_\bbB(A)$ containing the Frobenius endomorphism $\pi_A$, we seek an algebra embedding from $N$ into $C':=\mathrm{End}_{\bbB'}(\bbV')=(\bbB')^{\mathrm{op}}$ compatible with the Rosati involution $*_\la$ on $N$ and the involution $*'$ on $C'$, and take $\ga_0'$ to be the image of $\pi_A^{-1}$ under this embedding. To show the existence of such an embedding, one replaces the algebra $C$ in \emph{loc.cit.} by our $C'$ everywhere in pages 420-422. This works in our setting because $\bbB'$ and hence $C'$ splits at any place of $\bbF$ above $p$.\par 
Next we check that $(\ga_0';\ga,\delta)$ forms a generalized Kottwitz triple. For each $\ell\ne p$, over the algebraic closure $\overline{\bbQ}_\ell$ we have an isomorphism $\bbG'(\bar{\bbQ}_\ell)\cong\mathrm{GL}_n(\bar{\bbQ}_\ell)^{\mathrm{Hom}(\bbF_0,\bar{\bbQ}_\ell)}\times\bar{\bbQ}_\ell^\times$. Under this isomorphism, each $\mathrm{GL}_n$ factor of $\ga_\ell$ and $\ga_0'$ have the same characteristic polynomial, which is the image of the characteristic polynomial of $\pi_A^{-1}$ in the semisimple $\bbF$-algebra $N$ under suitable embeddings of $\bbF$ into $\bar{\bbQ}_\ell$. Also they have the same similitude factor in $\bar{\bbQ}_\ell^{\times}$, which equals to $(\pi_A\pi_A^{*_\la})^{-1}$. 
%{\cg Inverse correct?}. Jingren: yes I have corrected it.
Hence $\ga$ and $\ga_0'$ are stably conjugate in $\bbG'(\bbQ_\ell)=\bbG(\bbQ_\ell)$.\par 
At $p$, it is clear that the naive norm $N_r\delta\in\bbG(\bbQ_{p^r})$ and $\ga_0'\in\bbG'(\bbQ_p)$ have the same similitude factor, which equals to $(\pi_A\pi_A^{*_\la})^{-1}$. Next we recall that there is an $\bbB\otimes_{\bbQ}\bbQ_{p^r}$-module $H_r$ such that $H_r\otimes_{\bbQ_{p^r}}\Breve{\bbQ}_p$ is isomorphic to the covariant rational Dieudonn\'e module of the $p$-divisible group $A[p^\infty]$. The element $\delta\in\bbG(\bbQ_{p^r})$ is defined by choosing an isomorphism of $\bbB\otimes_{\bbQ}\bbQ_{p^r}$-modules $H_r\cong\bbV\otimes_{\bbQ}\bbQ_{p^r}$ (which exists by the assumption that the fixed point $x=(A_\bullet,\la,\bar{\eta})$ lies in the closure of the generic fiber). Under the natural embedding $\mathrm{End}_\bbB(A)\to\mathrm{End}_\bbB(H_r)\cong\mathrm{End}_\bbB(\bbV\otimes\bbQ_{p^r})$ we have $N_r\delta=\pi_A^{-1}$. So the characteristic polynomials of $N_r\delta$ and $\ga_0'$ both equal to the characteristic polynomial of $\pi_A^{-1}$ in the semisimple $\bbF$-algebra $N$. Hence $N_r\delta$ and $\ga_0'$ are stably conjugate.\par 
Finally, the fact that $\kappa_{\bbG_{\bbQ_p}}(p\delta)=-\mu^{\#}$ follows from the determinant condition for the fixed point $x=(A_\bullet,\la,\bar{\eta})$.
\end{proof}

\subsubsection{Virtual abelian varieties constructed from generalized Kottwitz triples}\label{sec:ab-var-associated-to-generalized-triple}
As a converse to Lemma \ref{lem:generalized-Kott-triple}, starting from a generalized Kottwitz triple $(\gamma_0';\gamma,\delta)$, we will contruct a $c$-polarized virtual $\bbB$-abelian variety up to isogeny. The crucial ingredient will be Lemma~\ref{lem:Kottwitz-simple}, which implies vanishing of the Kottwitz invariant in our setting.\par 
Let $I_0$ be the centralizer of $\ga_0'$ in $\bbG'$. It is a connected reductive group over $\bbQ$ with simply connected derived group. As usual, we let $\hat{I}_0$ be the neutral component of the Langlands dual group of $I_0$ and put $\Gamma:=\mathrm{Gal}(\overline{\bbQ}/\bbQ)$. Recall that we have canonical $\Gamma$-equivariant isomorphism $Z(\hat{\bbG'})=Z(\hat{\bbG})$, which leads to a canonical $\Gamma$-equivariant embedding $Z(\hat{\bbG})\subset Z(\hat{I}_0)$.\par 
Using Honda-Tate theory, one constructs from $\gamma_0'$ a $c$-polarizable virtual $\bbB$-abelian variety $(A,i)$ over $\bbF_{p^r}$ up to isogeny as in \cite[p.438-439]{Ko-points}, with slight modifications. More precisely, let $f\in\bbF[T]$ be the characteristic polynomial of $(\gamma_0')^{-1}\in\bbB'$. For any root $\alpha$ of $f$ in $\bar{\bbF}$, one first checks that $\alpha$ is a $c$-number in the sense of \emph{loc. cit.} page 405, paragraph above Lemma 10.4, where $c=(\gamma_0'\gamma_0'^*)^{-1}$. Then one checks that the multiplicity $m(\alpha)$ of the root $\alpha$ is divisible by $(\dim_\bbF\bbB)^{\frac{1}{2}}$ and the quotient annihilates certain invariants at each place $v$ of $\bbF[\alpha]$.
The arguments on \emph{loc. cit.} page 438 go through in our slightly modified setting. Indeed, if $v$ is a place above a prime $\ell\ne p$, then we have $\bbB'\otimes_\bbF\bbF[\alpha]_v\cong\bbB\otimes_\bbF\bbF[\alpha]_v$; whereas if $v$ divides $p$, we use that $N_r\delta\in\bbG(\bbQ_{p^r})$ has the same characteristic polynomial as $\gamma_0'\in\bbG'(\bbQ_p)$. Finally if $v$ is archimedean, then $v$ is complex and there is nothing to check. \par 

Let $\pi_A$ be the Frobenius endomorphism of the virtual abelian variety $A$. For each prime $\ell\ne p$, the $\bbB[\pi_A]$-modules $H_\ell:=H_1(A,\bbQ_\ell)$ and $V_\ell:=V\otimes_\bbQ\bbQ_\ell$ are isomorphic where $\pi_A$ acts on $V_\ell$ by $\ga_\ell^{-1}\in\bbG(\bbQ_\ell)$.\par
Choose a $c$-polarization $\lambda$ of $(A,i)$. Let $I:=\mathrm{Aut}(A,i,\la)$. By the same reasoning as in \cite{Ko-points}, last paragraph on page 423,  $I$ is an inner form of $I_0$ so that we have a canonical $\Gamma$-equivariant isomorphism $Z(\hat{I})=Z(\hat{I}_0)$.\par  

Next we would like to adjust the polarization $\la$ by a class in $H^1(\bbQ,I)$ so that the resulting $c$-polarized virtual $B$-abelian variety gives rise to the generalized Kottwitz triple $(\ga_0';\ga,\delta)$. This is done by constructing local classes in $H^1(\bbQ_v,I)$ at each place $v$ that ``measure the discrepancy between $(A,i,\la)$ and $(\ga,\delta)$ at $v$" and show that this collection of local classes comes from a global class.\par 
First consider a prime $\ell\ne p$. The $c$-polarization $\la$ induces a $*$-alternating form $(\cdot,\cdot)_\la$  on the $B[\pi_A]$-module $H_\ell$, where we extend the involution $*$ to $B[\pi_A]$ by requiring $\pi_A^*=c\pi_A^{-1}$. By \cite[Lemma 10.7]{Ko-points} we know that $I_{\bbQ_\ell}$ is the automorphism group of the skew-Hermitian $B[\pi_A]$-module $H_\ell$, where automorphisms are only required to preserve the form $(\cdot,\cdot)_\la$ up to a scalar in $\bbQ_\ell^\times$. Fix an isomorphism of $B[\pi_A]$-modules $\varphi_\ell:H_\ell\xrightarrow{\sim}V_\ell$ where $\pi_A$ acts on $V_\ell$ by $\gamma_\ell^{-1}$. It takes the $*$-Hermitian pairing $(\cdot,\cdot)$ on $V$ to a $*_{\lambda}$-Hermitian pairing $\varphi_\ell^*(\cdot,\cdot)$ on $H_\ell$. There exists $x_\ell\in \mathrm{Aut}_{B[\pi_A]}(H_\ell\otimes\bar{\bbQ}_\ell)$ such that $x_\ell^*\varphi_\ell^*(\cdot,\cdot)=(\cdot,\cdot)_\la$. For each $\theta\in\mathrm{Gal}(\overline{\bbQ}_\ell/\bbQ_\ell)$, let $a_\ell(\theta):=x_\ell\theta(x_\ell)^{-1}$. Then $a_\ell$ defines a class in $H^1(\bbQ_\ell,I)$ and its image under the map
\[\mathrm{Inv}_\ell: H^1(\bbQ_\ell,I)\to X^*(Z(\hat{I}_0)^{\Gamma_\ell})\]
equals 
\[\mathrm{Inv}_\ell(a_\ell)=\alpha_\ell(\ga_0';A,\la,i)-\alpha_\ell(\ga_0';\ga,\delta).\]
Here $\alpha_\ell(\gamma_0';A,\lambda,i)$ is the invariant defined in \cite[Page 424]{Ko-points} which measures the difference between the skew-Hermitian $\bbB[\pi_A]$-modules $H_\ell$ and $V_\ell$ where $\pi_A$ acts on $V_\ell$ by $(\gamma_0')^{-1}$; while $\alpha_\ell(\ga_0';\ga,\delta)$ is the invariant defined in \cite[\S2]{Ko-lambda} that measures the difference between the conjugacy classes $\ga_\ell$ and $\ga_0'$ inside the same stable conjugacy class of $\bbG(\bbQ_\ell)$.\par 

Next consider the prime $p$. The covariant rational Dieudonn\'e module of $A$ defines a skew-Hermitian isocrystal $(H_r,\Phi)$ over $\bbQ_{p^r}$ where $\Phi$ is the inverse of the Verschiebung operator. By \cite[Lemma 10.8]{Ko-points}, $I_{\bbQ_p}$ is isomorphic to the automorphism group of $(H,\Phi)$ and hence is an inner form of products of Weil restrictions of general linear groups by our assumption that $p$ splits in the imaginary quadratic extension $\bbK$. Hence $H^1(\bbQ_p,I)=1$ and we let $a_p$ be the trivial class.\par

The following observation of Kottwitz is the key to finding the desired global cohomology class $H^1(\bbQ,I)$.
\begin{lem}\label{lem:Kottwitz-simple}
We have $Z(\hat{I}_0)^\Gamma=Z(\hat{\bbG})^\Gamma$.
\end{lem}
\begin{proof}
This is contained in the proof of \cite[Lemma 2]{Ko-simple}.
\end{proof}

\begin{cor}
There exists a cohomology class $a\in\ker(H^1(\bbQ,I)\to H^1(\bbR,I))$ such that for all finite primes $v$ of $\bbQ$, its image in $H^1(\bbQ_v,I)$ equals to the class $a_v$ defined as above.
\end{cor}
\begin{proof}
By Lemma~\ref{lem:Kottwitz-simple}, we have $Z(\hat{I}_0)^{\Gamma}=Z(\hat{\bbG})^\Gamma$. Moreover, since $a_p=1$, to show the existence of the class $a$, it suffices to show that the sum over all finite primes $\ell\ne p$ of the restrictions of $\mathrm{Inv}(a_\ell)$ to $X^*(Z(\hat{\bbG})^{\Gamma})$ is trivial. As in \cite{Ko-points}, we can choose a maximal CM subfield $N$ of $\bbB$ that also embeds into $\mathrm{End}_{\bbB}(A)$. Under a suitable isomorphism $\bbB\otimes_\bbF N\cong M_n(N)$, the $c$-polarized virtual $B$-abelian variety $(A,i,\la)$ over $\bbF_{p^r}$ is decomposed as $A=A_0^n$ where $A_0$ is a $c$-polarised abelian variety with CM by $N$. As in \cite[Lemma 13.2]{Ko-points}, by lifting CM abelian varieties to characteristic 0 we can prove the vanishing of $\sum_{\ell\ne p}\alpha_\ell(\ga_0',A_0)|_{Z(\hat{\bbG})^\Gamma}$, which implies the vanishing of $\sum_{\ell\ne p}\alpha_\ell(\ga_0';A,\la,i)|_{Z(\hat{\bbG})^\Gamma}$. On the other hand, by definition, $\alpha_\ell(\ga_0';\ga,\delta)$ is trivial on $Z(\hat{\bbG})$ since by construction (cf. \cite[\S2]{Ko-lambda}) it belongs to $\ker(H^1(\bbQ_\ell,I_0)\to H^1(\bbQ_\ell,\bbG'))$. Thus we get the desired vanishing statement. 
\end{proof}

We represent the cohomology class $a\in\ker(H^1(\bbQ,I)\to H^1(\bbR,I))$ by a cocycle of the form $x^{-1}\theta(x)$ for all $\theta\in\mathrm{Gal}(\overline{\bbQ}/\bbQ)$, where $x\in\mathrm{End}_\bbB(A)\otimes\overline{\bbQ}$. From the proof of \cite[Lemma 17.1]{Ko-points} we see that some scalar multiple $\la''$ of $\la':=\hat{x}\la x$ defines a $c$-polarization of the virtual $\bbB$-abelian variety $(A,i)$ and the generalized Kottwitz triple associated to $(A,i,\la'')$ is $(\ga_0';\ga,\delta)$.

\subsection{Finishing the proof of the main local theorem}
For each generalized Kottwitz triple $(\ga_0';\ga,\delta)$, choose a polarized virtual $\cO_\bbB$-abelian variety $(A,\la,i)$ that maps to it as in \S\ref{sec:ab-var-associated-to-generalized-triple}. Let $I=I_{(A,\la,i)}$ be the automorphism group of $(A,\la,i)$.
Define the number $c(\ga_0';\ga,\delta)$ to be the product of $\mathrm{vol}(I(\bbQ)\backslash I(\bbA_f))$ and the cardinality of the finite set
\[\ker[\ker^1(\bbQ,I_0)\to\ker^1(\bbQ,\bbG')]\]
By the discussion in \cite[\S3]{Ko-lambda}, the (isomorphism class of the) group $I$ is independent of the choice of $(A,\la,i)$ and thus $c(\ga_0';\ga,\delta)$ is well-defined.

\iffalse
{\cg Tom: here again I think $\beta$ should be included in the notation, and we should remind the reader this equality concerns the Shimura variety $Sh(\bbG_\beta, X)$ constructed above.} {\cc (Jingren: I have added $\beta$ in the notation in the following)} \fi

\begin{thm}\label{thm:Lefschetz-trace-formula}
The following equality holds
\[\mathrm{Tr}(\tau\times hf^p| H_{\beta,\xi}^*)=\sum_{(\ga_0';\ga,\delta)}c(\ga_0';\ga,\delta)\,O_\ga(f^p)\,TO_{\delta\sigma}(\phi_{\tau,h})\,\mathrm{tr}\,\xi(\ga_0')\]
where $(\ga_0';\ga,\delta)$ ranges over equivalence classes of generalized Kottwitz triples such that the image of $\ga_0'$ in $\bbG'(\bbQ_p)$ is transferred from $\bbG_\beta(\bbQ_{p^r})$.
\end{thm}
\begin{proof}
 For each 
$(A,u,\la)$ we associate a generalized Kottwitz triple $(\ga_0',\ga,\delta)$ and the number of $(A,u,\la)$ associated to a given triple $(\ga_0',\ga,\delta)$ is equal to $\ker[\ker^1(\bbQ,I_0)\to\ker^1(\bbQ,\bbG')]$. Then the identity follows from Corollary~\ref{cor:trace-coarse-expansion}.
\end{proof}

Let $f_{\tau,h}^*\in C_c^\infty(\bbG'(\bbQ_p))$ be the stable base change transfer of $\phi_{\tau,h}\in C_c^\infty(\bbG_\beta(\bbQ_{p^r}))$. Then for any conjugacy class $\ga_p\in\bbG'(\bbQ_p)\cong\mathrm{GL}_n(F)\times\bbQ_p^\times$ which is the norm of $\delta\in\bbG_\beta(\bbQ_{p^r})$, we have
\[SO_{\ga_p}(f_{\tau,h}^*)=O_{\ga_p}(f_{\tau,h}^*)=e(\delta)TO_{\delta\sigma}(\phi_{\tau,h})\]
where $e(\delta)=e(\bbG_{\beta,\delta\sigma})\in\{\pm1\}$ is the Kottwitz sign; see \S\ref{sec:kottwitz-sign}.

\begin{cor}\label{cor:trace-G'}
We have the following equality
\[\mathrm{Tr}(\tau\times hf^p| H_{\beta,\xi}^*)=\sum_{\pi'}m(\pi')\mathrm{tr}\,\pi'(f_{\tau,h}^*f^pf^{\bbG'}_{\xi,\infty})\]
where the sum runs over automorphic representations $\pi'$ of $\bbG'$ and $m(\pi')$ denotes the automorphic multiplicity.
\end{cor}
\begin{proof}
This follows from Theorem~\ref{thm:Lefschetz-trace-formula} by the pseudo-stabilization process in \S\ref{sec:pseudo-stabilization}.
\end{proof}

Combining Corollary~\ref{cor:trace-G} and Corollary~\ref{cor:trace-G'}, we will be able to deduce our main local theorems. We claim that for any irreducible generic representation $\pi_p'$ of $\bbG'(\bbQ_p)$ which is not transferred from $\bbG_\beta(\bbQ_p)$, $\mathrm{tr}\,\pi_p'(f_{\tau,h}^*)=0$. Since the $L^2$-Bernstein variety for $\bbG_\beta(\bbQ_p)$ is a union of connected components of the $L^2$-Bernstein variety for $\bbG'(\bbQ_p)$, the property of not being a transfer from $\bbG_\beta(\bbQ_p)$ is Zariski closed. Thus by \cite[Proposition 3.1]{Shin-AJM}, to prove the claim we may assume that $\pi'_p$ is the local component of an automorphic representation $\pi'$ of $\bbG'$ such that $\mathrm{tr}\,\pi_\infty'(f_{\xi,\infty}^{\bbG'})\ne0$. Choose $f^p\in C_c^\infty(\bbG'(\bbA_f^p))$ such
that its trace on $(\pi')^{p,\infty}$ is nonzero and its trace on all other irreducible admissible representations of $\bbG'(\bbA_f^p)\cong\bbG_\beta(\bbA_f^p)$ vanishes. Then by Corollary~\ref{cor:trace-G'}, $\mathrm{Tr}(\tau\times hf^p|H^*_\xi)$ is a nonzero multiple of $\mathrm{tr}\,\pi_p'(f_{\tau,h}^*)$. If this is nonzero, then by Corollary~\ref{cor:trace-G} (with the same choice of $f^p$), there is an automorphic representation $\pi$ of $\bbG_\beta$ such that $\pi^{p,\infty}\cong(\pi')^{p,\infty}$. But then by Theorem~\ref{thm:global-base-change} and strong multiplicity one, $\pi$ and $\pi'$ have the same transfer to an automorphic representation of $\mathrm{GL}_{n,\bbF}$ and hence $\pi_p'$ is the Jacquet-Langlands transfer of $\pi_p$, contradicting our assumption at the beginning. This proves our claim and then Theorem~\ref{thm:vanishing-property} follows by Lemma~\ref{lem:vanishing-criterion}.\par 
Next we consider an irreducible generic representation $\pi_p$ of $\bbG_\beta(\bbQ_p)$ and let $\pi_p'$ be its Jacquet-Langlands transfer to $\bbG'(\bbQ_p)$. We claim that
\[\tr(z_{\tau,-\mu}*h|\pi_p)=\tr(f_{\tau,h}^*|\pi_p').\]
We may assume that $\pi_p$ is the local component of an automorphic representation $\pi$ of $\bbG_\beta(\bbA)$ such that $\tr\,\pi_\infty(f_{\xi,\infty}^{\bbG})=1$. By Theorem~\ref{thm:global-base-change}, \cite[Theorem 5.1]{Badu} and \cite[Theorem VI.2.9]{HT01}, $\pi_p'$ is the local component of an automorphic representation $\pi'$ of $\bbG'$ such that $\tr\,\pi'_\infty(f_{\xi,\infty}^{\bbG'})=1$ and $\pi_f^p\cong(\pi'_f)^p$. We choose $f^p$ that singles out $\pi_f^p$ as before, then the claim follows from Corollary~\ref{cor:trace-G} and Corollary~\ref{cor:trace-G'}. Combined with the previous claim we proved (that $\mathrm{tr}\,\pi_p'(f_{\tau,h}^*)=0$ if $\pi_p'$ is not transferred from $\bbG_\beta(\bbQ_p)$), we see that $f_{\tau,h}^*$ is a Jacquet-Langlands transfer of $z_{\tau,-\mu}*h$. For any semisimple $\ga\in\bbG_\beta(\bbQ_p)=G(F)$ let $\ga^*\in G^*(F)$ be a matching element. If $\ga=\cN_r\delta$ and hence $\ga^*=\cN_r^*\delta$ for some $\delta\in\bbG_\beta(\bbQ_{p^r})=G(F_r)$, we have 
\[e(G_\ga)O_\ga(z_{\tau,\mu}*h)=O_{\ga^*}(f^*_{\tau,h})=e(G_{\delta\sigma})TO_{\delta\sigma}(\phi_{\tau,h}).\]
If $\ga$ is not conjugate to a norm from $G(F_r)$, then $\ga^*$ is also not a norm from $G(F_r)$ and hence we have
\[e(G_\ga)O_\ga(z_{\tau,\mu}*h)=O_{\ga^*}(f^*_{\tau,h})=0.\]
This shows that $z_{\tau,\mu}*h$ is a base change of $\phi_{\tau,h}$ in the sense of Definition \ref{st_bc_defn}, and hence this finishes the proof of Theorem~\ref{thm:main-local}.
\bibliographystyle{alpha}
\bibliography{shimura}

\begin{thebibliography}{Mum08}

\bibitem[AC89]{AC}
James Arthur and Laurent Clozel.
\newblock {\em Simple algebras, base change, and the advanced theory of the
  trace formula}, volume 120 of {\em Annals of Mathematics Studies}.
\newblock Princeton University Press, Princeton, NJ, 1989.

\bibitem[Bad03]{Bad03}
Alexandru~Ioan Badulescu.
\newblock Un r\'{e}sultat de transfert et un r\'{e}sultat
  d'int\'{e}grabilit\'{e} locale des caract\`eres en caract\'{e}ristique non
  nulle.
\newblock {\em J. Reine Angew. Math.}, 565:101--124, 2003.

\bibitem[Bad08]{Badu}
Alexandru~Ioan Badulescu.
\newblock Global {J}acquet-{L}anglands correspondence, multiplicity one and
  classification of automorphic representations.
\newblock {\em Invent. Math.}, 172(2):383--438, 2008.
\newblock With an appendix by Neven Grbac.

\bibitem[Bad18]{Bad18}
A.~I. Badulescu.
\newblock The trace formula and the proof of the global {J}acquet-{L}anglands
  correspondence.
\newblock In {\em Relative aspects in representation theory, {L}anglands
  functoriality and automorphic forms}, volume 2221 of {\em Lecture Notes in
  Math.}, pages 197--250. Springer, Cham, 2018.

\bibitem[Bor76]{Bo76}
Armand Borel.
\newblock Admissible representations of a semi-simple group over a local field
  with vectors fixed under an {I}wahori subgroup.
\newblock {\em Invent. Math.}, 35:233--259, 1976.

\bibitem[CL11]{CL2}
L.~Clozel and J.-P. Labesse.
\newblock Orbital integrals and distributions.
\newblock In {\em On certain {$L$}-functions}, volume~13 of {\em Clay Math.
  Proc.}, pages 107--115. Amer. Math. Soc., Providence, RI, 2011.

\bibitem[Clo90]{Clo90}
Laurent Clozel.
\newblock The fundamental lemma for stable base change.
\newblock {\em Duke Math. J.}, 61(1):255--302, 1990.

\bibitem[Coh18]{Coh18}
Jonathan Cohen.
\newblock Transfer of representations and orbital integrals for inner forms of
  ${\rm gl}_n$.
\newblock {\em Canad. J. Math.}, 70(3):595--627, 2018.

\bibitem[DKV84]{DKV}
P.~Deligne, D.~Kazhdan, and M.-F. Vign\'eras.
\newblock Repr\'esentations des alg\`ebres centrales simples {$p$}-adiques.
\newblock In {\em Representations of reductive groups over a local field},
  Travaux en Cours, pages 33--117. Hermann, Paris, 1984.

\bibitem[Far04]{Far04}
Laurent Fargues.
\newblock Cohomologie des espaces de modules de groupes {$p$}-divisibles et
  correspondances de {L}anglands locales.
\newblock Number 291, pages 1--199. Soc.\,Math.\,France, 2004.
\newblock Vari\'{e}t\'{e}s de Shimura, espaces de Rapoport-Zink et
  correspondances de Langlands locales.

\bibitem[Fen20]{Feng}
Tony Feng.
\newblock Nearby cycles of parahoric shtukas, and a fundamental lemma for base
  change.
\newblock {\em Selecta Mathematica}, 26(2):21, 2020.

\bibitem[Fla79]{Fl79}
D.~Flath.
\newblock Decomposition of representations into tensor products.
\newblock In {\em Automorphic forms, representations and {$L$}-functions
  ({P}roc. {S}ympos. {P}ure {M}ath., {O}regon {S}tate {U}niv., {C}orvallis,
  {O}re., 1977), {P}art 1}, volume XXXIII of {\em Proc. Sympos. Pure Math.},
  pages 179--183. Amer. Math. Soc., Providence, RI, 1979.

\bibitem[GR10]{GR}
Benedict~H. Gross and Mark Reeder.
\newblock Arithmetic invariants of discrete {L}anglands parameters.
\newblock {\em Duke Math. J.}, 154(3):431--508, 2010.

\bibitem[Hai05]{Ha05}
Thomas~J. Haines.
\newblock Introduction to {S}himura varieties with bad reduction of parahoric
  type.
\newblock In {\em Harmonic analysis, the trace formula, and {S}himura
  varieties}, volume~4 of {\em Clay Math. Proc.}, pages 583--642. Amer. Math.
  Soc., Providence, RI, 2005.

\bibitem[Hai09]{Hai09}
Thomas~J. Haines.
\newblock The base change fundamental lemma for central elements in parahoric
  {H}ecke algebras.
\newblock {\em Duke Math. J.}, 149(3):569--643, 2009.

\bibitem[Hai14]{Ha14}
Thomas~J. Haines.
\newblock The stable {B}ernstein center and test functions for {S}himura
  varieties.
\newblock In {\em Automorphic forms and {G}alois representations. {V}ol. 2},
  volume 415 of {\em London Math. Soc. Lecture Note Ser.}, pages 118--186.
  Cambridge Univ. Press, Cambridge, 2014.

\bibitem[HN02]{HN02}
T.~Haines and B.~C. Ng\^{o}.
\newblock Nearby cycles for local models of some {S}himura varieties.
\newblock {\em Compositio Math.}, 133(2):117--150, 2002.

\bibitem[HR20]{HaRi20}
Thomas~J. Haines and Timo Richarz.
\newblock The test function conjecture for local models of {W}eil-restricted
  groups.
\newblock {\em Compos. Math.}, 156(7):1348--1404, 2020.

\bibitem[HR21]{HaRi21}
Thomas~J. Haines and Timo Richarz.
\newblock The test function conjecture for parahoric local models.
\newblock {\em J. Amer. Math. Soc.}, 34(1):135--218, 2021.

\bibitem[HT01]{HT01}
Michael Harris and Richard Taylor.
\newblock {\em The geometry and cohomology of some simple {S}himura varieties},
  volume 151 of {\em Annals of Mathematics Studies}.
\newblock Princeton University Press, Princeton, NJ, 2001.
\newblock With an appendix by Vladimir G. Berkovich.

\bibitem[HZZ]{HZZ}
Thomas Haines, Rong Zhou, and Yihang Zhu.
\newblock The stable trace formula for shimura varieties with parahoric level
  structure.

\bibitem[Kot82]{Ko-conj}
Robert~E. Kottwitz.
\newblock Rational conjugacy classes in reductive groups.
\newblock {\em Duke Math. J.}, 49(4):785--806, 1982.

\bibitem[Kot83]{Ko-sign}
Robert~E. Kottwitz.
\newblock Sign changes in harmonic analysis on reductive groups.
\newblock {\em Trans. Amer. Math. Soc.}, 278(1):289--297, 1983.

\bibitem[Kot84]{Ko-twisted}
Robert~E. Kottwitz.
\newblock Shimura varieties and twisted orbital integrals.
\newblock {\em Math. Ann.}, 269(3):287--300, 1984.

\bibitem[Kot86a]{Ko-BC}
Robert~E. Kottwitz.
\newblock Base change for unit elements of {H}ecke algebras.
\newblock {\em Compositio Math.}, 60(2):237--250, 1986.

\bibitem[Kot86b]{Ko-ellsing}
Robert~E. Kottwitz.
\newblock Stable trace formula: elliptic singular terms.
\newblock {\em Math. Ann.}, 275:365--399, 1986.

\bibitem[Kot88]{Ko-Tam}
Robert~E. Kottwitz.
\newblock Tamagawa numbers.
\newblock {\em Ann. Math.}, 127(3):629--646, 1988.

\bibitem[Kot90]{Ko-lambda}
Robert~E. Kottwitz.
\newblock Shimura varieties and {$\lambda$}-adic representations.
\newblock In {\em Automorphic forms, {S}himura varieties, and {$L$}-functions,
  {V}ol. {I} ({A}nn {A}rbor, {MI}, 1988)}, volume~10 of {\em Perspect. Math.},
  pages 161--209. Academic Press, Boston, MA, 1990.

\bibitem[Kot92a]{Ko-simple}
Robert~E. Kottwitz.
\newblock On the {$\lambda$}-adic representations associated to some simple
  {S}himura varieties.
\newblock {\em Invent. Math.}, 108(3):653--665, 1992.

\bibitem[Kot92b]{Ko-points}
Robert~E. Kottwitz.
\newblock Points on some {S}himura varieties over finite fields.
\newblock {\em J. Amer. Math. Soc.}, 5(2):373--444, 1992.

\bibitem[KPS22]{KPS}
Mark Kisin, Keerthi~Madapusi Pera, and Sug~Woo Shin.
\newblock {Honda–Tate theory for Shimura varieties}.
\newblock {\em Duke Mathematical Journal}, 171(7):1559 -- 1614, 2022.

\bibitem[KR00]{KR00}
Robert~E. Kottwitz and Jonathan~D. Rogawski.
\newblock The distributions in the invariant trace formula are supported on
  characters.
\newblock {\em Canad. J. Math.}, 52(4):804--814, 2000.

\bibitem[Lab99]{La99}
Jean-Pierre Labesse.
\newblock Cohomologie, stabilisation et changement de base.
\newblock {\em Ast\'{e}risque}, (257):vi+161, 1999.
\newblock Appendix A by Laurent Clozel and Labesse, and Appendix B by Lawrence
  Breen.

\bibitem[Lau96]{Lau96}
G\'erard Laumon.
\newblock {\em Cohomology of {D}rinfeld modular varieties. {P}art {I}},
  volume~41 of {\em Cambridge Studies in Advanced Mathematics}.
\newblock Cambridge University Press, Cambridge, 1996.
\newblock Geometry, counting of points and local harmonic analysis.

\bibitem[Mum08]{Mum-AV}
David Mumford.
\newblock {\em Abelian varieties}, volume~5 of {\em Tata Institute of
  Fundamental Research Studies in Mathematics}.
\newblock Tata Institute of Fundamental Research, Bombay; by Hindustan Book
  Agency, New Delhi, 2008.
\newblock With appendices by C. P. Ramanujam and Yuri Manin, Corrected reprint
  of the second (1974) edition.

\bibitem[PR03]{PR03}
G.~Pappas and M.~Rapoport.
\newblock Local models in the ramified case. {I}. {T}he {EL}-case.
\newblock {\em J. Algebraic Geom.}, 12(1):107--145, 2003.

\bibitem[Rap90]{Rap90}
M.~Rapoport.
\newblock On the bad reduction of {S}himura varieties.
\newblock In {\em Automorphic forms, {S}himura varieties, and {$L$}-functions,
  {V}ol.\ {II} ({A}nn {A}rbor, {MI}, 1988)}, volume~11 of {\em Perspect.
  Math.}, pages 253--321. Academic Press, Boston, MA, 1990.

\bibitem[Rap05]{Rap05}
Michael Rapoport.
\newblock A guide to the reduction modulo {$p$} of {S}himura varieties.
\newblock Number 298, pages 271--318. 2005.
\newblock Automorphic forms. I.

\bibitem[Rog88]{Ro88}
J.~D. Rogawski.
\newblock Trace {P}aley-{W}iener theorem in the twisted case.
\newblock {\em Trans. Amer. Math. Soc.}, 309(1):215--229, 1988.

\bibitem[Rog90]{Ro90}
Jonathan~D. Rogawski.
\newblock {\em Automorphic representations of unitary groups in three
  variables}, volume 123 of {\em Annals of Mathematics Studies}.
\newblock Princeton University Press, Princeton, NJ, 1990.

\bibitem[Rog81]{Ro81}
Jonathan~D. Rogawski.
\newblock An application of the building to orbital integrals.
\newblock {\em Compositio Math.}, 42(3):417--423, 1980/81.

\bibitem[RZ96]{RZ}
M.~Rapoport and Th. Zink.
\newblock {\em Period spaces for {$p$}-divisible groups}, volume 141 of {\em
  Annals of Mathematics Studies}.
\newblock Princeton University Press, Princeton, NJ, 1996.

\bibitem[Sch13a]{Sch-LK-simple}
Peter Scholze.
\newblock The {L}anglands-{K}ottwitz approach for some simple {S}himura
  varieties.
\newblock {\em Invent. Math.}, 192(3):627--661, 2013.

\bibitem[Sch13b]{Sch-deformation}
Peter Scholze.
\newblock The {L}anglands-{K}ottwitz method and deformation spaces of
  {$p$}-divisible groups.
\newblock {\em J. Amer. Math. Soc.}, 26(1):227--259, 2013.

\bibitem[She18]{shen-1}
Xu~Shen.
\newblock On the {$\ell$}-adic cohomology of some {$p$}-adically uniformized
  {S}himura varieties.
\newblock {\em J. Inst. Math. Jussieu}, 17(5):1197--1226, 2018.

\bibitem[Shi12]{Shin-AJM}
Sug~Woo Shin.
\newblock On the cohomology of {R}apoport-{Z}ink spaces of {EL}-type.
\newblock {\em Amer. J. Math.}, 134(2):407--452, 2012.

\bibitem[SS13]{Sch-Shin}
Peter Scholze and Sug~Woo Shin.
\newblock On the cohomology of compact unitary group {S}himura varieties at
  ramified split places.
\newblock {\em J. Amer. Math. Soc.}, 26(1):261--294, 2013.

\bibitem[Tad90]{Tad90}
Marko Tadic.
\newblock Induced representations of ${\rm gl}_n(a)$ for $p$-adic division
  algebras $a$.
\newblock {\em J.\,reine angew.\,Math.}, 405:48--77, 1990.

\bibitem[Tad15]{Tad15}
Marko Tadi\'c.
\newblock Remark on representation theory of general linear groups over a
  non-archimedean local division algebra.
\newblock {\em Rad Hrvat. Akad. Znan. Umjet. Mat. Znan.}, 19(523):27--53, 2015.

\bibitem[Var07]{Var07}
Yakov Varshavsky.
\newblock Lefschetz-{V}erdier trace formula and a generalization of a theorem
  of {F}ujiwara.
\newblock {\em Geom. Funct. Anal.}, 17(1):271--319, 2007.

\bibitem[Zel80]{Zel80}
Andrey~V. Zelevinsky.
\newblock Induced representations of reductive p-adic groups.ii. on irreducible
  representations of ${\rm gl}_n$.
\newblock {\em Ann.\,Sci.\,\'{E}cole Norm.\,Sup.}, 13(4):165--210, 1980.

\end{thebibliography}

\vspace{.25in}
\end{document}